\documentclass[a4paper,11pt, english]{article}
\usepackage[english]{babel}
\usepackage{authblk}
\usepackage{amsmath}
\usepackage{amsthm}
\usepackage{eucal}
\usepackage{anysize}
\usepackage{mathtools}
\mathtoolsset{showonlyrefs}
\usepackage{mathrsfs}
\usepackage{amssymb}
\usepackage{xcolor}
\usepackage{xfrac}
\usepackage{esint}
\usepackage{graphicx}
\usepackage{float}
\usepackage{array}
\usepackage{enumitem}
\usepackage{tikz-cd}
\usepackage{centernot}
\usepackage{stmaryrd}
\usepackage{comment}
\usepackage{bbm}
\usepackage{overpic}
\usepackage{mathabx}
\usepackage{hyperref}

\usepackage{graphicx,color}
\usepackage{tikz,pgfplots,pgf}
\usepackage[font=small]{caption}

\UseRawInputEncoding

\numberwithin{equation}{section}
\counterwithin{figure}{section}

\newtheorem{theorem}{Theorem}[section]

\newtheorem{lemma}[theorem]{Lemma}
\newtheorem{proposition}[theorem]{Proposition}

\theoremstyle{definition}
\newtheorem{definition}[theorem]{Definition}

\newtheorem{remark}[theorem]{Remark}

\DeclareMathOperator{\Span}{span}

\newcommand{\field}[1] {\mathbb{#1}}
\newcommand{\N}{\field{N}}
\newcommand{\Z}{\field{Z}}

\newcommand{\R}{\field{R}}

\newcommand{\T}{\field{T}}

\def\a{\alpha}
\def\b{\beta}
\def\e{\varepsilon}

\def\d{\delta}
\def\g{\gamma}
\def\G{\Gamma}
\def\l{\lambda}

\def\L{\Lambda}
\def\o{\omega}
\def\O{\Omega}
\def\p{\partial}
\def\r{\rho}

\def\s{\sigma}
\def\t{\theta}

\def\da{\downarrow}

\newcommand{\mc}{\mathcal}
\newcommand{\mf}{\mathfrak}

 \renewcommand{\o}{\omega}

\renewcommand{\L}{\Lambda}

\makeatletter

\author[1]{Eduardo Mu\~{n}oz-Hern\'{a}ndez \thanks{\texttt{eduardmu@ucm.es}}}
\author[2]{Juan Carlos Sampedro \thanks{\texttt{juancarlos.sampedro@upm.es}}}
\author[2]{Andrea Tellini \thanks{\texttt{andrea.tellini@upm.es}}}
\affil[1]{\small Universidad Complutense de Madrid\\ Instituto de Matem\'atica Interdisciplinar (IMI)\\Departamento de An\'alisis Matem\'atico y Matem\'atica	 Aplicada\\ Plaza de las Ciencias 3\\ 28040 Madrid, Spain}
\affil[2]{\small Universidad Polit\'ecnica de Madrid\\ E.T.S. de Ingenier\'ia y Dise\~{n}o Industrial\\ Departamento de Matem\'atica Aplicada a la Ingenier\'ia Industrial\\ Ronda de Valencia 3\\ 28012 Madrid, Spain}

\title{\textbf{Bifurcation Theory for a Class of Periodic Superlinear Problems}}

\begin{document}
	
	\maketitle

    \vspace{-0.3cm}
    
   {\centering \emph{Dedicated to Juli\'an L\'opez-G\'omez, our mentor, with deep esteem and gratitude\\ on the occasion of his 65th birthday. }\par}
	\vspace{0.2cm}
    
	\begin{abstract}
		We analyze, mainly using bifurcation methods, an elliptic superlinear problem in one-dimension with periodic boundary conditions. One of the main novelties is that we follow for the first time a bifurcation approach, relying on a Lyapunov--Schmidt reduction and some recent global bifurcation results, that allows us to study the local and global structure of non-trivial solutions at bifurcation points where the linearized operator has a two-dimensional kernel. Indeed, at such points the classical tools in bifurcation theory, like the Crandall--Rabinowitz theorem or some generalizations of it, cannot be applied because the multiplicity of the eigenvalues is not odd, and a new approach is required.
		
		We apply this analysis to specific examples, obtaining new existence and multiplicity results for the considered periodic problems, going beyond the information variational and fixed point methods like Poincar\'e--Birkhoff theorem can provide.
	\end{abstract}

    \textbf{Keywords:} Bifurcation theory, Periodic problems, Nodal periodic solutions, Components of periodic solutions, Eigenvalues with even multiplicity

    \textbf{MSC 2020:} 47J15, 34C23, 70H12, 34C25

	\section{Introduction}
	\label{section:1}
	In this work, we analyze the following paradigmatic periodic boundary value problem
	\begin{equation}
		\label{eq:1.1}
		\left\{
		\begin{array}{l}
			-u''=\l u + a(t) u^{3}, \quad \;\, t\in (0,T),  \\[3pt]
			u(0)=u(T), \quad u'(0)=u'(T),
		\end{array}
		\right.
	\end{equation}
	where $\l\in\R$ will be regarded as a bifurcation parameter, and  $T>0$ is a fixed period. Regarding the weight $a(t)$, we assume
	\begin{equation}
		a\in L^{\infty}([0,T])\qquad\hbox{and}\qquad a\gneq 0.
		\tag{HLoc}
		\label{ass:aloc}
	\end{equation}
	In this setting, a solution of \eqref{eq:1.1} will be a function $u:[0,T]\to\R$ in the Sobolev class $H^2([0,T])$, that solves the differential equation almost everywhere together with the periodic boundary conditions. 
	\par 
	Here and in the rest of this work, when referring to essentially bounded functions, by $f_1\gneq f_2$ we mean $f_1(t)\geq f_2(t)$ for a.e. $t\in[0,T]$ and $f_1\not\equiv f_2$. Since the right-hand side of the ordinary differential equation in \eqref{eq:1.1}, $f(t,u):=\l u+a(t)u^3$, satisfies
	$$\frac{f(t,u)}{u}\to +\infty \qquad \text{as $u\to\pm\infty$, for all $t\in(0,T)$ such that $a(t)\neq 0$,}$$
	the problem is referred to as \emph{superlinear} (at infinity). 
	\par
    For some results, additionally to \eqref{ass:aloc}, in order to obtain a priori bounds of solutions of \eqref{eq:1.1} in compact intervals of the parameter $\l$, we will suppose the existence of $0=t_0<t_1<\cdots<t_N=T$ such that, for all $i\in\{0,\ldots,N-1\}$,
    \begin{equation}
		a\in\mc{C}((t_i,t_{i+1})), \qquad \text{and} \qquad \text{ either } \inf_{t\in(t_i,t_{i+1})} a(t)>0 \text{ or } a(t)=0 \text{ for all $t\in(t_i,t_{i+1})$}.
		\tag{HGlob}
		\label{ass:aglob}
	\end{equation} 
	\par
	The study of periodic problems associated to a second order superlinear ordinary differential equation has received a great deal of attention in the past decades. This study goes back to the seminal paper by Nehari \cite[Sec. 8]{Ne}, who studied the existence of $T$-periodic solutions of 
	\begin{equation}
		\label{eq:class}
		-u''=uF(t,u)
	\end{equation}
	with $F$ continuous in both variables, $T$-periodic in $t$, $F(t,u)>0$ for all $(t,u)\in[0,T]\times\R_{>0}$, and $uF(t,u)$ superlinear in $u$. Namely, Theorem 8.1 of \cite{Ne} proves the existence of $T$-periodic solutions with $2k$ zeros for any $k\geq1$ by the use of variational methods coming from the analysis of the Dirichlet boundary value problem. Some years later, Jacobowitz \cite{Ja}, following a suggestion of Moser (see \cite[Sec. 1]{Ja}), extended the results of Nehari via a pioneering use of the Poincar\'e--Birkhoff theorem in this kind of periodic problems, whose Hamiltonian structure allows its application.   The year after, Hartman \cite{Ha}, in a more general setting and relaxing the sign conditions of \cite{Ja},  obtained the existence of $T$-periodic solutions by proving a suitable a priori bound for solutions with a given number of zeros, applying again the Poincar\'e--Birkhoff theorem. Moreover, as pointed out by Moser (see \cite[Sec. 1]{Ha}), under the same assumptions on $F$ but assuming that it depends also on $u'$, the existence of periodic solutions is no longer true. This stresses the fact that the Hamiltonian structure of \eqref{eq:class} plays a fundamental role in the appearance of periodic solutions. Since then, much work has been done in the analysis of periodic superlinear second order equations and systems, see for instance \cite{Bo, FS, QTW, FoGi,Gi}. We also send the interested readers to \cite[Sec. 1]{JEF2} (and the references therein) for a more complete review about the huge amount of techniques and approaches developed to address periodic problems. 
	\par
	However, the above commented variational and topological methods, apart from existence and multiplicity results, do not give precise information about the structure of the solutions set of the problem, which is instead the case of bifurcation theory analyzed in this paper.
	\par
	In the setting of such a theory, solutions are regarded as pairs $(\l,u)$, where $\l\in\R$ is the bifurcation parameter and $u$ solves \eqref{eq:1.1} in the sense explained above. Since $(\l,0)$ is a solution of \eqref{eq:1.1} for all $\l\in\R$, it is referred to as the \emph{trivial solution}. One of the main goals of bifurcation theory is finding specific values of $\lambda$, say $\l_0$, such that there exist non-trivial solutions in a ball $B_r$ centered at $(\l_0,0)$, for all sufficiently small radii $r$, and to study the structure of such non-trivial solutions. If this occurs, $(\l_0,0)$ is referred to as \emph{bifurcation point} from the set of trivial solutions. Observe that solutions of \eqref{eq:1.1} can be seen as the zeros of the operator
	$$\mf{F}(\l,u):=u''+\l u+a(t)u^3, $$
	considered in suitable Banach spaces, where it is of class $\mc{C}^1$ and Fredholm of index 0. The implicit function theorem guarantees that, for $(\l_0,0)$ to be a bifurcation point, necessarily $\p_u\mf{F}(\l_0,0)$ must be non-invertible, or, equivalently, it must have a non-trivial kernel. In this case, the kernel of $\p_u\mf{F}(\l_0,0)$ consists of solutions of the periodic linear eigenvalue problem 
	\begin{equation*}
		\left\{
		\begin{array}{l}
			-u''=\l_0 u,\quad\text{ in $(0,T)$},\\[3pt]
			u(0)=u(T),\quad u'(0)=u'(T),
		\end{array}
		\right.
	\end{equation*}
	which has non-trivial solutions if, and only if,
	\[
	\l_0=\s_k:=\left(\frac{2\pi k}{T}\right)^2, \qquad \text{for some } k\in\N\cup\{0\}.
	\]
	The main issue in this case is that, except for $k=0$, the eigenvalues $\s_k$ have geometric multiplicity 2, which makes the classical bifurcation theory, going back to Krasnosel'skii \cite{Kr} and Crandall and Rabinowitz \cite{CR, Ra1}, as well as its generalizations, collected e.g. in the book by L\'opez-G\'omez \cite{LG01}, inapplicable because the eigenvalues have non-odd multiplicity.
	
	This is an essential difference when comparing with the corresponding one-dimensional eigenvalue problems under Dirichlet, Neumann, Robin, or mixed boundary conditions, whose eigenvalues have geometric multiplicity 1, and where classical bifurcation theory has extensively been used. 
	
	In order to tackle this problem, we follow two approaches. The first one consists in assuming additional symmetry properties for the weight $a(t)$, taken as an even periodic function, and modify the spaces in which the operator $\mf{F}$ is defined, so that the geometric multiplicity of the eigenvalues reduces to 1, and classical bifurcation results are applicable.	We point out that this kind of strategy, simplified by the special logistic structure of the nonlinearity or some hidden symmetries, was the one followed by L\'opez-G\'omez and collaborators in \cite{LOT, JE, JEF2}, which are the only works we are aware of that apply bifurcation theory to the study of periodic problems.
	
	The second approach, which constitutes one of the main novelties of this work, is based on a Lyapunov--Schmidt reduction that, without any additional assumption on $a(t)$, allows us to transform the study of the local bifurcated curves of non-trivial solutions of \eqref{eq:1.1} in a neighborhood of $(\s_k,0)$, $k\in\N$, to the equivalent problem of existence of solutions to certain algebraic equations.
	
	Moreover, once this local information is obtained, we apply some recent global bifurcation results, valid also for the case of eigenvalues with non-odd multiplicities, to establish the possible global behaviors of these components of nontrivial solutions, thus reaching the maximum information that bifurcation theory can provide.

Namely, the main new phenomena that we obtain, related to existence and multiplicity of solutions of \eqref{eq:1.1} are:
\begin{itemize}
    \item When $a(t)$ is even, we show that for each $\lambda < \sigma_k$, $k\geq 1$, problem \eqref{eq:1.1} admits at least two nontrivial even and two odd solutions with exactly $2k$ zeros in the interval $[0, T)$.

    \item When $a(t)$ is not necessarily even, the existence and multiplicity of solutions depend on the qualitative features of the weight $a(t)$. For instance:
    \begin{itemize}
        \item If $a(t)$ satisfies some structural properties which hold true, for example, in the case of the characteristic function of $[0,\pi/4]$ and odd $k$, we prove that, for $\lambda < \sigma_{k}$ and $\lambda \sim \sigma_{k}$, problem \eqref{eq:1.1} admits at least four nontrivial solutions with $2k$ zeros in $[0,T)$.

        \item For other specific profiles of $a(t)$ (see Section~\ref{section:5}), we prove that for $\lambda < \sigma_1$ and $\lambda \sim \sigma_1$, equation \eqref{eq:1.1} admits at least eight nontrivial solutions with exactly two zeros in $[0,T)$. See Figures~\ref{Fig:5.1} and~\ref{Fig:5.2} for illustrations.
    \end{itemize}
    \item In both cases these solutions are not isolated: they belong to connected components of the set of nontrivial solutions of \eqref{eq:1.1} (see Sections~\ref{section:3} and \ref{section:4} for the precise meaning and construction of these components).
\end{itemize}
\noindent
It is worth emphasizing that variational methods or the Poincar\'e-–Birkhoff approach typically yield multiplicity results for $\l\leq\l_*$  or $k\geq k_*$, with $\l_*<\s_k$ and $k_*\geq 1$ not specified. In contrast, our bifurcation approach not only covers a complete range of parameters with $\l_*=\s_k$ and $k_*=1$, but also guarantees that the solutions are organized into global connected components, rather than appearing as isolated critical/fixed points.

	The paper is structured as follows. In Section \ref{section:2}, we give some preliminaries on the autonomous version of \eqref{eq:1.1}, i.e., when the weight is constant $a(t)\equiv a>0$. In particular, we obtain the corresponding bifurcation diagram of periodic solutions. It is interesting to compare that case, where, for all $k\in\N$, a surface $\mc{S}_k$ of $T$-periodic solutions with $2k$ zeros in $[0,T)$ bifurcates from the trivial curve at $(\s_k,0)$, with the non-autonomous case treated later, where the set of solutions bifurcating from $(\s_k,0)$ consists of branches of curve. Moreover, by means of phase plane analysis, we obtain a priori bounds on the solutions of the general case, that will be later used in Section \ref{section:4}.
	
	Section \ref{section:3} obtains local bifurcation results for non-constant weights $a(t)$, first when the classical theory can be applied. Precisely, at the principal eigenvalue $\s_0$, which is always simple, a branch of positive and one of negative solutions are obtained (see Section \ref{sec:3.1}). At $\s_k$, $k\in\N$, assuming that the weight is even, we prove the existence of two branches of even and two of odd solutions (see Section \ref{sec:3.2}). Then, the general case is treated in Section \ref{sec:3.3}, and a particular, though quite general example is presented in Section \ref{sec:3.4}, giving the existence of 4 local branches of solutions of \eqref{eq:1.1} with $2k$ zeros in $[0,T)$ bifurcating from $(\s_k,0)$, with $k$ odd.
	
	Section \ref{section:4} gives the global results about the components obtained locally at $(\s_k,0)$, both in the classical case, with eigenvalues having odd multiplicity, and in the non-classical case. Moreover, some qualitative properties of the solutions are proved, essentially showing that the number of zeros of the solutions is maintained along each of such components.
	
	Then, in Section \ref{section:5}, together with some final remarks, we present additional examples where the general local theory of Section \ref{section:3} is applied, showing that the number of branches bifurcating from $(\s_k,0)$ might be higher (we give examples of cases with 8 bifurcating branches). 
	
	For the reader's convenience, we have collected in Appendices \ref{A1} and \ref{app:B} classical and recent results in local and global bifurcation theory that are used throughout this work. Finally, Appendix \ref{AC} collects some technical computations that are required in Section \ref{section:3}.
	
	Finally, we remark that the same techniques used in this paper can be adapted to the general case $f(t,u)=\l u+a(t)u^p$ with $p\in\N\setminus\{1\}$. Here, we have decided to take $p=3$ to highlight the main abstract ideas underlying the new bifurcation-theoretic approach proposed for periodic problems.
		
	\section{Preliminary study of the autonomous case}
	\label{section:2}
	This section analyzes problem \eqref{eq:1.1} following a dynamical systems approach. In order to do so, we consider the associated planar system to the equation in \eqref{eq:1.1}, i.e., 
	\begin{equation}
		\label{E37}
		\left\{
		\begin{array}{ll}
			u'=v,\\
			v'=-\l u-a(t)u^3,
		\end{array}
		\right.
	\end{equation}
	and the angular component of a solution $(u,v)$ of \eqref{E37} with $(u(0),v(0))=z_0\neq(0,0)$, that it is denoted by $\theta(t;z_0)$. By the Cauchy--Lipschitz theorem, $(u(t),v(t))\neq(0,0)$ for all $t\in[0,T]$ if $z_0\neq0$ and, hence, we can define the \emph{winding number} around the origin of a nontrivial solution of \eqref{E37} in any interval $[\a,\b]\subseteq[0,T]$ as
	\begin{equation}
		\label{eq:rot}
		\mc{W}([\a,\b];z_0):=-\frac{\theta(\b;z_0)-\theta(\a;z_0)}{2\pi}.
	\end{equation}
	Note that a solution of \eqref{eq:1.1} has an angular displacement $\theta(\b;z_0)-\theta(\a;z_0)\leq0$, which is non-positive because the solution travels clockwise around the origin and, thus, by \eqref{eq:rot}, the winding number will be a non-negative quantity. Moreover, given any $T$-periodic nontrivial solution, $(u,v)$, of \eqref{E37},
	\[
	\mc{W}([0,T];z_0)=k\in\mathbb{N}\cup\{0\}.
	\]
	Equivalently, we will say that any ($T$-periodic) solution of \eqref{eq:1.1} has winding number $k\in\mathbb{N}\cup\{0\}$ around the origin if it possesses exactly $2k$ zeros in $[0,T)$. 
	
	\subsection{The autonomous case}
	
	We first focus on periodic solutions that change sign in the autonomous framework $a(t)\equiv a>0$ for all $t\in [0,T]$. Hence, we consider the periodic autonomous problem
	\begin{equation}
		\label{E31}
		\left\{
		\begin{array}{l}
			-u''=\l u + au^{3}, \quad t\in (0,T),  \\[3pt]
			u(0)=u(T), \quad u'(0)=u'(T).
		\end{array}
		\right.
	\end{equation}
	The planar system associated with the differential equation in \eqref{E31} is 
	\begin{equation}
		\label{E32}
		\left\{
		\begin{array}{ll}
			u'=v,\\
			v'=-\l u-au^3,
		\end{array}
		\right.
	\end{equation}
	and the corresponding energy function (or Hamiltonian), which is conserved along its solutions, is
	\[
	\mathcal{E}(u,v)=\frac{1}{2}v^2+\frac{\l}{2}u^2+\frac{a}{4}u^4.
	\]
	The structure of the phase portrait corresponding to \eqref{E32} depends on the sign of $\l$, as sketched in Figure \ref{Fig:2.1}, where positive energy orbits are depicted in red, zero energy equilibria and orbits appear in black, and negative energy equilibria and orbits are drawn in blue. 
	\begin{figure}[h!]
		\centering
		\includegraphics[scale=0.215]{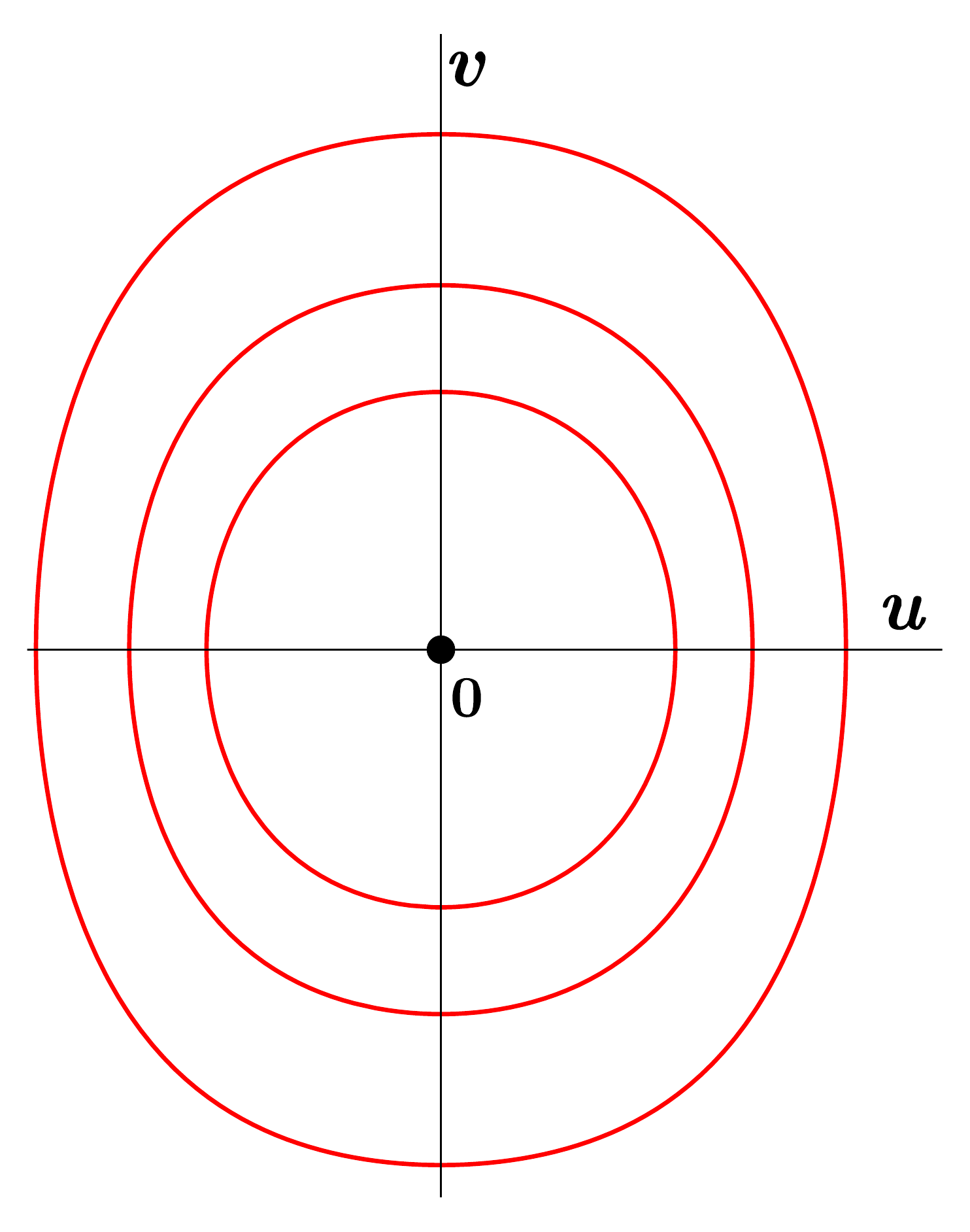}\quad\quad
		\includegraphics[scale=0.248]{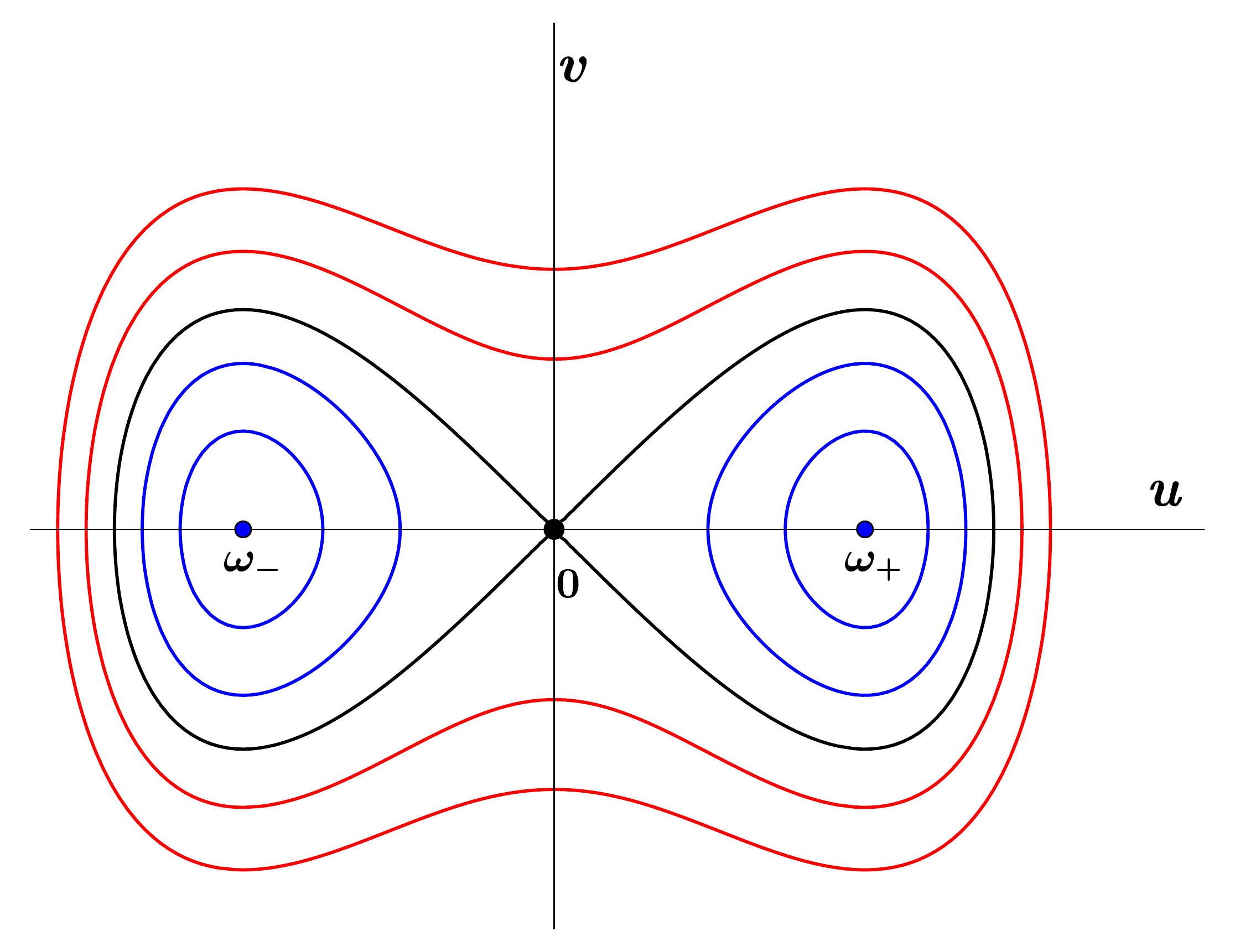}
		\caption{Phase planes of \eqref{E32} for $\l\geq 0$ (left) and $\l<0$ (right).}
		\label{Fig:2.1}
	\end{figure}	
	\noindent Indeed, if $\l\geq0$, the origin is the unique equilibrium of \eqref{E32}, which is a global center surrounded by closed orbits corresponding to periodic solutions of \eqref{E32}. Instead, if $\l<0$, there are two additional equilibria $(\o_\pm,0)$, with 
	\begin{equation}
		\label{eq:equilibria}
		\o_{\pm}=\o_{\pm}(\l):=\pm\left(\frac{-\l}{a}\right)^{\frac{1}{2}},
	\end{equation}
	which are local centers surrounded by closed orbits, corresponding to definite sign (positive and negative) periodic solutions, respectively. In this case, the origin is a saddle point having as stable and unstable manifolds two homoclinic connections (which are zero energy level sets) enclosing the other equilibria and intersecting the $u$-axis at $(\pm u_*,0)$, where 
	\[
	u_*:=\left(\frac{-2\l}{a}\right)^{\frac{1}{2}}.
	\]
	Outside the homoclinic connections all the orbits are closed, and provide us with nodal, i.e. sign-changing, periodic solutions of \eqref{E32}. For notational convenience, we set $u_*=0$ for $\l\geq 0$.
	\par
	Note that the equilibria of \eqref{E32} are constant (periodic) solutions of \eqref{E31}. The non-constant sign-changing periodic solutions of \eqref{E32} correspond to the red orbits surrounding the origin in Figure \ref{Fig:2.1} and are classified according to their winding number. To analyze them, given any $u_+>u_*$, we consider the closed orbit of \eqref{E32} passing through $(u_+,0)$, which satisfies
	\begin{equation}
		\label{E33}
		\mc{E}(u,v)=\frac{1}{2}v^2+\frac{\l}{2}u^2+\frac{a}{4}u^4=\frac{\l}{2}u_+^2+\frac{a}{4}u_+^{4}=:e>0.
	\end{equation}
	Notice that, for any $\l\in\mathbb{R}$, if we consider $u_+$ as a function of $e>0$, i.e., $u_+=u_+(e)$, we have that it is strictly increasing,  
	\[
	\lim_{e\to0^+}u_+(e)=u_*
	\quad\hbox{and}\quad
	\lim_{e\to+\infty}u_+(e)=+\infty.
	\]
	Thus, there exists a one-to-one correspondence between $u_+(e)\in(u_*,+\infty)$ and $e\in(0,+\infty)$. Moreover, the symmetries of \eqref{E32} ensure that the time to cross any quadrant through a periodic orbit is the same. Thus, the  period to wind exactly once on the orbit passing through $(u_+,0)$ can be explicitly written as 
	\begin{equation}
		\label{E34}
		\tau_\l(e)
		:=4\int_0^1\frac{{\mathrm{d}}s}{\sqrt{\l(1-s^2)+\frac{a}{2}u_+^{2}(e)(1-s^{4})}}.
	\end{equation}
	For a fixed $\l$, since $u_+(e)$ is positive and increasing, the period $\tau_\l(e)$ is strictly decreasing with respect to $e>0$, 
	\begin{equation}
		\label{E35}
		\lim_{e\to 0^+}\tau_\l(e)=
		\begin{cases*}
			+\infty & \text{if $\l\leq 0$,}\\[6pt]
			\frac{2\pi}{\sqrt{\l}} & \text{if $\l>0$,}
		\end{cases*}
		\qquad\hbox{and}\qquad
		\lim_{e\to+\infty}\tau_\l(e)=0.
	\end{equation}
	
	The next result provides us with a necessary and sufficient condition for the existence of $T$-periodic solutions of \eqref{E32} with winding number $k\geq 0$ around the origin.
	
	\begin{theorem}
		\label{th:2.2}
		For any $k\geq 1$, problem \eqref{E32} admits a unique orbit $\mathscr{O}_{k,\l}$ of $T$-periodic solutions with winding number $k\geq1$ around $(0,0)$ if, and only if,
		\begin{equation}
			\label{E36}
			\l<\s_{k}:=\left(\frac{2\pi k}{T}\right)^2.
		\end{equation}
		Moreover, if $\l_1<\l_2$, $\mathscr{O}_{k,\l_1}$ encloses $\mathscr{O}_{k,\l_2}$, and $\mathscr{O}_{k,\l}$ shrinks to $(0,0)$ as $\l \to \s_k^-$.
		
    	\par
		Furthermore, recalling \eqref{eq:equilibria}, the two-branch curve $\mc{C}_0:=\left(\l,\o_{\pm}(\l),0\right)$, $\l<0$, is a continuous curve of constant $T$-periodic solutions with zero winding number around $(0,0)$ arising subcritically from $(\l,u,v)=(\s_0,0,0)$, where $\s_0=0$. 
	\end{theorem}
	
	\begin{proof}
		Any periodic solution of \eqref{E32} has winding number $k\geq 1$ around $(0,0)$ if, and only if, it corresponds to a closed orbit passing through $(u_+(e),0)$  with $\tau_\l(e)=T/k$. From \eqref{E35}, it follows that 
		\[
		\tau_\l((0,+\infty))=
		\begin{cases*}
			(0,+\infty) & \text{if $\l\leq 0$,}\\[6pt]
			\left(0,\tfrac{2\pi}{\sqrt{\l}}\right) & \text{if $\l> 0$.}
		\end{cases*}
		\]
		Hence, there exists $e_{k,\l}$ such that $\tau_\l(e_{k,\l})=T/k$ if, and only if, $\l\leq 0$, or $\l>0$ and $\tfrac{T}{k}<\tfrac{2\pi}{\sqrt{\l}}$. Moreover this $e_{k,\l}$ is unique since for $\l<\s_k$ fixed, the period $\tau_\l(e)$ is monotone decreasing with respect to $e>0$. Then, the orbit $\mathscr{O}_{k,\l}$ is given by
		\[
		\mathscr{O}_{k,\l}:=\left\{(u,v)\in\R^2\colon \mc{E}(u,v)=e_{k,\l}\right\}.
		\]
		\par
		On the other hand, for any fixed $e>0$, \eqref{E34} shows that $\tau_\l(e)$ is decreasing with respect to $\l$. Then, given $\l_1<\l_2<\s_k$ and $e_1,e_2>0$, 
		\[
		\tau_{\l_1}(e_1)=\tau_{\l_2}(e_2)\quad\hbox{implies}\quad e_1>e_2.
		\]
		As a consequence, $e_{k,\l_1}>e_{k,\l_2}$.	Hence, the monotonicity of $u_+(e)$ gives $u_+(e_{k,\l_1})>u_+(e_{k,\l_2})$. To conclude that $\mathscr{O}_{k,\l_1}$ encloses $\mathscr{O}_{k,\l_2}$, assume by contradiction that there exists $(\bar u,\bar v)\in\mathscr{O}_{k,\l_1}\cap \mathscr{O}_{k,\l_2}$. Thus,
		\[
		\frac{\bar v^2}{2}+\l_1 \frac{\bar u^2}{2}+a \frac{\bar u^4}{4}=e_{k,\l_1}>e_{k,\l_2}=\frac{\bar v^2}{2}+\l_2 \frac{\bar u^2}{2}+a \frac{\bar u^4}{4},
		\]
		which leads to $\l_1>\l_2$, a contradiction.
		
		Finally, since $\lim_{\l\to\s_k^-}\tfrac{2\pi}{\sqrt{\l}}=\tfrac{T}{k}$, the energy $e_{k,\l}$ goes to 0 as $\l\to\s_k^-$, showing that $\mathscr{O}_{k,\l}$ shrinks to the origin as $\l\to\s_k^-$.
	\end{proof}
	\noindent 
	
	With the notation of Theorem \ref{th:2.2}, we define the set of $T$-periodic solutions of \eqref{E32} with winding number $k\geq 1$ around the origin, 
			\[
			\mc{S}_{k}:=\{(\s_k,0,0)\}\cup\bigcup_{\l<\s_k} \left\{\l\right\}\times\mathscr{O}_{k,\l}.
			\]
	The explicit expression of the level sets of the energy $\mc{E}$, the definition of $\mathscr{O}_{k,\l}$ and the continuous dependence on the parameters show that $\mc{S}_k$ is a surface in the $(\l,u,v)$ space. 
	Figure \ref{Fig:2.2} represents $\mc{S}_1$ and $\mc{S}_2$, together with the curve of constant solutions $\mc{C}_0$ (in blue) and the line of trivial solutions $\{(\l,0,0)\colon\l\in\R\}$ (in black).
	\par
	Moreover, when considering \eqref{eq:1.1} with $a(t)\equiv0$ in $[0,T]$, the resulting linear problem is 
	\begin{equation}
    \label{eq:lin}
		\left\{
		\begin{array}{l}
			-u''=\l u, \quad \quad t\in (0,T),  \\[3pt]
			u(0)=u(T), \quad u'(0)=u'(T).
		\end{array}
		\right.
	\end{equation}
	A direct analysis of this problem (see also Lemma \ref{L3.2}) shows that it has $T$-periodic solutions with winding number $k\geq0$ if, and only if, $\l=\s_k$, and that the set of such solutions is a one-dimensional vector space if $k=0$ and a two-dimensional vector space if $k\geq1$. Hence, the bifurcation diagram $(\l,u,v)$ of solutions of \eqref{eq:lin} exhibits a line vertical bifurcation at $(\s_0,0,0)$ and a plane vertical bifurcation at $(\s_k,0,0)$ for $k\geq1$. This is profoundly different in the case of the superlinear problem with $a(t)\equiv a>0$, where the bifurcation diagrams consists of the curve $\mc{C}_0$ and of the surfaces $\mc{S}_k$, $k\geq 1$, as shown in Figure \ref{Fig:2.2}. 
	
	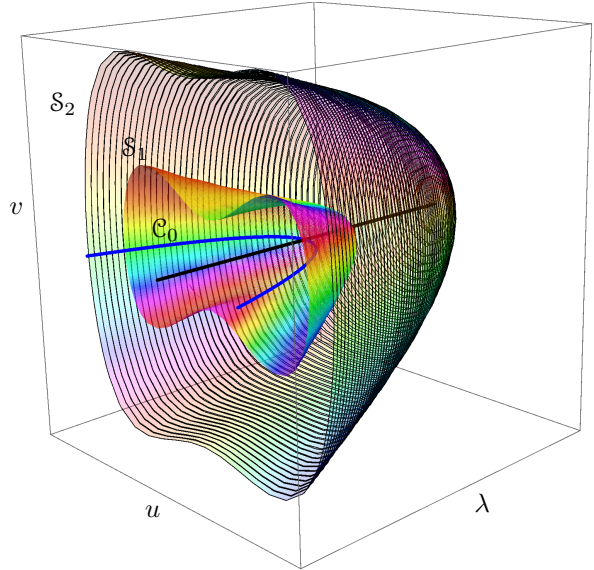
\begin{figure}[ht]
		\centering
		\begin{overpic}[scale=0.6]{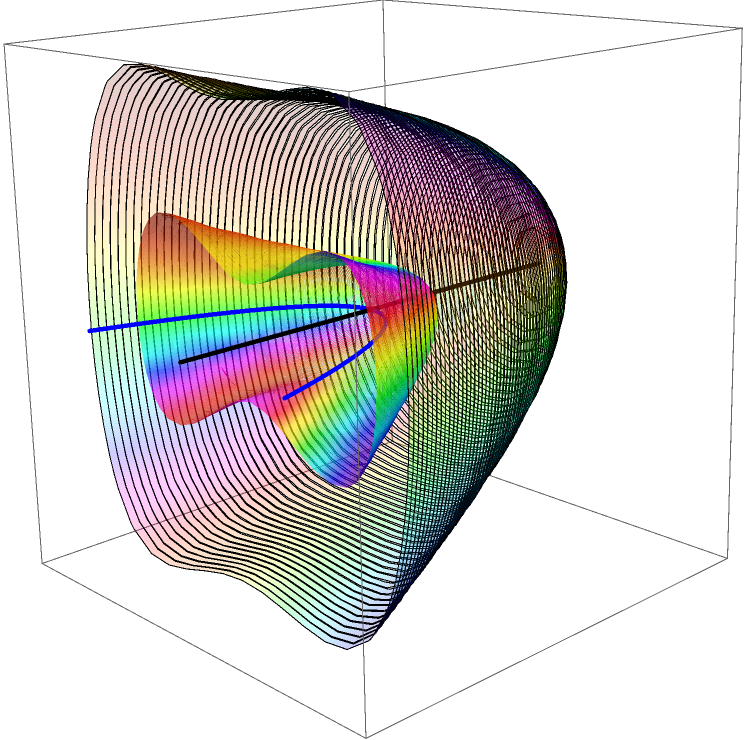}
			\put (79,10) {\small{$\l$}}
			\put (22,9) {\small{$u$}}
			\put (-1.5,62) {\small{$v$}}
			\put (23,58) {\small{$\mc{C}_0$}}
			\put (18,72) {\small{$\mc{S}_1$}}
			\put (5.5,80) {\small{$\mc{S}_2$}}
		\end{overpic}
		\caption{The constant solutions curve $\mc{C}_0$ and the surfaces $\mc{S}_1$ and $\mc{S}_2$ of problem \eqref{E32}, determined in Theorem \ref{th:2.2}.} \label{Fig:2.2}
	\end{figure} 
	
	\begin{remark}
		Although it goes beyond the scope of this paper, the existence of a surface of positive and negative $T$-periodic solutions of \eqref{E32} with winding number $k\geq1$ around $(\o_{\pm},0)$ (and zero winding number around $(0,0)$) emerging at $(\l,u,v)=(-\s_k,\o_{\pm}(-\s_k),0)$ can be proved with a similar argument to the one of the proof of Theorem \ref{th:2.2}.
	\end{remark}
	
	\subsection{A priori bounds}
	As a byproduct of the previous analysis, for solutions of the general case \eqref{eq:1.1} with a fixed winding number around the origin, we obtain: (i) an upper bound on $\l$ for the existence of solutions when $a(t)$ satisfies \eqref{ass:aloc}, and (ii) a priori bounds for $u$, provided $\l$ lies in an (arbitrary) compact interval of $\R$ and the weight $a(t)$ satisfies also \eqref{ass:aglob}. 
	
	\begin{proposition}
		\label{pr:2.4}
		(i) Assume that $a(t)$ satisfies \eqref{ass:aloc} and that $(\l,u)$ is a nontrivial solution of \eqref{eq:1.1} with winding number $k\geq0$ around the origin. Then, 
		\begin{equation}
			\label{eq:boundlambda}
			\l<\s_k=\left(\frac{2\pi k}{T}\right)^2.
		\end{equation}
		(ii) Suppose  that $a(t)$ satisfies \eqref{ass:aloc} and  \eqref{ass:aglob}. Then, for every compact subinterval $K\subset\mathbb{R}$ and for any $\l\in K$ such that $(\l,u)$ is a solution of with winding number $k\geq0$ around the origin, there exists $M=M(k,K)>0$ such that
		\begin{equation}
			\label{eq:boundu}
			\max_{t\in[0,T]}|u(t)|+\max_{t\in[0,T]}|u'(t)|<M.
		\end{equation}
	\end{proposition}
	
	\begin{proof} 
		(i) We begin with the case $k=0$ and assume that $u$ is a solution of \eqref{eq:1.1} with fixed sign. Then, integrating the differential equation in $(0,T)$ and using the periodic boundary conditions leads to
		\begin{align*}
			0=\int_{0}^{T}u''\, {\rm{d}}t & = \l \int_{0}^{T} u \, {\rm{d}}t + \int_{0}^{T}a(t) u^{3} \, {\rm{d}}t.
		\end{align*}
		Hence, $\l<0$ because, thanks to the fixed sign of $u$ and $a(t)$ in \eqref{ass:aloc}, $ \int_{0}^{T} u \, {\rm{d}}t$ and $\int_{0}^{T}a(t) u^{3} \, {\rm{d}}t$
		have the same sign. This proves \eqref{eq:boundlambda} for $k=0$.
		
		In order to prove \eqref{eq:boundlambda} for $k\geq 1$, we assume by contradiction that there exists $\l\geq\s_k$ and a solution $u$ of \eqref{eq:1.1} with winding number $k$ around the origin. Consider the system 
		\[
		\left\{
		\begin{array}{ll}
			u'=\sqrt{\l}v,\\[3pt]
			v'=-\sqrt{\l}u-\frac{1}{\sqrt{\l}}a(t)u^3,
		\end{array}
		\right.
		\]
		which is equivalent to \eqref{E37}. In this case, the derivative of the angular component is given by 
		\[
		\t'(t;z_0)=-\frac{\sqrt{\l}u^2(t;z_0)+\frac{1}{\sqrt{\l}}a(t)u^4(t;z_0)+\sqrt{\l}v^2(t;z_0)}{u^2(t;z_0)+v^2(t;z_0)},
		\]	
		and hence, by the definite sign of $a(t)$ in \eqref{ass:aloc}, 
		\begin{equation}
			\label{eq:angder}
			\t'(t;z_0)\lneq-\frac{\sqrt{\l}u^2(t;z_0)+\sqrt{\l}v^2(t;z_0)}{u^2(t;z_0)+v^2(t;z_0)}=-\sqrt{\l}\qquad\text{for a.e. $t\in[0,T]$}. 
		\end{equation}
		Thus, by \eqref{eq:rot} and \eqref{eq:angder}, since we are assuming $\l\geq \s_k$
		\[
		\mc{W}([0,T];z_0)>\frac{T}{2\pi}\sqrt{\l}\geq k.
		\]
		This gives a contradiction with the fact that $u$ has winding number $k$, and shows that \eqref{eq:boundlambda} holds. 
		\par
		(ii) Regarding \eqref{eq:boundu}, we fix $k\geq 0$ and $K\subset\mathbb{R}$. By \eqref{ass:aloc} and \eqref{ass:aglob}, there exists $j\in\{0,\dots,N-1\}$ such that 
		\[
		\inf_{t\in(t_j,t_{j+1})}a(t)=:c>0.
		\]
		Then, we can argue in the interval $(t_j,t_{j+1})$ as in Corollary 2.1 of Hartman \cite{Ha}. Indeed, the function 
		\[
		F(t,u)=\l u+a(t)u^3,\qquad  t\in(t_j,t_{j+1}),
		\]
		satisfies all the hypothesis of \cite[Corollary 2.1]{Ha} uniformly in $\l\in K$. This implies the existence of a constant $ M_j(k,K)= M_j>0$ such that the solutions of the equation in \eqref{eq:1.1} with at most $2k$ zeros on the interval $[t_j,t_{j+1}]$ satisfy 
		\begin{equation}
			\label{eq:boundsub}
			\max_{t\in[t_j,t_{j+1}]}|u(t)|+\max_{t\in[t_j,t_{j+1}]}|u'(t)|<M_j.
		\end{equation}
		Consider now, if it exists, an adjacent interval to $[t_j,t_{j+1}]$, e.g., $(t_{j+1},t_{j+2})$. If the weight $a(t)$ is positive there, the same argument can be applied to obtain a bound like \eqref{eq:boundsub} in such an interval, possibly with a different constant $M_{j+1}$.
		
		\par
		Instead, if  $a(t)=0$ for all $t\in(t_{j+1},t_{j+2})$, then any solution of \eqref{eq:1.1}  solves
		\[
		-u''=\l u,\quad t\in(t_{j+1},t_{j+2}),
		\]
		with $|u(t_{j+1})|<M_j$ and $|u'(t_{j+1})|<M_j$. The energy conservation gives
		\[
		u'(t)^2+\l u(t)^2=u'(t_{j+1})^2+\l u(t_{j+1})^2<(1+C)M_j^2 \quad \text{for all $t\in (t_{j+1},t_{j+2})$,}
		\]
		where $C:=\max\{|\l|\colon \l\in K\}$. If $\l\geq 0$, we directly obtain the desired bound $M_{j+1}$ in $[t_{j+1},t_{j+2}]$. On the other hand, if $\l<0$, the maximum of $|u|$ in $[t_{j+1},t_{j+2}]$ is achieved at the boundary of this interval, and the bound follows as in the proof of \eqref{eq:boundsub}.
		
		Finally, by repeating this argument in all the intervals of the partition of $[0,T]$ given by \eqref{ass:aglob} and setting $M:=\max_{j\in\{0,\dots,N-1\}}M_j$, \eqref{eq:boundu} holds. 
	\end{proof}

	\section{Local bifurcation analysis}
	\label{section:3}
	\noindent In this section we perform a local analysis for the general problem \eqref{eq:1.1} under assumption \eqref{ass:aloc}. Such an analysis will be based on the identification of bifurcation points from the trivial solution $(\l,0)$ and will rely on the general functional analytic framework that we present here. Then, in the next subsections, we will consider several specific cases, depending on the bifurcation points and/or the types of weights $a(t)$.
	
	We set $\T:=\R/(T\Z)$ and, given $s\in \N\cup \{0\}$, we consider the Sobolev spaces (see \cite{Ro}),
	\begin{align*}
		H^{s}(\mathbb{T}) &: =\left\{ u\in H^{s}([0,T]) \colon D^{\alpha}u(0)=D^{\alpha}u(T), \, \, \a\in\Z\cap[0,s)\right\} \\
		&\phantom{:} = \left\{u(t)=\sum_{k\in\Z}c_{k} e^{i 2\pi k t / T} \colon \bar{c}_{k}=c_{-k}, \; \sum_{k\in\Z} k^{2s} \left|c_{k}\right|^{2} < +\infty\right\},
	\end{align*}
	where we use the usual notation $H^0(\T)=L^2(\T)$. In particular, given $u\in L^2(\T)$, we can uniquely identify $u$ with the double-sided sequence of coefficients of its Fourier series expansion:
	$$u\longleftrightarrow\{\widehat{u}(k)\}_{k\in\Z}, \qquad \text {with} \qquad u(t)=\sum_{k\in\Z}\widehat{u}(k) e^{i 2\pi k t / T}.$$
	In $H^{s}(\mathbb{T})$, we will work with the following inner product and the corresponding induced norm
	$$\left(u,v\right)_{H^{s}}:= \sum_{0\leq \alpha \leq s} \int_0^T D^{\a}u\, D^{\a}v\, {\rm{d}}t, \qquad
	\left\|u\right\|_{H^{s}}:=\sqrt{\left(u,u\right)_{H^{s}}}=\left(\sum_{0\leq \alpha \leq s} \left\|D^{\alpha} u\right\|^{2}_{L^{2}}\right)^{\frac{1}{2}}.
	$$
	Moreover, when working with the Fourier series expansion, it will be convenient considering the following equivalent norm:
	$$
	\left\|u\right\|_{H^{s}_{\mathbf{f}}}=\left(\sum_{k\in\Z}\left(1+|k|^{2s}\right)\left|\widehat{u}(k)\right|^{2}\right)^{\frac{1}{2}}.$$
	The solutions of the periodic boundary value problem \eqref{eq:1.1} can be rewritten as the zeros of the nonlinear operator
	\begin{equation}
		\label{EqDe}
	\mf{F}\colon\R\times H^{2}(\mathbb{T})\longrightarrow L^{2}(\mathbb{T}), \qquad \mf{F}(\l,u):=u''+\l u + a(t) u^{3}.
    \end{equation}
	We denote by $\mc{T}$ the set of trivial solutions of $\mf{F}$, that is,
	\begin{equation}
		\label{eq:trivial}
		\mc{T}:=\{(\l,0)\in\R\times  H^{2}(\mathbb{T}) \colon \l\in\R\}.
	\end{equation}
	The following lemma collects some regularity properties of the operator $\mf{F}$.
	\begin{lemma}
		\label{L1}
		The operator $\mf{F}$ is of class $\mc{C}^{\infty}$, i.e., $\mf{F}\in\mc{C}^{\infty}(\R\times H^{2}(\mathbb{T}), L^{2}(\mathbb{T}))$. Moreover, its derivatives are given by
		\begin{align*}
		\partial_{u}\mf{F}(\l,u)&\colon H^{2}(\mathbb{T}) \longrightarrow L^{2}(\mathbb{T}), & & \partial_{u}\mf{F}(\l,u)\left[v\right]=v''+\l v+3a(t)u^{2}v, \\
		\partial_{\l}\mf{F}(\l,u)&\colon \R\longrightarrow L^{2}(\mathbb{T}), & & \partial_{\l}\mf{F}(\l,u)\left[\mu\right]=\mu u, \nonumber \\
		\partial_{uu}\mf{F}(\l,u)&\colon H^{2}(\mathbb{T})\times H^{2}(\mathbb{T})\longrightarrow L^{2}(\mathbb{T}), & & \partial_{uu}\mf{F}(\l,u)\left[v_{1},v_{2}\right]=6 a(t) u v_{1} v_{2}, \nonumber \\
		\partial_{\l u}\mf{F}(\l,u)&\colon \R \times H^{2}(\mathbb{T})\longrightarrow L^{2}(\mathbb{T}), & & \partial_{\l u}\mf{F}(\l,u)\left[\mu, v\right]=\mu v. \nonumber\\
		\partial_{\l\l}\mf{F}(\l,u)&\colon \R\times \R \longrightarrow L^{2}(\mathbb{T}), & & \partial_{\l\l}\mf{F}(\l,u)\left[\mu_{1},\mu_{2}\right]=0, \nonumber \\
		\partial_{uuu}\mf{F}(\l,u)&\colon H^{2}(\mathbb{T})\times H^{2}(\mathbb{T})\times H^{2}(\mathbb{T})\longrightarrow L^{2}(\mathbb{T}), & & \partial_{uuu}\mf{F}(\l,u)\left[v_{1},v_{2},v_{3}\right]=6 a(t) v_{1} v_{2} v_{3}, \nonumber \\
		\partial_{\l uu}\mf{F}(\l,u)&\colon \R\times H^{2}(\mathbb{T})\times H^{2}(\mathbb{T})\longrightarrow L^{2}(\mathbb{T}), & & \partial_{\l uu}\mf{F}(\l,u)\left[\mu,v_{1},v_{2}\right]=0, \nonumber \\
		\partial_{\l \l u}\mf{F}(\l,u)&\colon \R\times \R\times H^{2}(\mathbb{T})\longrightarrow L^{2}(\mathbb{T}), & & \partial_{\l \l u}\mf{F}(\l,u)\left[\mu_1,\mu_2,v\right]=0, \nonumber \\
		\partial_{\l \l\l}\mf{F}(\l,u)&\colon \R\times \R\times \R\longrightarrow L^{2}(\mathbb{T}), & & \partial_{\l \l\l}\mf{F}(\l,u)\left[\mu_1,\mu_{1},\mu_{3}\right]=0. \nonumber	
		\end{align*}
	\end{lemma} 
		
	\begin{proof}
		Let us start by proving the continuity of $\mf{F}$. Let $\{(\l_{n},u_{n})\}_{n\in\N}\subset\R\times H^{2}(\mathbb{T})$ be a sequence satisfying
		$$\lim_{n\to+\infty}\l_n = \l_0 \quad \text{ and } \quad
		\lim_{n\to+\infty}u_{n}=u_{0} \quad \text{ in  $H^{2}(\mathbb{T})$}.$$
		Let $N\in\N$ such that $\left\|u_{n}-u_{0}\right\|_{H^{2}}<1$ for each $n\geq N$. In particular, this implies that
		$$\left\|u_{n}\right\|_{H^{2}}<1+\left\|u_{0}\right\|_{H^{2}}, \quad n\geq N.$$
		We rewrite the difference of the corresponding operator as
		\begin{align*}
			&\mf{F}(\l_{n},u_{n})-\mf{F}(\l_{0},u_{0})=(u_{n}-u_{0})''+(\l_{n}-\l_{0})u_{n}+\l_{0}(u_{n}-u_{0})+a(t)(u_{n}^{3}-u_{0}^{3}).
		\end{align*}
		Fix $n\geq N$. By taking the $L^{2}$-norm, we obtain
		\begin{align}
			\label{E8}
			\nonumber
			& \left\|\mf{F}(\l_{n},u_{n})-\mf{F}(\l_{0},u_{0})\right\|_{L^{2}} \\ \nonumber
			& \leq  \left\|(u_{n}-u_{0})''\right\|_{L^{2}}+\left|\l_{n}-\l_{0}\right|  \left\|u_{n}\right\|_{L^{2}}+\left|\l_{0}\right|  \left\|u_{n}-u_{0}\right\|_{L^{2}}+\left\|a\right\|_{L^{\infty}} \left\|u_{n}^{3}-u_{0}^{3}\right\|_{L^{2}} \\ \nonumber
			& \leq  \left\|u_{n}-u_{0}\right\|_{H^{2}}+\left|\l_{n}-\l_{0}\right| \left\|u_{n}\right\|_{H^{2}}+\left|\l_{0}\right| \left\|u_{n}-u_{0}\right\|_{H^{2}}+\left\|a\right\|_{L^{\infty}}\left\|u_{n}^{3}-u_{0}^{3}\right\|_{L^{2}} \\
			&\leq \left\|u_{n}-u_{0}\right\|_{H^{2}}+\left|\l_{n}-\l_{0}\right| \left(1+\left\|u_{0}\right\|_{H^{2}}\right)+\left|\l_{0}\right| \left\|u_{n}-u_{0}\right\|_{H^{2}}+\left\|a\right\|_{L^{\infty}}\left\|u_{n}^{3}-u_{0}^{3}\right\|_{L^{2}}.
		\end{align}
		In order to bound the nonlinear term, we proceed as follows
		\begin{align}
			\left\|u_{n}^{3}-u_{0}^{3}\right\|_{L^{2}}^{2} & =\int_{0}^{T}\left|u_{n}^{3}-u_{0}^{3}\right|^{2} \, {\rm{d}}t= \int_{0}^{T}\left|u_{n}-u_{0}\right|^{2} \left|u_{n}^{2}+ u_{n} u_{0} + u_{0}^{2}\right|^{2} \, {\rm{d}}t  \\
			& \leq 	\left\|u_{n}^{2}+u_{n}u_{0}+u_{0}^{2}\right\|_{L^{\infty}}^{2} \int_{0}^{T} \left|u_{n}-u_{0}\right|^{2} \, {\rm{d}}t  \\
			& \leq \left(\left\|u_{n}\right\|_{L^{\infty}}^{2}+\left\|u_{n}\right\|_{L^{\infty}} \left\|u_{0}\right\|_{L^{\infty}}+\left\|u_{0}\right\|_{L^{\infty}}^{2}\right)^{2}  \left\|u_{n}-u_{0}\right\|^{2}_{L^{2}}. \label{eq:cotacubo}
		\end{align}
		The Sobolev embedding $H^{2}(\mathbb{T})\hookrightarrow \mc{C}^{0}(\mathbb{T})$ gives
		\begin{align*}
			&\left\|u_{n}\right\|_{L^{\infty}}\leq C \left\|u_{n}\right\|_{H^{2}} < C \left(1+\left\|u_{0}\right\|_{H^{2}}\right), \quad n\geq N,\\
			& \left\|u_{0}\right\|_{L^{\infty}}\leq C \left\|u_{0}\right\|_{H^{2}},
		\end{align*}
		for some constant $C>0$. Hence, we obtain
		\begin{align}
			\label{E9}
			\left\|u_{n}^{3}-u_{0}^{3}\right\|_{L^{2}}^{2} & \leq \left(\left\|u_{n}\right\|_{L^{\infty}}^{2}+\left\|u_{n}\right\|_{L^{\infty}} \left\|u_{0}\right\|_{L^{\infty}}+\left\|u_{0}\right\|_{L^{\infty}}^{2}\right)^{2} \left\|u_{n}-u_{0}\right\|^{2}_{L^{2}} \nonumber \\
			&\leq \left(C^2\left(1+\left\|u_{0}\right\|_{H^{2}}\right)^{2}+C\left(1+\left\|u_{0}\right\|_{H^{2}}\right)\left\|u_{0}\right\|_{L^{\infty}}+\left\|u_{0}\right\|^{2}_{L^{\infty}}\right)^{2} \left\|u_{n}-u_{0}\right\|_{H^{2}}^{2}. 
		\end{align}
		Thanks to \eqref{E8} and \eqref{E9}, we obtain
		$$\lim_{n\to+\infty}\|\mf{F}(\l_{n},u_{n})-\mf{F}(\l_{0},u_{0})\|_{L^{2}}=0,$$
		which proves the continuity of $\mf{F}$. Let us prove now that
		\begin{equation}
			\label{E7.1}
			\partial_{u}\mf{F}(\l,u)\left[v\right]=v''+\l v+3a(t)u^{2}v, \quad v\in H^{2}(\mathbb{T}),
		\end{equation}
		for each $(\l,u)\in\R\times H^{2}(\mathbb{T})$. We start by rewriting
		\begin{align*}
			\mf{F}(\l,u+v)-\mf{F}(\l,u)-v''-\l v-3 a(t)u^{2} v=a(t)\left[(u+v)^{3}-u^{2}(u+3 v)\right]=a(t)v^{2}(v+3u).
		\end{align*}
		By taking the $L^{2}$-norm and using the Sobolev embedding $H^{2}(\mathbb{T})\hookrightarrow \mc{C}^{0}(\mathbb{T})$, we obtain 
		\begin{align*}
			\left\|a(t)v^{2}(v+3u)\right\|_{L^{2}}^{2} & \leq  T \left\|a\right\|_{L^{\infty}}^{2} \left(\left\|v\right\|_{L^{\infty}}+3 \left\|u\right\|_{L^{\infty}}\right)^2\left\|v\right\|_{L^{\infty}}^{4}  \\
			& \leq C^6 T \left\|a\right\|_{L^{\infty}}^{2} \left(\left\|v\right\|_{H^{2}}+3\left\|u\right\|_{H^{2}}\right)^2 \left\|v\right\|^{4}_{H^{2}} .
		\end{align*}
		Therefore we deduce that
		\begin{align*}
			\lim_{\left\|v\right\|_{H^{2}}\to 0}\frac{\left\|\mf{F}(\l,u+v)-\mf{F}(\l,u)-v''-\l v-3 a(t)u^{2} v\right\|_{L^{2}}}{\left\|v\right\|_{H^{2}}}=0,
		\end{align*}
		which proves \eqref{E7.1}. The proofs of the remaining identities and the $\mc{C}^{\infty}$ regularity follow similar patterns.
	\end{proof}
		
	The next result is fundamental in order to use topological degree arguments like those necessary to obtain global bifurcation results (see Appendix \ref{app:B.3}).
		
	\begin{lemma}
		\label{LF4}
		The operator $\mf{F}$ is proper on closed and bounded subsets of $\R\times H^{2}(\mathbb{T})$.
	\end{lemma}
		
	\begin{proof}
		It suffices to prove that the restriction of $\mf{F}$ to the closed subset ${K}:=[\lambda_{-},\lambda_{+}]\times \bar {B}_{R}$ is proper, where $\lambda_{-}<\lambda_{+}$, and $B_{R}$ indicates the open ball of $H^{2}(\mathbb{T})$ of radius $R>0$ centered at $0$. According to \cite[Th. 2.7.1]{B}, we must check that $\mf{F}(K)$ is closed in $L^{2}(\mathbb{T})$ and that, for every $f\in L^{2}(\mathbb{T})$, the set $\mf{F}^{-1}(f)\cap {K}$ is compact in $\mathbb{R}\times H^{2}(\mathbb{T})$.
		\par
		To show that $\mf{F}(K)$ is closed in $L^{2}(\mathbb{T})$, let $\left\{f_{n}\right\}_{n\in\mathbb{N}}$ be a sequence in $\mf{F}(K)\subset L^{2}(\mathbb{T})$ such that
		\begin{equation}
			\label{7.2}
			\lim_{n\to+\infty} f_{n}=f \quad \text{in } L^{2}(\mathbb{T}).
		\end{equation}
		Then, there exists a sequence $\left\{(\l_{n},u_{n})\right\}_{n\in\mathbb{N}}$ in $K$ such that
		\begin{equation}
			\label{7.3}
			f_{n}=\mf{F}(\l_{n},u_{n}) \quad \text{for all $n\in\mathbb{N}$}.
		\end{equation}
		Due to the compactness of the embedding $H^{2}(\mathbb{T})\hookrightarrow \mc{C}^{1}(\mathbb{T})$, we can extract a subsequence $\left\{(\l_{n_{k}},u_{n_{k}})\right\}_{k\in\mathbb{N}}$ such that, for some $(\l_{0},u_{0})\in [\l_-,\l_+]\times \mc{C}^{1}(\mathbb{T})$, $u_0\in\bar B_R$,
		\begin{equation}
			\label{7.4}
			\lim_{k\to+\infty} \l_{n_{k}} = \l_{0} \quad \text{and} \quad
			\lim_{k\to+\infty} u_{n_{k}} =u_{0} \quad \text{in $\mc{C}^{1}(\mathbb{T})$}.
		\end{equation}
		As a direct consequence of \eqref{7.2}, \eqref{7.3} and \eqref{7.4}, it becomes apparent that $u_{0}$ must be a weak solution of the nonlinear elliptic problem
		\begin{equation}
			\label{7.5}
			\left\{
			\begin{array}{l}
				u_{0}''+\l_{0} u_{0} + a(t) u_{0}^{3}=f, \quad \quad t\in (0,T),  \\
				u_{0}(0)=u_{0}(T), \quad u_{0}'(0)=u_{0}'(T).
			\end{array}
			\right.
		\end{equation}
		By elliptic regularity, $u_{0}\in H^{2}(\mathbb{T})$ and $f=\mf{F}(\l_{0},u_{0})$. Therefore, $f\in \mf{F}(K)$, and $\mf{F}(K)$ is closed.
		\par
		Now, fix $f\in L^{2}(\mathbb{T})$. To show that $\mf{F}^{-1}(f)\cap K$ is compact in $[\l_-,\l_+]\times H^{2}(\mathbb{T})$, let $\{(\l_{n},u_{n})\}_{n\in\mathbb{N}}$ be a sequence in $\mf{F}^{-1}(f)\cap K$. Then,
		\begin{equation}
			\label{7.6}
			\mf{F}(\l_{n},u_{n})=f \quad \text{for all $n\in\mathbb{N}$}.
		\end{equation}
		Due again to the compactness of the embedding $H^{2}(\mathbb{T}) \hookrightarrow \mc{C}^{1}(\mathbb{T})$, we can extract a subsequence $\left\{(\l_{n_{k}},u_{n_{k}})\right\}_{k\in\mathbb{N}}$ such that, for some $(\l_{0},u_{0})\in [\l_-,\l_+]\times \mc{C}^{1}(\mathbb{T})$,  \eqref{7.4} holds.	As above,  $u_{0}\in \mc{C}^{1}(\mathbb{T})$ is a weak solution of \eqref{7.5} and, by elliptic regularity, $u_{0}\in H^{2}(\mathbb{T})$ and $\mf{F}(\l_{0},u_{0})=f$. In particular, for every $k\in\mathbb{N}$,
		\begin{equation*}
			-\left(u_{n_{k}}-u_{0}\right)''=\l_{n_k} \left(u_{n_{k}}-u_{0}\right)+\left(\l_{n_k}-\l_0\right)u_0+ a(t)\left(u^{3}_{n_{k}}-u^{3}_{0}\right) \qquad \text{in $(0,T)$}.
		\end{equation*}
		By $L^2$-elliptic estimates, there exists a positive constant $C>0$ such that
		\begin{equation*}
			\left\|u_{n_{k}}-u_{0}\right\|_{H^{2}}\leq C \left( \left\|u_{n_{k}}-u_{0}\right\|_{L^{2}}+ \left\|u^{3}_{n_{k}}-u^{3}_{0}\right\|_{L^{2}}+\left|\l_{n_k}-\l_0\right|\right)
		\end{equation*}
		for sufficiently large $k\in\N$. Therefore, since \eqref{7.4} implies that $u_{n_k}\to u_0$ in $L^2(\T)$ as $k\to+\infty$, the previous relation and \eqref{eq:cotacubo} imply that $u_{n_k}\to u_0$ in $H^2(\T)$, which concludes the proof.
	\end{proof}
		
	\begin{lemma}
		\label{L3}
		The linearization of $\mf{F}$ with respect to $u$ at $(\l,u)\in \R\times H^{2}(\mathbb{T})$,
		$$\partial_{u}\mf{F}(\l,u)\colon H^{2}(\mathbb{T}) \longrightarrow L^{2}(\mathbb{T}), \quad \partial_{u}\mf{F}(\l,u)\left[v\right]=v''+\l v+3a(t)u^{2}v,$$
		belongs to the class $\Phi_{0}(H^{2}(\mathbb{T}), L^{2}(\mathbb{T}))$ of Fredholm operators of index zero.
	\end{lemma}
		
	\begin{proof}
		Observe first that, as $u\in H^{2}(\mathbb{T})$, necessarily $u\in L^{\infty}(\mathbb{T})$ by the Sobolev embedding $H^{2}(\mathbb{T})\hookrightarrow L^{\infty}(\mathbb{T})$. Take $\mu\in\R$ such that $\mu>-\l$.
		Then, the bilinear form
		$$\mf{a}\colon H^{1}(\mathbb{T})\times H^{1}(\mathbb{T}) \longrightarrow \R, \quad \mf{a}(v_{1},v_{2}):=\int_{0}^{T}v_{1}' v_{2}' \, {\rm{d}}t + \int_{0}^{T}\left(\l+\mu+3 a(t) u^{2}\right)v_{1}v_{2} \, {\rm{d}}t,$$
		is continuous and coercive. By the Lax--Milgram theorem and elliptic regularity theory, we deduce that 
		$$\partial_{u}\mf{F}(\l,u)+\mu J \in GL(H^{2}(\mathbb{T}), L^{2}(\mathbb{T})),$$
		where $J\colon H^{2}(\mathbb{T}) \hookrightarrow L^{2}(\mathbb{T})$ is the inclusion. By the Rellich--Kondrachov theorem, the operator $J$ is compact and, since
		$$\partial_{u}\mf{F}(\l,u)=(\partial_{u}\mf{F}(\l,u)+\mu J)-\mu J,$$
		the operator $\partial_{u}\mf{F}(\l,u)$ can be expressed as the sum of an isomorphism and a compact operator. Therefore, by \cite[Chap. XV, Th. 4.1]{GGS}, $\partial_{u}\mf{F}(\l,u)$ is Fredholm of index zero.
	\end{proof}
		
	\noindent In particular, the linearization $\mf{L}(\l):=\p_u\mf{F}(\l,0)$ is given by
	$$\mf{L}(\l)\colon H^{2}(\mathbb{T}) \longrightarrow L^{2}(\mathbb{T}), \quad \mf{L}(\l)\left[v\right]=v''+\l v.$$
	As a direct consequence of Lemmas \ref{L1} and \ref{L3}, we deduce that $\mf{L}(\l)$ is a Fredholm operator of index zero for each $\l\in\R$, and that the family
	$$\mf{L}\colon\R\longrightarrow\Phi_{0}(H^{2}(\mathbb{T}), L^{2}(\mathbb{T}))$$
	is continuous. To analyze the spectral properties of the family $\mf{L}$, we will use the setting of nonlinear spectral theory collected in Appendix \ref{A1} for the convenience of the reader.
	Recall that the \textit{generalized spectrum} of $\mf{L}$, formed by the so-called \textit{generalized eigenvalues}, is defined as
	$$\Sigma(\mf{L}):=\left\{\l\in \R \colon \mf{L}(\l)\notin GL(H^{2}(\mathbb{T}), L^{2}(\mathbb{T}))\right\}.$$
	The next lemma describes the set $\Sigma(\mf{L})$; we recall the values $\s_k=\left(\frac{2\pi k}{T}\right)^2$, $k\in\N\cup \{0\}$, already introduced in Section \ref{section:2}.

	\begin{lemma}
		\label{L3.2}
		The generalized spectrum of the family $\mf{L}(\l)$ is given by
		\begin{equation}
			\label{E1.1}
			\Sigma(\mf{L})=\left\{\s_k\right\}_{k\in\N\cup\{0\}}=\left\{\left(\frac{2\pi k}{T}\right)^2 \colon k\in\N\cup \{0\}\right\}.
		\end{equation}
		Moreover:
		\begin{itemize}
			\item[(i)] for $\l\notin \Sigma(\mf{L})$, the inverse $\mf{L}^{-1}(\l)\colon L^{2}(\mathbb{T})\to H^{2}(\mathbb{T})$, is given by
			$$\mf{L}^{-1}(\l)\left[v\right]=\sum_{k\in \Z}\frac{\widehat{v}(k)}{\l-\left(\frac{2\pi k}{T}\right)^2}e^{i 2\pi k t/T}, \quad v\in L^{2}(\mathbb{T});$$
			\item[(ii)] the nullspace of $\mf{L}(\l)$ is given by
			\[
			N[\mf{L}(\l)]=
			\begin{cases}
				\Span\{1\}, & \text{for $\l=\s_0=0$,}\\
					\Span\left\{\cos\left(\frac{2\pi k }{T}t\right),\sin\left(\frac{2\pi k }{T}t\right)\right\}, & \text{for $\l=\s_k$, $k\in\N$.}
			\end{cases}
			\] 
		\end{itemize}
	\end{lemma}
	
	\begin{proof}
		Let $u\in N[\mf{L}(\l)]\subset H^{2}(\mathbb{T})$. Then $\mf{L}(\l)[u]=0$ or, equivalently, $u''+\l u=0$. By plugging the Fourier series expansion of $u$ in this differential equation, we obtain
		\begin{equation}
			\label{E2}
			\sum_{k\in \Z}\left(-\left(\tfrac{2\pi k}{T}\right)^{2}+\l\right)\widehat{u}(k) e^{i 2\pi k t / T}=0.
		\end{equation}
		Hence, if $\l\neq \left(\tfrac{2\pi k}{T}\right)^{2}$ for all $k\in\N\cup\{0\}$, necessarily $\widehat{u}(k)=0$ for all $k\in\Z$ and this implies that $u=0$. Thus, $N[\mf{L}(\l)]=\{0\}$ for $\l\neq \left(\tfrac{2\pi k}{T}\right)^{2}$.
		\par For $\l\neq \left(\tfrac{2\pi k}{T}\right)^{2}$ and $v\in L^{2}(\mathbb{T})$, consider the equation $\mf{L}(\l)[u]=v$, which can be rewritten as 
		\begin{equation}
			\label{E4}
			u''+\l u=v, \quad u\in H^{2}(\mathbb{T}).
			\end{equation}
		By considering the Fourier series expansions of $u$ and $v$, the previous relation becomes
		$$\sum_{k\in \Z}\left(-\left(\tfrac{2\pi k}{T}\right)^{2}+\l\right)\widehat{u}(k) e^{i 2\pi k t / T}=\sum_{k\in \Z} \widehat{v}(k) e^{i 2\pi k t / T}.$$
		This implies that 
		$$\left(-\left(\tfrac{2\pi k}{T}\right)^{2}+\l\right)\widehat{u}(k)=\widehat{v}(k), \quad \text{for all $k\in\Z$},$$
		and, consequently,
		\begin{equation}
			\label{E5}
			u(t)=\sum_{k\in \Z}\widehat{u}(k) e^{i 2\pi k t / T}= \sum_{k\in \Z}\frac{\widehat{v}(k)}{\l-\left(\tfrac{2\pi k}{T}\right)^{2}}e^{i 2\pi k t/T}.
		\end{equation}
		Moreover, $u\in H^{2}(\mathbb{T})$, because
		\begin{equation}
			\label{E7}
			\|u\|_{H^{2}_{\mathbf{f}}}^{2}=\sum_{k\in \Z}\frac{1+k^{4}}{\left(\l-\left(\tfrac{2\pi k}{T}\right)^{2}\right)^2}\left|\widehat{v}(k)\right|^{2}\leq \left\|\left\{x_{k}(\l)\right\}_{k\in\Z}\right\|_{\ell^{\infty}} \sum_{k\in \Z}\left|\widehat{v}(k)\right|^{2} \leq \left\|\left\{x_{k}(\l)\right\}_{k\in\Z}\right\|_{\ell^{\infty}} \left\|v\right\|_{L^{2}_{\mathbf{f}}}^{2},
		\end{equation}
		where the sequence $\left\{x_{k}(\l)\right\}_{k\in\Z}\in \ell^{\infty}(\Z)$ is defined as 
		$$x_{k}(\l):=\frac{1+k^{4}}{\left(\l-\left(\tfrac{2\pi k}{T}\right)^{2}\right)^2}, \quad k\in\Z.$$
		Consequently, $u$ defined in \eqref{E5} satisfies $\mf{L}(\l)[u]=v$ and, hence, $\mf{L}(\l)$ is onto. Moreover, \eqref{E7} shows that the operator $\mf{L}^{-1}(\l)\colon L^{2}(\mathbb{T})\to H^{2}(\mathbb{T})$ is bounded.
		By the open mapping theorem, $\mf{L}(\l)\in GL(H^{2}(\mathbb{T}), L^{2}(\mathbb{T}))$. This proofs (i) and shows that
		$$\Sigma(\mf{L})\subset \left\{\left(\tfrac{2\pi k}{T}\right)^{2} \colon k\in\N\cup \{0\}\right\}.$$
		Let $\l=\s_0=0$ and $u\in N[\mf{L}(0)]$. Then, \eqref{E2} implies that 
		$$\widehat{u}(k)=0, \quad \text{for all $k\in\Z\setminus\{0\}$}.$$
		Hence, $u(t)=\widehat{u}(0)$, i.e., $u$ is constant. Thus, $0\in\Sigma(\mf{L})$, and $N[\mf{L}(0)]=\Span\{1\}\neq 0$. If $\l=\left(\tfrac{2\pi k_0}{T}\right)^{2}$ for some $k_{0}\in\N$ and $u\in N[\mf{L}(\l)]$, then equation \eqref{E2} implies that 
		$$\widehat{u}(k)=0, \quad \text{for all $k\in\Z\setminus\{k_0,-k_0\}$}.$$
		Hence, $u(t)=\widehat{u}(k_{0})e^{i 2\pi k_{0} t/ T}+\widehat{u}(-k_{0})e^{-i 2\pi k_{0} t/ T}$ and, therefore,
		$$N\left[\mf{L}\left(\left(\tfrac{2\pi k_0}{T}\right)^{2}\right)\right]=\Span\left\{e^{i 2\pi k_{0} t/ T},e^{-i 2\pi k_{0} t/ T}\right\}=\Span\left\{\cos\left(\frac{2\pi k_0 }{T}t\right),\sin\left(\frac{2\pi k_0 }{T}t\right)\right\}.$$
		Then $\left(\tfrac{2\pi k}{T}\right)^{2}\in \Sigma(\mf{L})$ for each $k\in\N$, and \eqref{E1.1} and (ii) are proved.
	\end{proof}
		
	\begin{lemma}
		\label{le:3.5}
		For each $\l\in\R$, $\mf{L}(\l)\colon H^{2}(\mathbb{T})\subset L^{2}(\mathbb{T}) \to L^{2}(\mathbb{T})$ is a self-adjoint operator.
	\end{lemma}
	
	\begin{proof}
		Let $u, v\in H^{2}(\mathbb{T})$. Then, by using the periodic boundary conditions when we integrate by parts,
		\begin{align*}
			\left(\mf{L}(\l)[u],v\right)_{L^{2}} & =\int_{0}^{T}\mf{L}(\l)[u]v \, {\rm{d}}t=\int_{0}^{T} \left(u''+\l u\right)v \, {\rm{d}}t \\
			& =\int_{0}^{T}u''v \, {\rm{d}}t + \l \int_{0}^{T}uv \, {\rm{d}}t  = \int_{0}^{T}v'' u \, {\rm{d}}t + \l \int_{0}^{T}uv \, {\rm{d}}t \\
			&= \int_{0}^{T}\left(v''+\l v\right)u \, {\rm{d}}t=\int_{0}^{T} \mf{L}(\l)[v] u \, {\rm{d}}t=\left(u,\mf{L}(\l)[v]\right)_{L^{2}}.
		\end{align*}
		Consequently, $\mf{L}(\l)$ is a symmetric operator. Moreover, one can easily show that the domain of $\mf{L}(\l)$ coincides with the one of the adjoint operator $\mf{L}^{\ast}(\l)$. Hence, $\mf{L}(\l)$ is self-adjoint.
	\end{proof}
		
	In the following result, we compute the algebraic multiplicity of each generalized eigenvalue in $\Sigma(\mf{L})$. 
	
	\begin{proposition}
		\label{P3.5}
		For each $k\in\N\cup\{0\}$, the generalized eigenvalue $\s_k$ is $1$-transversal (cf., Definition \ref{de2.3}), and its generalized algebraic multiplicity is
		$$\chi[\mf{L},\s_{k}]=
		\begin{cases*}
			1 & \text{ if $k=0$,}\\
			2 & \text{ if $k\in\N$.}
		\end{cases*}
		$$
	\end{proposition}
			
	\begin{proof}
		An elementary computation gives
		$$\mf{L}_{1}(\l):=\frac{{\rm{d}}\mf{L}}{{\rm{d}}\l}(\l)= J, \quad \l\in\R,$$
		where $J\colon H^{2}(\mathbb{T})\hookrightarrow L^{2}(\mathbb{T})$ is the inclusion. Therefore,
		$$\mf{L}_{1}(\s_{0})[1]=1,$$
		and, for each $k\in\N$,
		$$ \mf{L}_{1}(\s_{k})\left[\cos\left(\frac{2\pi k }{T}t\right)\right]=\cos\left(\frac{2\pi k }{T}t\right), \quad \mf{L}_{1}(\s_{k})\left[\sin\left(\frac{2\pi k }{T}t\right)\right]=\sin\left(\frac{2\pi k }{T}t\right).$$
		As $\mf{L}(\l)$ is a Fredholm operator for all $\l\in\R$ (see Lemma \ref{L3}), $R[\mf{L}(\l)]$ is closed. Therefore,
		$$R[\mf{L}(\l)]=\overline{R[\mf{L}(\l)]}=N[\mf{L}^{\ast}(\l)]^{\perp}=N[\mf{L}(\l)]^{\perp},$$
		as $\mf{L}(\l)$ is self-adjoint, thanks to Lemma \ref{le:3.5}. Consequently,
		\begin{equation}
			\label{eq:rango}
			R[\mf{L}(\s_{0})]=\left\{u\in L^{2}(\mathbb{T}) \colon \int_{0}^{T} u(t) \, {\rm{d}}t=0 \right\},
		\end{equation}
		while, for $k\in\N$,
		$$R[\mf{L}(\s_{k})]=\left\{u\in L^{2}(\mathbb{T}) \colon \int_{0}^{T}\cos\left(\frac{2\pi k }{T}t\right) \, u(t) \, {\rm{d}}t=\int_{0}^{T} \sin\left(\frac{2\pi k }{T}t\right) \, u(t) \, {\rm{d}}t=0 \right\}.$$
		Hence, 
		$$\mf{L}_{1}(\s_{0})\left[1\right]\notin R[\mf{L}(\s_0)],$$
		and, for $k\in\N$, 
		$$\mf{L}_{1}(\s_{k})\left[\cos\left(\frac{2\pi k }{T}t\right)\right]\notin R[\mf{L}(\s_{k})], \quad \text{and} \quad \mf{L}_{1}(\s_{k})\left[\sin\left(\frac{2\pi k }{T}t\right)\right]\notin R[\mf{L}(\s_{k})].$$
		Therefore, the transversality condition 
		$$\mf{L}_{1}(\s_{k})\left(N[\mf{L}(\s_{k})]\right)\oplus R[\mf{L}(\s_{k})]= L^{2}(\mathbb{T}),$$
		holds for each $k\in\N\cup\{0\}$. This implies that the eigenvalues $\s_{k}$ are $1$-transversal, and \eqref{CTCR} and Lemma \ref{L3.2}(ii) give that
		$$\chi[\mf{L},\s_{k}]=\dim N[\mf{L}(\s_{k})]=
		\begin{cases*}
			1 & \text{ if $k=0$,}\\
			2 & \text{ if $k\in\N$,}
		\end{cases*}
		$$
		as we wanted.
	\end{proof}
			
	\subsection{Local bifurcation at $\l=\s_0$}
	\label{sec:3.1}
			
	This subsection is devoted to the local bifurcation analysis of equation \eqref{eq:1.1} near the trivial solution $u=0$ at the eigenvalue $\s_0=0$. The main ingredient will be a direct application of the Crandall--Rabinowitz local bifurcation Theorem \ref{Cr-Rb}. This tool can be applied in this case because $\chi[\mf{L},\s_0]=1$. However, this is not the case for $\s_k$, $k\geq 1$, as $\chi[\mf{L},\s_k]=2$, and different approaches has to be followed in those situations (see Sections \ref{sec:3.2} and \ref{sec:3.3} below). Recalling the set $\mc{T}$ of \textit{trivial solutions} of $\mf{F}$, defined in \eqref{eq:trivial}, the set of \textit{non-trivial solutions} of $\mf{F}$ is defined as
	\begin{equation*}
		\mf{S}:=\left( \mf{F}^{-1}(0)\setminus \mc{T}\right)\cup \left\{(\l,0)\colon \l\in \Sigma(\mf{L})\right\}\subset \R\times H^{2}(\mathbb{T}).
	\end{equation*}
	Since $\mf{F}$ is continuous (see Lemma \ref{L1}), $\mf{S}$ is a closed subset of $\R\times H^{2}(\mathbb{T})$. The following is the main result of this subsection and characterizes the structure of $\mf{S}$ in a neighborhood of $(\s_0,0)$. Recall that $N[\mf{L}(\s_0)]=\Span\{1\}$, which explains the appearance of the function $1$ in the theorem.
	\begin{theorem}
		\label{LCRW}
		The point $(\s_0,0)\in \R \times H^{2}(\mathbb{T})$ is a bifurcation point of 
		$$\mf{F}\colon\R\times H^{2}(\mathbb{T})\longrightarrow L^{2}(\mathbb{T}), \quad \mf{F}(\l,u):=u''+\l u + a(t) u^{3},$$
		from the trivial branch $\mc{T}$ to a connected component $\mathscr{C}_{0}\subset\mf{S}$ of non-trivial solutions of \eqref{eq:1.1}. Moreover, setting
		$$Y:=\left\{u\in H^{2}(\mathbb{T}) \colon \left(u,1\right)_{H^{2}}=0 \right\},$$
		the following statements hold.
		\begin{enumerate}
			\item[(i)] {\rm\textbf{Existence.}} There exist $\e>0$ and two $\mc{C}^{\infty}$-maps
			\begin{equation}
				\L\colon (-\e,\e) \longrightarrow \R, \quad \L(0)=\s_0=0, \qquad \qquad \Gamma\colon (-\e,\e)\longrightarrow Y, \quad  \Gamma(0)=0,
				\label{Cu}
			\end{equation}
			such that, for each $s\in(-\e,\e)$,
			\begin{equation}
				\mathfrak{F}(\L(s),u(s))=0, \quad \text{with  $u(s):= s\left(1+\Gamma(s)\right)$}.
				\label{Cu1}
			\end{equation}
			We set
			\begin{align*}
				\mathscr{C}_{0,{\mathrm{loc}}}^+&:=\left\{\left(\L(s),u(s)\right)\colon s\in(0,\e)\right\}\subset\R\times H^2(\T),\\
				\mathscr{C}_{0,{\mathrm{loc}}}^-&:=\left\{\left(\L(s),u(s)\right)\colon s\in(-\e,0)\right\}\subset\R\times H^2(\T).
			\end{align*}
		
			\item[(ii)] {\rm\textbf{Uniqueness.}} There exists $\r >0$ such that, if $\mathfrak{F}(\l,u)=0$ and	$(\l,u)\in B_\r(0,0)\subset \R\times H^{2}(\mathbb{T})$, then either $u = 0$ or $(\l,u)=(\L(s),u(s))$ for some $s\in(-\e,\e)$. In other words,
			$$\mf{S}\cap B_{\rho}(0,0)=\mathscr{C}_{0,{\mathrm{loc}}}^+\uplus\mathscr{C}_{0,{\mathrm{loc}}}^-\uplus\{(\s_k,0)\}.$$
			Therefore exactly two branches of solutions of \eqref{eq:1.1} emanate from $(\s_0,0)$: one of positive and the other one of negative solutions. Moreover,
			$$\mathscr{C}_{0,{\mathrm{loc}}}^+=-\mathscr{C}_{0,{\mathrm{loc}}}^-,$$
			where, given a subset $\mathscr{C}\subset \R\times H^{2}(\mathbb{T})$, we are denoting $-\mathscr{C}:=\left\{(\l,-u) \colon (\l,u)\in\mathscr{C}\right\}$.
		\end{enumerate}
	\end{theorem}
			
	\begin{proof}
		Our aim is to apply the Crandall--Rabinowitz bifurcation Theorem \ref{Cr-Rb} to the operator $\mf{F}$ at $\s_{0}=0$. To do so, we observe first of all that $\mf{F}$ satisfies the assumptions (F1)--(F3) indicated in Appendix \ref{app:B.1}, thanks to the definition of $\mf{F}$, Lemma \ref{L3}, and Lemma \ref{L3.2}(ii). Moreover, by Proposition \ref{P3.5}, $\s_0$ is a $1$-transversal eigenvalue of $\mf{L}(\l):=\partial_{u}\mf{F}(\l,0)$, and $\chi[\mf{L}, \s_0]=1$. Consequently, applying Theorem \ref{Cr-Rb}, we deduce the existence and uniqueness of the curves \eqref{Cu} satisfying \eqref{Cu1}.
				
		We now observe that, if $(\l,u)$ is a solution of \eqref{eq:1.1}, then $(\l,-u)$ is also a solution of \eqref{eq:1.1}. Thus, the uniqueness of the curve that we have just obtained forces $(\L(-s),u(-s))=(\L(s),-u(s))$ for each $s\in (-\varepsilon,\varepsilon)$. In particular, as $u(s)=s\left(1+\Gamma(s)\right)$ and $\G(0)=0$, we deduce that, up to reducing $\e$ if necessary, for $s\in(0,\varepsilon)$, $u(s)$ is positive, i.e., $\mathscr{C}_{0,{\mathrm{loc}}}^+$ consists of positive solutions, while $\mathscr{C}_{0,{\mathrm{loc}}}^-$ consists of negative solutions. 
	\end{proof}
				
	Observe that, as a consequence of Proposition \ref{pr:2.4}(i), $\L(s)<\s_0=0$ for all $s\in(-\e,\e)\setminus\{0\}$, i.e. the bifurcation obtained in the previous theorem is subcritical. 
				
	\subsection{Local bifurcation of even and odd solutions for $\l=\s_k$, $k\geq 1$}
		\label{sec:3.2}
				
	In general, the local bifurcation analysis near $(\s_k,0)$ for $k\geq 1$ is much more involved than the case $(\s_0,0)$ because, as seen in Proposition \ref{P3.5}, $\chi[\mf{L},\s_{k}]=2$; thus the Crandall--Rabinowitz Theorem \ref{Cr-Rb} cannot be applied directly. This general case will be treated in Section \ref{sec:3.3}. Nevertheless, if we further assume that the weight $a\in L^{\infty}(\mathbb{T})$ is even, i.e., 
	\begin{equation}
		\label{even}
		a(T-t)=a(t) \quad \text{ for a.e. $t\in[0,T]$},
	\end{equation}
	and we restrict ourselves to the analysis of even or odd solutions of \eqref{eq:1.1}, these symmetries reduce the algebraic multiplicity $\chi[\mf{L},\s_{k}]$ to one, and we can still apply the Crandall--Rabinowitz theorem also near $(\s_{k},0)$, $k\geq 1$, as we detail below.
				
	Observe that weights satisfying \eqref{even} are referred to as \emph{even}, because if we extend them $T$-periodically in $\R$, we obtain even functions. In the rest of this subsection, we will assume condition \eqref{even} in addition to \eqref{ass:aloc}, without further reference.

	\vspace{5pt}
	\par \noindent \textbf{Even solutions.} We start the analysis with the case of even solutions. Throughout this section, for each $s\in\N\cup\{0\}$, we denote
	$$H^{s}_{\textbf{\texttt{ev}}}(\mathbb{T}):=\left\{u\in H^{s}(\mathbb{T}) \colon u(T-t)=u(t) \text{ for a.e. $t\in [0,T]$}\right\}.$$
	Note that $H^{s}_{\textbf{\texttt{ev}}}(\mathbb{T})$ is a closed subspace of $H^{s}(\mathbb{T})$.
	The key observation is that even solutions of \eqref{eq:1.1} can be rewritten as the zeros of the nonlinear operator
	$$\mf{F}_{\textbf{\texttt{ev}}}\colon\R\times H^{2}_{\textbf{\texttt{ev}}}(\mathbb{T})\longrightarrow L^{2}_{\textbf{\texttt{ev}}}(\mathbb{T}), \quad \mf{F}_{\textbf{\texttt{ev}}}(\l,u):=u''+\l u + a(t) u^{3},$$
	which is well-defined because, if $u\in H^{2}_{\textbf{\texttt{ev}}}(\mathbb{T})$, the Sobolev embedding $H^{2}(\mathbb{T})\hookrightarrow \mc{C}^{1}(\mathbb{T})$ and \eqref{even} guarantee that $u''$ and $a(t)u^{3}$ belong to $L^{2}_{\textbf{\texttt{ev}}}(\mathbb{T})$.
	
	A similar analysis to that of the beginning of this section allows us to prove that $\mf{F}_{\mathbf{\mathtt{ev}}}\in \mc{C}^{\infty}(\R\times H^{2}_{\mathbf{\mathtt{ev}}}(\mathbb{T}), L^{2}_{\mathbf{\mathtt{ev}}}(\mathbb{T}))$, that $\mf{F}_{\mathbf{\mathtt{ev}}}$ is proper on closed and bounded subsets of $\R\times H^{2}_{\mathbf{\mathtt{ev}}}(\mathbb{T})$, that the linearization of $\mf{F}_{\mathbf{\mathtt{ev}}}$ with respect to $u$ at $(\l,u)\in \R\times H^{2}_{\mathbf{\mathtt{ev}}}(\mathbb{T})$,
	$$\partial_{u}\mf{F}_{\mathbf{\mathtt{ev}}}(\l,u)\colon H^{2}_{\mathbf{\mathtt{ev}}}(\mathbb{T}) \longrightarrow L^{2}_{\mathbf{\mathtt{ev}}}(\mathbb{T}), \quad \partial_{u}\mf{F}_{\mathbf{\mathtt{ev}}}(\l,u)\left[v\right]=v''+\l v+3a(t)u^{2}v,$$
	is a Fredholm operator of index zero, and that the generalized spectrum of $\mf{L}_{\mathbf{\mathtt{ev}}}(\l):=\p_u\mf{F}(\l,0)$ 
	is given by
	\begin{equation}
		\label{E1.11}
		\Sigma(\mf{L}_{\mathbf{\mathtt{ev}}})=\left\{\s_k\right\}_{k\in\N\cup\{0\}}=\left\{\left(\frac{2\pi k}{T}\right)^{2} \colon k\in\N\cup \{0\}\right\}.
	\end{equation}
				
	The only substantial difference of this case lies in the computation of the generalized algebraic multiplicity when $k\in\N$. Indeed, by reasoning as in the proof of Lemma \ref{L3.2}(ii), and observing that the function $t\mapsto \sin\left(\frac{2\pi k }{T}t\right)$ is not even, we obtain, for all $k\in\N$, that 
	$$N[\mf{L}_{\mathbf{\mathtt{ev}}}\left(\s_{k}\right)]=\Span\left\{\varphi_k(t)\right\}, \qquad  R[\mf{L}_{\mathbf{\mathtt{ev}}}(\s_{k})]=\left\{u\in L^{2}_{\mathbf{\mathtt{ev}}}(\mathbb{T}) \colon \int_{0}^{T}\varphi_k(t)\, u(t) \, {\rm{d}}t = 0 \right\},$$
	where we have denoted
	$$\varphi_k(t):=\sqrt{\frac{2}{T}}\cos\left(\frac{2\pi k }{T}t\right), \qquad k\in\N. $$
	Then, by adapting the proof of Proposition \ref{P3.5}, we can show that $\s_k$ is $1$-transversal (cf., Definition \ref{de2.3}), and its generalized algebraic multiplicity is
	\begin{equation}
		\label{eq:algmult}
		\chi[\mf{L}_{\mathbf{\mathtt{ev}}},\s_{k}]=1, \quad k\in \N.
	\end{equation}
	If we define the set of  \textit{non-trivial even solutions} of $\mf{F}_{\mathbf{\mathtt{ev}}}$ as
	\begin{equation*}
		\mf{S}_{\mathbf{\mathtt{ev}}}:=\left( \mf{F}_{\mathbf{\mathtt{ev}}}^{-1}(0)\setminus \mc{T}\right)\cup \{(\l,0)\colon \l\in \Sigma(\mf{L}_{\mathbf{\mathtt{ev}}})\}\subset \R\times H^{2}_{\mathbf{\mathtt{ev}}}(\mathbb{T}),
	\end{equation*}
	which is a closed set in $\R\times H^{2}_{\mathbf{\mathtt{ev}}}(\mathbb{T})$ and also in $\R\times H^{2}(\mathbb{T})$ satisfying $\mf{S}_{\mathbf{\mathtt{ev}}}\subset \mf{S}$, we can apply the Crandall--Rabinowitz local bifurcation Theorem \ref{Cr-Rb} to the operator $\mf{F}_{\mathbf{\mathtt{ev}}}$ at $\s_{k}$, $k\in\N$, obtaining the following result about the local structure of even solution at the bifurcation points.
	
	\begin{theorem}
		\label{LCRWw}
		For each $k\in\N$, the point $(\s_{k},0)\in \R \times H^{2}_{\mathbf{\mathtt{ev}}}(\mathbb{T})$ is a bifurcation point of 
		$$\mf{F}_{\mathbf{\mathtt{ev}}}\colon\R\times H^{2}_{\mathbf{\mathtt{ev}}}(\mathbb{T})\longrightarrow L^{2}_{\mathbf{\mathtt{ev}}}(\mathbb{T}), \quad \mf{F}_{\mathbf{\mathtt{ev}}}(\l,u):=u''+\l u + a(t) u^{3},$$
		from the trivial branch $\mc{T}$ to a connected component $\mathscr{C}_{\mathbf{\mathtt{ev}}, k}\subset \mf{S}_{\mathbf{\mathtt{ev}}}$ of non-trivial even solutions of \eqref{eq:1.1}. Moreover, setting 
		$$Y_{k}:=\left\{u\in H^{2}_{\mathbf{\mathtt{ev}}}(\mathbb{T}) \colon \left(u,\varphi_{k}\right)_{H^{2}}=0 \right\},$$
		the following statements hold.
		\begin{enumerate}
			\item[(i)] {\rm\textbf{Existence.}} There exist $\e>0$ and two $\mc{C}^{\infty}$-maps
			\begin{equation}
				\L_k\colon (-\e,\e) \longrightarrow \R, \quad \L_k(0)=\s_k, \qquad \qquad \Gamma_k\colon (-\e,\e)\longrightarrow Y_{k}, \quad  \Gamma_k(0)=0,
			\end{equation}
			such that for each $s\in(-\e,\e)$,
			\begin{equation}
				\mathfrak{F}_{\mathbf{\mathtt{ev}}}(\L_k(s),u_k(s))=0, \quad \text{with $u_k(s):= s\left(\varphi_{k}+\Gamma_k(s)\right)$}.
				\label{Cu1even}
			\end{equation}
			We set
			\begin{align*}
			\mathscr{C}_{\mathbf{\mathtt{ev}}, k,\mathrm{loc}}^+&:=\left\{\left(\L_k(s),u_k(s)\right)\colon s\in(0,\e)\right\}\subset\R\times H^{2}_{\mathbf{\mathtt{ev}}}(\mathbb{T}),\\
			\mathscr{C}_{\mathbf{\mathtt{ev}}, k,\mathrm{loc}}^-&:=\left\{\left(\L_k(s),u_k(s)\right)\colon s\in(-\e,0)\right\}\subset\R\times H^{2}_{\mathbf{\mathtt{ev}}}(\mathbb{T}).
			\end{align*}

			\item[(ii)] {\rm\textbf{Uniqueness.}} There exists $\r >0$ such that, if $\mathfrak{F}_{\mathbf{\mathtt{ev}}}(\l,u)=0$ and
			$(\l,u)\in B_\r(\s_{k},0)\subset \R\times H^{2}_{\mathbf{\mathtt{ev}}}(\mathbb{T})$, then either $u = 0$ or 	$(\l,u)=(\L_k(s),u_k(s))$ for some $s\in(-\e,\e)$. In other words,
			$$\mf{S}_{\mathbf{\mathtt{ev}}}\cap B_{\rho}(\s_k,0)=\mathscr{C}_{\mathbf{\mathtt{ev}}, k,\mathrm{loc}}^+\uplus\mathscr{C}_{\mathbf{\mathtt{ev}}, k,\mathrm{loc}}^-\uplus\{(\s_k,0)\}.$$
			Therefore exactly two branches of even solutions of \eqref{eq:1.1} emanate from $(\s_k,0)$, both of them with $2k$ zeros in $(0,T)$. Moreover,
			$$\mathscr{C}_{\mathbf{\mathtt{ev}}, k,\mathrm{loc}}^+=-\mathscr{C}_{\mathbf{\mathtt{ev}}, k,\mathrm{loc}}^-.$$
		\end{enumerate}
	\end{theorem}
	The property on the number of zeros of the solutions on $\mathscr{C}_{\mathbf{\mathtt{ev}}, k,\mathrm{loc}}^\pm$ follows from \eqref{Cu1even}, because the eigenfunction $\varphi_k$ has exactly $2k$ simple zeros in $(0,T)$. Indeed, since $\G_k(0)=0$, $H^2(\T)\hookrightarrow \mc{C}^{1}(\T)$, and $s\mapsto\left\|u_k(s)\right\|_{\mc{C}^1}$ is continuous, this ensures that also $u_k(s)$ has $2k$ zeros in $(0,T)$, for $s\sim 0$.
	Observe that, from Proposition \ref{pr:2.4}(i) one gets that the obtained curves bifurcate subcritically, i.e., $\L_k(s)<\s_k$ for all $s\in(-\e,\e)\setminus\{0\}$.
	\begin{remark}
		\label{re:3.9}
		If one applies this analysis also to the case $k=0$, one can obtain a counterpart of Theorem \ref{LCRW}, similar to Theorem \ref{LCRWw}, in the case of even weights. By doing so, the main improvement is that one guarantees that the branches of positive and negative solutions on $\mathscr{C}_0$ found in Theorem \ref{LCRW} consist of even solutions of \eqref{eq:1.1}.
	\end{remark}
				
	\par \noindent \textbf{Odd solutions.} An analogous analysis can be done for odd solutions of \eqref{eq:1.1} by considering, for $s\in\N\cup\{0\}$, the space
	$$H^{s}_{\textbf{\texttt{odd}}}(\mathbb{T}):=\left\{u\in H^{s}(\mathbb{T}) \colon u(T-t)=-u(t) \text{ for a.e. $t\in [0,T]$}\right\}.$$
	Odd solutions of the boundary value problem \eqref{eq:1.1} can be rewritten as the zeros of the operator
	$$\mf{F}_{\textbf{\texttt{odd}}}\colon\R\times H^{2}_{\textbf{\texttt{odd}}}(\mathbb{T})\longrightarrow L^{2}_{\textbf{\texttt{odd}}}(\mathbb{T}), \quad \mf{F}_{\textbf{\texttt{odd}}}(\l,u):=u''+\l u + a(t) u^{3},$$
	which is well-defined because, if $u\in H^{2}_{\textbf{\texttt{odd}}}(\mathbb{T})$ and $a$ satisfies \eqref{even}, then $u''$ and $a(t)u^{3}$ belong to $L^{2}_{\textbf{\texttt{odd}}}(\mathbb{T})$. A similar analysis to the even case can be done for this operator, with the difference that, now, 
	$$\Sigma(\mf{L}_{\mathbf{\mathtt{odd}}})=\left\{\s_k\right\}_{k\in\N}=\left\{\left(\frac{2\pi k}{T}\right)^{2}\colon k\in\N\right\},$$
	i.e., $\s_{0}=0$ is no longer an eigenvalue of the linearization, because the corresponding eigenfunction $1\in H^{2}(\mathbb{T})$ is not odd. Moreover, for all $k\in\N$, 
	$$N[\mf{L}_{\mathbf{\mathtt{odd}}}\left(\s_{k}\right)]=\Span\left\{\phi_k(t)\right\}, \qquad  R[\mf{L}_{\mathbf{\mathtt{odd}}}(\s_{k})]=\left\{u\in L^{2}_{\mathbf{\mathtt{odd}}}(\mathbb{T}) \colon \int_{0}^{T}\phi_k(t)\, u(t) \, {\rm{d}}t = 0 \right\},$$
	where we have denoted
	$$\phi_k(t):=\sqrt{\frac{2}{T}}\sin\left(\frac{2\pi k }{T}t\right), \qquad k\in\N, $$
	and $\chi[\mf{L}_{\mathbf{\mathtt{odd}}},\s_{k}]=1$. The set of  \textit{non-trivial odd solutions} of $\mf{F}_{\mathbf{\mathtt{odd}}}$ is defined as
	\begin{equation*}
		\mf{S}_{\mathbf{\mathtt{odd}}}=\left( \mf{F}_{\mathbf{\mathtt{odd}}}^{-1}(0)\setminus \mc{T}\right)\cup \{(\l,0)\colon\l\in \Sigma(\mf{L}_{\mathbf{\mathtt{odd}}})\}\subset \R\times H^{2}_{\mathbf{\mathtt{odd}}}(\mathbb{T}),
	\end{equation*}
	and the following result holds.
				
	\begin{theorem}
		\label{LCRWwo}
		For each $k\in\N$, the point $(\s_{k},0)\in \R \times H^{2}_{\mathbf{\mathtt{odd}}}(\mathbb{T})$ is a bifurcation point of 
		$$\mf{F}_{\mathbf{\mathtt{odd}}}\colon\R\times H^{2}_{\mathbf{\mathtt{odd}}}(\mathbb{T})\longrightarrow L^{2}_{\mathbf{\mathtt{odd}}}(\mathbb{T}), \quad \mf{F}_{\mathbf{\mathtt{odd}}}(\l,u):=u''+\l u + a(t) u^{3},$$
		from the trivial branch $\mc{T}$ to a connected component $\mathscr{C}_{\mathbf{\mathtt{odd}}, k}\subset \mf{S}_{\mathbf{\mathtt{odd}}}$ of non-trivial odd solutions of \eqref{eq:1.1}. Moreover, setting
		$$Y_{k}:=\left\{u\in H^{2}_{\mathbf{\mathtt{odd}}}(\mathbb{T}) \colon \left(u,\phi_{k}\right)_{H^{2}}=0 \right\},$$
		the following statements hold.
		\begin{enumerate}
			\item[(i)] {\rm\textbf{Existence.}} There exist $\e>0$ and two $\mc{C}^{\infty}$-maps
			\begin{equation}
				\L_k\colon (-\e,\e) \longrightarrow \R, \quad \L_k(0)=\s_k, \qquad \qquad \Gamma_k\colon (-\e,\e)\longrightarrow Y_{k}, \quad  \Gamma_k(0)=0,
			\end{equation}
			such that for each $s\in(-\e,\e)$,
			\begin{equation}
				\mathfrak{F}_{\mathbf{\mathtt{odd}}}(\L_k(s),u_k(s))=0, \quad \text{with $u_k(s):= s\left(\phi_{k}+\Gamma_k(s)\right)$}.
			\end{equation}
			We set
			\begin{align*}
			\mathscr{C}_{\mathbf{\mathtt{odd}}, k,\mathrm{loc}}^+&:=\left\{\left(\L_k(s),u_k(s)\right)\colon s\in(0,\e)\right\}\subset\R\times H^{2}_{\mathbf{\mathtt{odd}}}(\mathbb{T}),\\
			\mathscr{C}_{\mathbf{\mathtt{odd}}, k,\mathrm{loc}}^-&:=\left\{\left(\L_k(s),u_k(s)\right)\colon s\in(-\e,0)\right\}\subset\R\times H^{2}_{\mathbf{\mathtt{odd}}}(\mathbb{T}).
			\end{align*}
			\item[(ii)] {\rm\textbf{Uniqueness.}} There exists $\r >0$ such that, if $\mathfrak{F}_{\mathbf{\mathtt{odd}}}(\l,u)=0$ and
			$(\l,u)\in B_\r(\s_{k},0)\subset \R\times H^{2}_{\mathbf{\mathtt{odd}}}(\mathbb{T})$, then either $u = 0$ or 	$(\l,u)=(\L_k(s),u_k(s))$ for some $s\in(-\e,\e)$. In other words,
			$$\mf{S}_{\mathbf{\mathtt{odd}}}\cap B_{\rho}(\s_k,0)=\mathscr{C}_{\mathbf{\mathtt{odd}}, k,\mathrm{loc}}^+\uplus\mathscr{C}_{\mathbf{\mathtt{odd}}, k,\mathrm{loc}}^-\uplus\{(\s_k,0)\}.$$
			Therefore exactly two branches of odd solutions of \eqref{eq:1.1} emanate from $(\s_k,0)$, both of them with $2k$ zeros in $[0,T)$. Moreover,
			$$\mathscr{C}_{\mathbf{\mathtt{odd}}, k,\mathrm{loc}}^+=-\mathscr{C}_{\mathbf{\mathtt{odd}}, k,\mathrm{loc}}^-.$$
		\end{enumerate}
	\end{theorem}
				
	As above, Proposition \ref{pr:2.4}(i) implies that $\L_k(s)<\s_k$ for all $s\in(-\e,\e)\setminus\{0\}$, i.e., that the bifurcation is subcritical.
				
	Gathering the results of Theorems \ref{LCRW}, \ref{LCRWw} and \ref{LCRWwo} and the spectral analysis at the beginning of Section \ref{section:3}, we have proved the following result related to \eqref{eq:1.1} with even weights.
	
	\begin{theorem}
		\label{CR} 
		Suppose that $a$ satisfies \eqref{ass:aloc} and \eqref{even}, i.e., it is an even function. Then:
		\begin{itemize}
			\item[(i)] the unique bifurcation points of $$\mf{F}\colon\R\times H^{2}(\mathbb{T})\longrightarrow L^{2}(\mathbb{T}), \quad \mf{F}(\l,u):=u''+\l u + a(t) u^{3},$$
			from the branch of trivial solutions $\mc{T}$ are $(\s_{k},0)\subset\R\times H^{2}(\mathbb{T})$, with $k\in\N\cup\{0\}$;
			\item[(ii)] a connected component of non-trivial solutions $\mathscr{C}_{k}\subset \mf{S}$ emanates at $(\s_{k},0)$ for all $k\in\N\cup\{0\}$;
			\item[(iii)] there exists $\rho>0$ such that
			$$B_{\rho}(\s_{0},0)\cap \left(\mathscr{C}_{0}\setminus\{(\s_{0},0)\}\right)=\mathscr{C}_{0,\mathrm{loc}}^{+}\uplus \mathscr{C}_{0,\mathrm{loc}}^{-}$$ where $\mathscr{C}^{+}_{0,\mathrm{loc}}$ (resp. $\mathscr{C}^{-}_{0,\mathrm{loc}}$) consists of positive (resp. negative) even solutions of \eqref{eq:1.1}. Moreover, $-\mathscr{C}^{+}_{0,\mathrm{loc}}=\mathscr{C}^{-}_{0,\mathrm{loc}}$;
			\item[(iv)] for each $k\in \N$, the component $\mathscr{C}_{k}$ contains two subcomponents $$\mathscr{C}_{\mathbf{\mathtt{ev}},k}\subset \R\times H^{2}_{\mathbf{\mathtt{ev}}}(\mathbb{T}), \quad \mathscr{C}_{\mathbf{\mathtt{odd}},k}\subset \R\times H^{2}_{\mathbf{\mathtt{odd}}}(\mathbb{T}),$$
			of even and odd solutions of \eqref{eq:1.1}, respectively. Moreover, there exists $\rho>0$ such that 
			\begin{align*}
				B_{\rho}(\s_{k},0)\cap \left(\mathscr{C}_{\mathbf{\mathtt{ev}}, k}\setminus\{(\s_{k},0)\}\right)&=\mathscr{C}_{\mathbf{\mathtt{ev}}, k,\mathrm{loc}}^{+}\uplus \mathscr{C}_{\mathbf{\mathtt{ev}}, k,\mathrm{loc}}^{-}, \\ B_{\rho}(\s_{k},0)\cap \left(\mathscr{C}_{\mathbf{\mathtt{odd}},k}\setminus\{(\s_{k},0)\}\right)&=\mathscr{C}_{\mathbf{\mathtt{odd}},k,\mathrm{loc}}^{+}\uplus \mathscr{C}_{\mathbf{\mathtt{odd}},k,\mathrm{loc}}^{-},
			\end{align*}
			where $\mathscr{C}_{\mathbf{\mathtt{ev}},k,\mathrm{loc}}^{\pm}$ (resp. $\mathscr{C}_{\mathbf{\mathtt{odd}},k,\mathrm{loc}}^{\pm}$) consist of even (resp. odd) solutions of \eqref{eq:1.1} with exactly $2k$ zeros in $[0,T)$. Finally,
			$-\mathscr{C}_{\mathbf{\mathtt{ev}},k,\mathrm{loc}}^{+}=\mathscr{C}_{\mathbf{\mathtt{ev}},k,\mathrm{loc}}^{-}$ and $-\mathscr{C}_{\mathbf{\mathtt{odd}},k,\mathrm{loc}}^{+}= \mathscr{C}_{\mathbf{\mathtt{odd}},k,\mathrm{loc}}^{-}$.
		\end{itemize}
	\end{theorem}

	\subsection{Local bifurcation for $\l=\s_k$, $k\geq 1$}
	\label{sec:3.3}
	We now analyze the local bifurcation of solutions of \eqref{eq:1.1} near $(\s_k,0)$, $k\geq 1$, in the general case, i.e., without any further assumption on the weight $a(t)$ apart from \eqref{ass:aloc}. This analysis is much more involved than in the cases analized in Sections \ref{sec:3.1} and \ref{sec:3.2}, because the Crandall--Rabinowitz bifurcation theorem cannot be applied. Here, instead, we will perform the following Lyapunov--Schmidt reduction.
	\par For $k\in \N$, consider the orthonormal projections 
	\begin{align}
		P_{k}&\colon H^{2}(\mathbb{T})\longrightarrow N[\mf{L}(\s_{k})], & P_{k}[u]&:=\left(u,\tilde \varphi_{k}\right)_{H^{2}} \tilde\varphi_{k} + (u,\tilde\phi_{k})_{H^{2}}\, \tilde\phi_{k}, \\
		Q_{k}&\colon L^{2}(\mathbb{T})  \longrightarrow R[\mf{L}(\s_{k})]^{\perp}, & Q_{k}[v]&:=\left(v,\varphi_{k}\right)_{L^{2}} \varphi_{k} + \left(v, \phi_{k}\right)_{L^{2}} \phi_{k}, \label{eq:def_Q_k}
	\end{align}
	where we recall that $(\cdot,\cdot)_{L^2}$ and $(\cdot,\cdot)_{H^2}$ denote the corresponding scalar products in $L^{2}(\mathbb{T})$ and $H^{2}(\mathbb{T})$, respectively,
	\begin{equation*}
		\varphi_{k}(t)=\sqrt{\frac{2}{T}}\cos\left(\frac{2\pi k}{T}t\right), \qquad  \phi_{k}(t)=\sqrt{\frac{2}{T}}\sin\left(\frac{2\pi k}{T}t\right),
	\end{equation*}
	and we have set
	$$\tilde\varphi_{k}(t):=\left(1+\frac{4\pi^2k^2}{T^2}+\frac{16\pi^4k^4}{T^4}\right)^{-\frac{1}{2}}\varphi_{k}(t), \qquad \tilde\phi_{k}(t):=\left(1+\frac{4\pi^2k^2}{T^2}+\frac{16\pi^4k^4}{T^4}\right)^{-\frac{1}{2}}\phi_{k}(t), $$
	so that
	$$\left\|\varphi_k\right\|_{L^2}=\left\|\phi_k\right\|_{L^2}=\left\|\tilde\varphi_k\right\|_{H^2}=\big\|\tilde\phi_k\big\|_{H^2}=1.$$
	Therefore, $\mf{F}(\l,u)=0$ can be equivalently written as
	\begin{equation}
		\label{E10}
		\left\{\begin{array}{l}
			Q_{k}\mf{F}\left(\l,P_{k}[u]+\left(I_{H^{2}}-P_{k}\right)[u]\right)=0, \\[2pt]
			\left(I_{L^{2}}-Q_{k}\right)\mf{F}\big(\l,P_{k}[u]+\left(I_{H^{2}}-P_{k}\right)[u]\big)=0,
		\end{array}\right.
	\end{equation}
	where $I_{L^{2}}\colon L^{2}(\mathbb{T})\to L^{2}(\mathbb{T})$ and $I_{H^{2}}\colon H^{2}(\mathbb{T})\to H^{2}(\mathbb{T})$ are the identity operators. Consider the auxiliary function 
	$$\mathscr{G}_{k}\colon\R\times N[\mf{L}(\s_{k})]\times N[\mf{L}(\s_{k})]^{\perp} \longrightarrow R[\mf{L}(\s_{k})], \qquad \mathscr{G}_{k}(\l,x,y):=\left(I_{L^{2}}-Q_{k}\right)\mf{F}(\l,x+y);$$
	since $\mf{F}$ is of class $\mc{C}^{\infty}$, $\mathscr{G}_{k}$ is also of class $\mc{C}^{\infty}$, and
	\begin{equation*}
		\partial_{y}\mathscr{G}_{k}(\l,x,y)[z]=\left(I_{L^{2}}-Q_{k}\right)\partial_{u}\mf{F}(\l,x+y)[z], \qquad z\in N[\mf{L}(\s_{k})]^{\perp}.
	\end{equation*}
	Note that 
	\begin{align*}
	\mathscr{G}_{k}(\s_{k},0,0)&=\left(I_{L^{2}}-Q_{k}\right)\mf{F}(\s_{k},0)=0, \\ \partial_{y}\mathscr{G}_{k}(\s_{k},0,0)&=\left(I_{L^{2}}-Q_{k}\right)\partial_{u}\mf{F}(\s_{k},0)|_{N[\mf{L}(\s_{k})]^{\perp}}=\left(I_{L^{2}}-Q_{k}\right)\mf{L}(\s_k)|_{N[\mf{L}(\s_{k})]^{\perp}}.
	\end{align*}
	Hence, $\partial_{y}\mathscr{G}_{k}(\s_{k},0,0)$ is an isomorphism, and the implicit function theorem then guarantees the existence of two open neighborhoods
	$$\s_{k}\in I_{k}\subset \R, \quad 0\in \mc{N}_{k}\subset N[\mf{L}(\s_{k})],$$
	and of a function $\psi\colon I_{k}\times \mc{N}_{k}\longrightarrow N[\mf{L}(\s_{k})]^{\perp}$ of class $\mc{C}^{\infty}$ such that $\psi(\s_{k},0)=0$ and 
	\begin{equation}
		\label{E12}
		\mathscr{G}_{k}(\l,x,\psi(\l,x))=\left(I_{L^{2}}-Q_{k}\right)\mf{F}(\l,x+\psi(\l,x))=0, \qquad (\l,x)\in I_{k}\times \mc{N}_{k}.
	\end{equation}
	Moreover, there exists a neighborhood $(\s_{k},0)\in \mc{O}_{k}\subset \R\times H^{2}(\mathbb{T})$ such that, if $(\l,u)=(\l,x+y)\in\mc{O}_{k}$ and $\mathscr{G}_{k}(\l,x,y)=0$, then $y=\psi(\l,x)$. Therefore, system \eqref{E10} can be equivalently rewritten as the finite-dimensional equation
	\begin{equation}
		\label{E11}
		Q_{k}\mf{F}(\l,x+\psi(\l,x))=0, \qquad x\in \mc{N}_{k}\subset N[\mf{L}(\s_{k})], \,\, \l\in I_{k}\subset\R.
	\end{equation}
	Let us consider the linear isomorphisms
	\begin{equation}
		\label{eq:TS}
		\begin{aligned}
			L&\colon\R^{2}\longrightarrow N[\mf{L}(\s_{k})], & L(x_1,x_2)&:=x_1\varphi_{k}+x_2\phi_{k}, \\
			S&\colon R[\mf{L}(\s_{k})]^{\perp}\longrightarrow \R^{2}, & S(z)&:=(\left(z,\varphi_{k}\right)_{L^{2}},\left(z,\phi_{k}\right)_{L^{2}}),
		\end{aligned}
	\end{equation}
	and $\O_{k}:=L^{-1}(\mc{N}_{k})\subset\R^{2}$. With these isomorphisms, equation \eqref{E11} is equivalent to
	\begin{equation*}
		\mf{H}(\l,\left(x_1,x_2\right)):=SQ_{k}\mf{F}(\l,L(x_1,x_2)+\psi(\l,L(x_1,x_2)))=0, \qquad (x_1,x_2)\in\O_{k}, \, \, \l\in I_{k}.
	\end{equation*}
	In the sequel, with a small abuse of notation, we will write $(\l,x_1,x_2)$ instead of $(\l,\left(x_1,x_2\right))$.
	Observe that the function $\mf{H}\colon I_{k}\times \O_{k}\subset \R\times \R^{2}\longrightarrow \R^{2}$ is explicitly given by
	\begin{equation}
	\label{E16}
		\mf{H}(\l,x_1,x_2):=\left[ \begin{array}{l}
			h_{1}(\l,x_1,x_2) \\
			h_{2}(\l,x_1,x_2)
		\end{array}\right]=\left[ \begin{array}{l}
			\left(Q_{k}\mf{F}(\l,L(x_1,x_2)+\psi(\l,L(x_1,x_2))),\varphi_{k}\right)_{L^{2}} \\
			\left(Q_{k}\mf{F}(\l,L(x_1,x_2)+\psi(\l,L(x_1,x_2))),\phi_{k}\right)_{L^{2}}
			\end{array}\right].
	\end{equation}
	Consequently, we have established the following result.
	\begin{theorem}
		\label{Th3.13}
		The maps 
		\begin{align*}
			\Psi_{k} &\colon \mf{H}^{-1}(0) \longrightarrow \mf{F}^{-1}(0)\cap \mc{O}_{k}, & (\l,x_1,x_2)&\mapsto (\l,L(x_1,x_2)+\psi(\l,L(x_1,x_2))), \\
			\Psi_{k}^{-1}&\colon \mf{F}^{-1}(0)\cap \mc{O}_{k} \longrightarrow \mf{H}^{-1}(0), & (\l,u)&\mapsto (\l,L^{-1}P_{k}\left[u\right]),
		\end{align*}
		are inverse one of the other. Therefore, $\Psi_{k}\colon \mf{H}^{-1}(0) \to \mf{F}^{-1}(0)\cap \mc{O}_{k}$ is a homeomorphism. Moreover, for every $(\l,x_{1},x_{2})\in \mf{H}^{-1}(0)$, it holds that $\partial_{(x_{1},x_{2})}\mf{H}(\l,x_{1},x_{2})\in GL(\R^{2})$ if and only if $\partial_{u}\mf{F}(\Psi_{k}(\l,x_{1},x_{2}))\in GL(H^{2}(\mathbb{T}),L^{2}(\mathbb{T}))$.
	\end{theorem}
	\begin{proof}
		The first part follows from the construction carried out in this subsection. For the second part, we refer the reader to Proposition 8.3.4, statement (a), of \cite{BT}.
	\end{proof}
	In conclusion, to determine the local structure of $\mf{F}^{-1}(0)$ near $(\s_{k},0)$, $k\geq 1$, we must solve the equivalent finite-dimensional system $\mf{H}(\l,x_1,x_2)=0$ defined in \eqref{E16} with $(\l,x_1,x_2)\in I_k\times\O_k$. 
	\par
	The next result gives the Taylor expansion of the function $\mf{H}\colon I_{k}\times \O_{k}\subset \R\times \R^{2}\to\R^{2}$. Its proof is based on the computation of the derivatives of the functions $\psi$, $h_{1}$ and $h_{2}$ that are detailed in Appendix \ref{AC}.
				
	\begin{theorem}
		\label{th:4.3}
		The Taylor expansions of $h_1$ and $h_2$ at $(\l,x_1,x_2)=(\s_{k},0,0)$ are given by
		\begin{equation*}\begin{array}{l}
			h_1(\l,x_1,x_2)  =  \left(\l-\s_k\right)x_1 +
			a_k x_1^3+3b_kx_1^2x_2+3c_kx_1x_2^2+d_kx_2^3 +R_1(\l,x_1,x_2), \\[3pt]
			h_2(\l,x_1,x_2)  =  \left(\l-\s_k\right)x_2
			+
			 b_kx_1^3+3c_kx_1^2x_2+3d_kx_1x_2^2+e_kx_2^3+R_2(\l,x_1,x_2),\end{array}
		\end{equation*}
		where we have set
		\begin{equation}
			\label{eq:coeff}
			\begin{split}
				a_k\!:=\left(a(t)\varphi_k^3,\varphi_k\right)_{L^2}, \;\; c_k\!:=\left(a(t) \varphi_k \phi^2_k,\varphi_k\right)_{L^2}\!=\!\left(a(t) \varphi^2_k \phi_k,\phi_k\right)_{L^2}, \;\;  e_k\!:= \left(a(t)\phi_k^3,\phi_k\right)_{L^2},\\
				b_k\!:=\left(a(t) \varphi^2_k \phi_k,\varphi_k\right)_{L^2}\!=\!\left(a(t)\varphi_k^3,\phi_k\right)_{L^2}, \quad
				d_k\!:=\left(a(t)\phi_k^3,\varphi_k\right)_{L^2}\!=\!\left(a(t) \varphi_k \phi^2_k,\phi_k\right)_{L^2},
			\end{split}
		\end{equation}
		and,  taking any fixed norm $\|\cdot\|$ in $\R^2$, the remainders satisfy
		\[
		R_1(\l,x_1,x_2),R_2(\l,x_1,x_2)\in
		O\Bigg(\sum_{j=0}^{4}\left|\l-\s_k\right|^{j}\|(x_1,x_2)\|^{4-j}\Bigg).
		\]
	\end{theorem}
				
	Thanks to this analysis, we can prove the following local structure result, establishing that, near the bifurcation point, every curve of solutions of \eqref{eq:1.1} bifurcating from $(\s_{k},0)$, $k\geq 1$, consists of solutions with $2k$ zeros in $[0,T)$. 
				
	\begin{proposition}
		\label{pr:locbif}
		Let $\varepsilon>0$ and suppose that $\gamma\colon(-\varepsilon,\varepsilon)\to \mf{H}^{-1}(0)$, $\gamma(s)=(\l(s),\gamma_{1}(s),\gamma_{2}(s))$, is a continuous curve of solutions of $\mf{H}(\l,x_1,x_2)=0$ such that $\gamma(0)=(\s_{k},0,0)$, $k\in\N$. Then, the corresponding solutions of $\mf{F}(\l,u)=0$ given by the homeomorphism $\Psi_{k}\colon \mf{H}^{-1}(0) \to \mf{F}^{-1}(0)\cap \mc{O}_{k}$ admit the following asymptotic expansion as $s\to 0$ 
		\begin{equation}
			\label{E3.22}
			\begin{aligned}
			\!\!	(\l,u(t)) & \!=\!  \Big(\l(s),\gamma_{1}(s)\varphi_{k}(t)+\gamma_{2}(s)\phi_{k}(t)+O\Big(\sum_{j=0}^{3}|\l(s)-\s_k|^{3-j}\left(\left|\g_1(s)\right|+\left|\g_2(s)\right|\right)^{j}\Big)\Big)\\
			\!\!	& \!=\!  \Bigg(\l(s),A(s) \cos\left(\frac{2\pi k }{T}t-\d(s)\right)+O\Big(\sum_{j=0}^{3}|\l(s)-\s_k|^{3-j}\left(\left|\g_1(s)\right|+\left|\g_2(s)\right|\right)^{j}\Big)\Bigg)
			\end{aligned}
		\end{equation}
		where 
		$$A(s):= \sqrt{\frac{2}{T}\left(\gamma_{1}^{2}(s)+\gamma_{2}^{2}(s)\right)}, \quad \d(s):=\arccos\left(\tfrac{\gamma_{1}(s)}{\sqrt{\gamma_{1}^{2}(s)+\gamma_{2}^{2}(s)}}\right), \quad s\in(-\varepsilon,\varepsilon).$$
		Therefore, the corresponding solutions of $\mf{F}(\l,u)=0$ in a neighborhood of $(\s_k,0)$, given by $\Psi_{k}\circ \gamma\colon(-\varepsilon,\varepsilon)\to \mf{F}^{-1}(0)\cap \mc{O}_{k}$, have exactly $2k$ zeros in $[0,T)$ for sufficiently small $\varepsilon>0$.
	\end{proposition}
				
	\begin{proof}
		Fix $k\in\N$. The explicit expressions of $\Psi_{k}$ in Theorem \ref{Th3.13} and of $L$ in \eqref{eq:TS} lead to
		\begin{align*}
			(\Psi_{k}\circ\gamma)(s)&=\Psi_{k}(\l(s), \gamma_{1}(s), \gamma_{2}(s))\\
			& =\left(\l(s), L(\gamma_{1}(s), \gamma_{2}(s))+\psi(\l(s), L(\gamma_{1}(s), \gamma_{2}(s)))\right) \\
			&= \left(\l(s),\gamma_{1}(s)\varphi_{k}(t)+\gamma_{2}(s)\phi_{k}(t)+\psi(\l(s), L(\gamma_{1}(s), \gamma_{2}(s)))\right).
		\end{align*}
		Expansion \eqref{E3.22} follows because the derivatives of $\psi$ at $(\s_k,0)$ computed in Lemma \ref{LC1} ensure, as $s\to 0$,
		$$\psi(\l(s), \gamma_{1}(s)\varphi_k(t)+ \gamma_{2}(s)\phi_k(t))=O\Big(\sum_{j=0}^{3}|\l(s)-\s_k|^{3-j}\left(\left|\g_1(s)\right|\left\|\varphi_k\right\|_{\mc{C}^1}+\left|\g_2(s)\right|\left\|\phi_k\right\|_{\mc{C}^1}\right)^{j}\Big).$$
		Finally, the statement about the number of zeros of the solutions for $s\sim 0$ follows by continuity with respect to $s$ in the $\mc{C}^1$ norm, observing that the function $\cos\left(\frac{2\pi k }{T}t-\d(s)\right)$ has exactly $2k$ transversal zeros in $[0,T)$ for all $s\in(-\e,\e)$.
	\end{proof}
				
	\subsection{Study of the local bifurcation in some particular cases}
	\label{sec:3.4}
				
	In this section, we further investigate the sharp local structure of the solutions of \eqref{eq:1.1} at the bifurcation points $(\s_{k},0)$, $k\geq 1$, obtained in Section \ref{sec:3.3}, by assuming some particular relations on the coefficients introduced in \eqref{eq:coeff}. Observe first of all that, since we are assuming \eqref{ass:aloc}, with the notation of \eqref{eq:coeff} we have
	\[
	a_k>0, \qquad c_k>0, \qquad e_k>0.
	\]
	Moreover, throughout this section we suppose 
	\begin{equation}
		\label{ass:H}
		(i)\; a_k=3c_k, \qquad (ii)\; b_k=d_k, \qquad (iii)\; a_k=e_k, \qquad (iv)\; b_k\neq 0. \tag{H}
	\end{equation}
	Observe that $(iv)$ in \eqref{ass:H} is incompatible with the possibility of $a(t)$ being even considered in Section \ref{sec:3.2}. Indeed, if $a(t)$ is even, the integrand in the definition of $b_k$ (see \eqref{eq:coeff}) is odd, implying $b_k=0$. Moreover, from assumptions $(i)$, $(ii)$ and $(iii)$ in \eqref{ass:H}, which imply
	\begin{gather*}
		\int_0^{T}a(t)\cos^4\left(\sqrt{\s_k}t\right)\,{\mathrm{d}}t=\int_0^{T}a(t)3\cos^2\left(\sqrt{\s_k}t\right)\sin^2\left(\sqrt{\s_k}t\right)\,{\mathrm{d}}t= \int_0^{T}a(t)\sin^4\left(\sqrt{\s_k}t\right)\,{\mathrm{d}}t,\\
		\int_0^{T}a(t)\cos^3\left(\sqrt{\s_k}t\right)\sin\left(\sqrt{\s_k}t\right)\,{\mathrm{d}}t=\int_0^{T}a(t)\sin^3\left(\sqrt{\s_k}t\right)\cos\left(\sqrt{\s_k}t\right)\,{\mathrm{d}}t,
	\end{gather*}
	we obtain
	\begin{equation}
		\label{subcrit}
		a_k\pm2b_k>0.
	\end{equation} 
	Indeed, 
	\begin{align*}
	0&<\int_0^{T}a(t)\left(\cos\left(\sqrt{\s_k}t\right)\pm\sin\left(\sqrt{\s_k}t\right)\right)^4\,{\mathrm{d}}t\\
	&=\int_0^{T}a(t)\left(\cos^4\left(\sqrt{\s_k}t\right)\pm4\cos^3\left(\sqrt{\s_k}t\right)\sin\left(\sqrt{\s_k}t\right)\right)\,{\mathrm{d}}t\\
	&\;\;\;+\int_0^{T}a(t)\left(6\cos^2\left(\sqrt{\s_k}t\right)\sin^2\left(\sqrt{\s_k}t\right)\pm4\cos\left(\sqrt{\s_k}t\right)\sin^3\left(\sqrt{\s_k}t\right)+\sin^4\left(\sqrt{\s_k}t\right)\right)\,{\mathrm{d}}t\\
	&=\int_0^{T}a(t)\left(\cos^4\left(\sqrt{\s_k}t\right)\pm4\cos^3\left(\sqrt{\s_k}t\right)\sin\left(\sqrt{\s_k}t\right)\right)\,{\mathrm{d}}t\\
	&\;\;\;+\int_0^{T}a(t)\left(2\cos^4\left(\sqrt{\s_k}t\right)\pm4\cos^3\left(\sqrt{\s_k}t\right)\sin\left(\sqrt{\s_k}t\right)+\cos^4\left(\sqrt{\s_k}t\right)\right)\,{\mathrm{d}}t\\
	&=T^2\left(a_k\pm2b_k\right).
	\end{align*}
	
	Using \eqref{ass:H}, \eqref{E16} and Theorem \ref{th:4.3}, the relation $\mf{H}(\l,x,y)=0$, which, thanks to Theorem \ref{Th3.13}, is equivalent to $\mf{F}(\l,u)=0$ in a neighborhood of $(\s_k,0)$, $k\geq 1$, can be written as
	\begin{equation}
		\label{Si1}
		\left(\l-\s_{k}\right)I_{\R^2}(x,y)+\mathscr{C}_k(x,y)+
		O\Bigg(\sum_{j=0}^{4}\left|\l-\s_k\right|^{j}\|(x,y)\|^{4-j}\Bigg)=0,
	\end{equation}
	where $I_{\R^2}\colon \R^{2}\to\R^{2}$ is the identity operator and $\mathscr{C}_k\colon \R^{2}\to\R^{2}$ is the homogeneous cubic 
	\begin{equation}
		\label{eq:cubsimpl}
		\mathscr{C}_k(x,y):=\left(
		\begin{array}{l}
			a_k x^{3}+3b_kx^{2}y+a_kxy^{2}+b_ky^{3} \\
			b_k x^{3}+a_kx^{2}y+3b_kx y^{2}+a_k y^{3}
		\end{array}
		\right).
	\end{equation}
	Observe that, thanks to Proposition \ref{pr:2.4}(i), we may restrict to the case $\l<\s_k$, condition that we assume in the rest of this section.
				
	The aim of the rest of this section is the study of the structure of the solutions of system \eqref{Si1} near $(\l,x,y)=(\s_{k},0,0)$. A key point will be showing that, under conditions \eqref{ass:H}, the zero set \eqref{Si1} is, near $(\s_{k},0,0)$, qualitatively the same if we truncate such a Taylor expansion at third-order terms. With this goal in mind, we establish the following preliminary result.
	
	\begin{lemma}
		\label{L20}
		If $\mathscr{C}_k(x,y)=(0,0)$, then $(x,y)=(0,0)$.
	\end{lemma} 
	
	\begin{proof}
		Thanks to \eqref{ass:H}, the condition $\mathscr{C}_k(x,y)=(0,0)$
		can be rewritten as
		\begin{equation}
			\label{E1}
			\left\{
			\begin{array}{l}
				x\left(a_k x^{2}+3b_k x y+a_k y^{2}\right)=-b_k y^{3}, \\
				y\left(a_k x^{2}+3b_k x y+a_k y^{2}\right)=-b_k x^{3}.
			\end{array}
			\right.
		\end{equation}
		If $a_k x^{2}+3b_kxy+a_ky^{2}=0$, then the first equation gives $y=0$, and the second one $x=0$. Similarly, if $x=0$, the first equation gives $y=0$, while, if $y=0$, the second equation gives $x=0$.
						
		Thus, the only remaining situation is when $x\neq 0$, $y\neq 0$ and  $a_k x^{2}+3b_kxy+a_ky^{2}\neq0$. In such a case, dividing the two equations of \eqref{E1}, we obtain $x^{4}=y^{4}$, that is, $y=x$ or $y=-x$. If $y=x$, system \eqref{E1} reduces to $2(a_k+2b_k)x^{3}=0$, which is impossible because of \eqref{subcrit}. Similarly, if $y=-x$, system \eqref{E1} reduces to $2(a_k-2b_k)x^{3}=0$, and we conclude again by \eqref{subcrit}.
	\end{proof}
					
	\begin{lemma}
		\label{L31}
		There exists a neighborhood $\mc{V}$ of $(\l,x,y)=(\s_{k},0,0)$ and a constant $\beta>0$ such that any solution of system \eqref{Si1} in $\mc{V}$ with $\l<\s_k$ must satisfy
		$$\|\mathbf{x}\|\leq \beta \sqrt{\s_{k}-\l},\qquad \text{where } \mathbf{x}:=(x,y).$$
	\end{lemma}
	
	\begin{proof}
		Assume by contradiction that there exists a sequence of solutions $\{(\l_{n},x_{n},y_{n})\}_{n\in\N}$ of \eqref{Si1} such that $\l_n<\s_k$, $(\l_{n},x_{n},y_{n})\to (\s_{k},0,0)$ as $n\to+\infty$, and 
		\begin{equation}
			\label{eq:contr}
		\lim_{n\to+\infty}\frac{\sqrt{\s_{k}-\l_{n}}}{\|\mathbf{x}_{n}\|}=0, \quad \mathbf{x}_{n}=(x_{n},y_{n}).
	\end{equation}
		Divide equation \eqref{Si1} by $\|\mathbf{x}_{n}\|^{3}$ to obtain
		$$\frac{\l_{n}-\s_{k}}{\|\mathbf{x}_{n}\|^2}\frac{\mathbf{x}_{n}}{\|\mathbf{x}_{n}\|}+\mathscr{C}_k\left(\frac{\mathbf{x}_{n}}{\|\mathbf{x}_{n}\|}\right)+O\Bigg(\sum_{j=0}^{4}\left|\l_n-\s_k\right|^{j}\|\mathbf{x}_{n}\|^{1-j}\Bigg)=0.$$
		This is equivalent to 
		\begin{equation}
			\left(\frac{\sqrt{\s_{k}-\l_{n}}}{\|\mathbf{x}_{n}\|}\right)^{2}\frac{\mathbf{x}_{n}}{\|\mathbf{x}_{n}\|}+\mathscr{C}_k\left(\frac{\mathbf{x}_{n}}{\|\mathbf{x}_{n}\|}\right)+O\Bigg(\sum_{j=0}^{4}\left(\frac{\sqrt{\s_{k}-\l_{n}}}{\|\mathbf{x}_{n}\|}\right)^{2j}\|\mathbf{x}_{n}\|^{1+j}\Bigg)=0.
		\end{equation}
		Since $\tfrac{\mathbf{x}_{n}}{\|\mathbf{x}_{n}\|}$ has norm one, we may assume, up to taking a subsequence that we do not relabel, that $\tfrac{\mathbf{x}_{n}}{\|\mathbf{x}_{n}\|}\to \mathbf{v}$ as $n\to+\infty$ with $\|\mathbf{v}\|=1$. Therefore, using \eqref{eq:contr} and letting $n\to+\infty$ in the previous relation, we obtain $\mathscr{C}(\mathbf{v})=0$. By Lemma \ref{L20}, we have $\mathbf{v}=0$, contradicting the fact that $\|\mathbf{v}\|=1$.
	\end{proof}
	\noindent Owing to Lemma \ref{L31}, we can rewrite equation \eqref{Si1} as 
	\begin{equation}
		\label{Si}
		\left(\l-\s_{k}\right)I_{\R^2}(x,y)+\mathscr{C}_k(x,y)+O\left(|\l-\s_{k}|^{2}\right)=0.
	\end{equation}
	From Lemma \ref{L20}, we know that, if $\l=\s_{k}$, the unique solution of equation \eqref{Si} is $(x,y)=(0,0)$. For $\l< \s_{k}$, we perform the invertible change of variables 
	\begin{equation}
		\label{CV}
		(x,y)=\sqrt{\s_{k}-\l}\left(z,w\right)
	\end{equation}
	in equation \eqref{Si}, to obtain
	\begin{align*}
		(\l-\s_{k})\sqrt{\s_{k}-\l}I_{\R^2}(z,w)+\left(\s_{k}-\l\right)^{\frac{3}{2}}\mathscr{C}_k(z,w)+ O\left(|\l-\s_{k}|^{2}\right)=0,
		\end{align*}
	which is equivalent to 
	\begin{equation}
		\label{eq:f}
		f(\l,z,w):=-I_{\R^2}(z,w)+\mathscr{C}_k(z,w)+ O\left(|\l-\s_{k}|^{\frac{1}{2}}\right)=0, \quad (\l,z,w)\in \mc{V}, \quad \l<\s_k.
	\end{equation}
	Since we are interested in solutions of \eqref{eq:f} for $\l< \s_{k}$ with $\l\sim \s_{k}$, our strategy will be applying the implicit function theorem to the solutions of $f(\s_{k},z,w)=0$, which can be rewritten as
	\begin{equation}
		\label{EE_new}
		\left\{
		\begin{array}{l}
			z\left(a_k z^{2}+3b_k zw+a_k w^{2}-1\right)=-b_k w^{3}, \\
			w\left(a_k z^{2}+3b_k zw+a_k w^{2}-1\right)=-b_k z^{3}.
		\end{array}
		\right.
	\end{equation}
	By recalling \eqref{ass:H} and \eqref{subcrit}, and adapting the proof of Lemma \ref{L20}, it is easy to see that the non-trivial solutions of this system satisfy $z=\pm w$ and are given by
	\begin{equation}
		\begin{aligned}
			\label{eq:sol}
			(\bar z_{1,k},\bar w_{1,k})&=\left(\frac{1}{\sqrt{2\left(a_k+2b_k\right)}},\frac{1}{\sqrt{2\left(a_k+2b_k\right)}}\right), & (\bar z_{2,k},\bar w_{2,k})&=(-\bar z_{1,k},-\bar w_{1,k}), \\ (\bar z_{3,k},\bar w_{3,k})&=\left(\frac{1}{\sqrt{2\left(a_k-2b_k\right)}},\frac{-1}{\sqrt{2\left(a_k-2b_k\right)}}\right),
			& (\bar z_{4,k},\bar w_{4,k})&=(-\bar z_{3,k},-\bar w_{3,k}).
		\end{aligned}
	\end{equation}
	Moreover, since we are assuming $b_k\neq 0$, direct computations show that these solutions are regular, in the sense that, denoting $f=(f_1,f_2)$,
	$$\det\frac{\partial(f_{1},f_{2})}{\partial(z,w)}(\s_{k},\bar z_{i,k},\bar w_{i,k})\neq 0 \qquad \text{for all $i\in\{1,2,3,4\}$.}$$
	Thus, the implicit function theorem, together with the change of variable \eqref{CV}, allows us to obtain the following result about the precise structure of the set of non-trivial solutions of \eqref{Si} near $(\s_k,0,0)$, with $\l<\s_k$.
	\begin{theorem}
		\label{th:local_bif}
		For each $i\in\{1,2,3,4\}$, there exist $\e_i>0$, a neighborhood $\mc{U}_{i,k}\subset\R^2$ of $(\bar z_{i,k},\bar w_{i,k})$ given in \eqref{eq:sol}, and a $\mc{C}^1$ function $g_{i,k}:(\s_k-\e_i,\s_k]\longrightarrow\mc{U}_{i,k}$, $g_{i,k}(\l)=(z_{i,k}(\l),w_{i,k}(\l))$, such that
		\begin{gather*}
			g_{i,k}(\s_k)=\left(z_{i,k}(\s_k),w_{i,k}(\s_k)\right)=(\bar z_{i,k},\bar w_{i,k}),\\
			f(\l,z_{i,k}(\l),w_{i,k}(\l))=0, \qquad \l\in(\s_k-\e_i,\s_k].
		\end{gather*}
		As a consequence, system \eqref{Si} admits, in a neighborhood of $(\s_{k},0,0)$, with $\l<\s_k$, exactly four branches of non-trivial solutions:
		\begin{equation*}
			\left(x_{i,k}(\l),y_{i,k}(\l)\right)=\sqrt{\s_k-\l}\left(z_{i,k}(\l),w_{i,k}(\l)\right), \qquad i\in\{1,2,3,4\}.
		\end{equation*}
	\end{theorem}
	Observe that it is possible to apply the implicit function theorem also from the solution $(\s_k,0,0)$ of $f=0$. Nevertheless, in such a case, the obtained branch of solutions of \eqref{Si} is the trivial one $\left(x(\l),y(\l)\right)=(0,0)$.
					
	Once the structure of the zero set of $f=(f_1,f_2)$ (equivalently, the set of solutions of system \eqref{Si}) has been determined in a neighborhood of each bifurcation point $(\s_k,0,0)$, with $\l<\s_k$, the homeomorphism $\Psi_{k}$, given in Theorem \ref{Th3.13}, and Proposition \ref{pr:locbif} guarantee that, when the weight $a(t)$ satisfies \eqref{ass:H}, the zero set of $\mf{F}$, i.e., the set of solutions of \eqref{eq:1.1} having exactly $2k$ zeros in $[0,T)$ is topologically equivalent to the one obtained in Theorem \ref{th:local_bif}, which is schematically represented in Figure \ref{F1}.
	\begin{center}
		\begin{figure}[ht!]
			\centering
			\tikzset{every picture/.style={line width=0.75pt}} 
			\begin{tikzpicture}[x=0.75pt,y=0.75pt,yscale=-1,xscale=1]

			\draw  [dash pattern={on 0.84pt off 2.51pt}]  (310.02,148.39) -- (310.23,304.22) ;
			\draw    (166,304.22) -- (455,304.22) ;
			\draw  [dash pattern={on 0.84pt off 2.51pt}]  (455,164) -- (310.23,304.22) ;
			\draw  [dash pattern={on 0.84pt off 2.51pt}]  (179.21,182.58) -- (310.23,304.22) ;
			\draw [color={rgb, 255:red, 74; green, 144; blue, 226 }  ,draw opacity=1 ][line width=1.5]    (250,81) -- (367,81) ;
			\draw  [color={rgb, 255:red, 65; green, 117; blue, 5 }  ,draw opacity=1 ][line width=1.5]  (243.8,116.55) .. controls (330.26,92.49) and (330.06,69.15) .. (243.2,46.56) ;
			\draw  [color={rgb, 255:red, 65; green, 117; blue, 5 }  ,draw opacity=1 ][line width=1.5]  (236.27,135.2) .. controls (332.42,99.27) and (332.58,62.93) .. (236.73,26.2) ;
			\draw  [fill={rgb, 255:red, 0; green, 0; blue, 0 }  ,fill opacity=1 ] (305.5,81) .. controls (305.5,79.34) and (306.84,78) .. (308.5,78) .. controls (310.16,78) and (311.5,79.34) .. (311.5,81) .. controls (311.5,82.66) and (310.16,84) .. (308.5,84) .. controls (306.84,84) and (305.5,82.66) .. (305.5,81) -- cycle ;
			\draw  [fill={rgb, 255:red, 155; green, 155; blue, 155 }  ,fill opacity=0.29 ][dash pattern={on 0.84pt off 2.51pt}] (166,170) -- (310.23,304.22) -- (166,304.22) -- cycle ;
			\draw  [fill={rgb, 255:red, 155; green, 155; blue, 155 }  ,fill opacity=0.29 ][dash pattern={on 0.84pt off 2.51pt}] (455,164) -- (310.23,304.22) -- (455,304.22) -- cycle ;
													
			\draw (458.61,146.42) node [anchor=north west][inner sep=0.75pt]    {$a_k-2b_k=0$};
			\draw (92.36,149.85) node [anchor=north west][inner sep=0.75pt]    {$a_k+2b_k=0$};
			\draw (314.62,141.03) node [anchor=north west][inner sep=0.75pt]    {$a_k$};
			\draw (460.54,296.45) node [anchor=north west][inner sep=0.75pt]    {$b_k$};
			\draw (310.5,84.4) node [anchor=north west][inner sep=0.75pt]    {$\sigma _{k}$};
										
			\end{tikzpicture}
		\caption{Scheme of the solutions set of system \eqref{Si} near $(\s_{k},0,0)$, according to the values of $(a_k,b_k)$ defined in \eqref{eq:coeff}.}\label{F1}
		\end{figure}
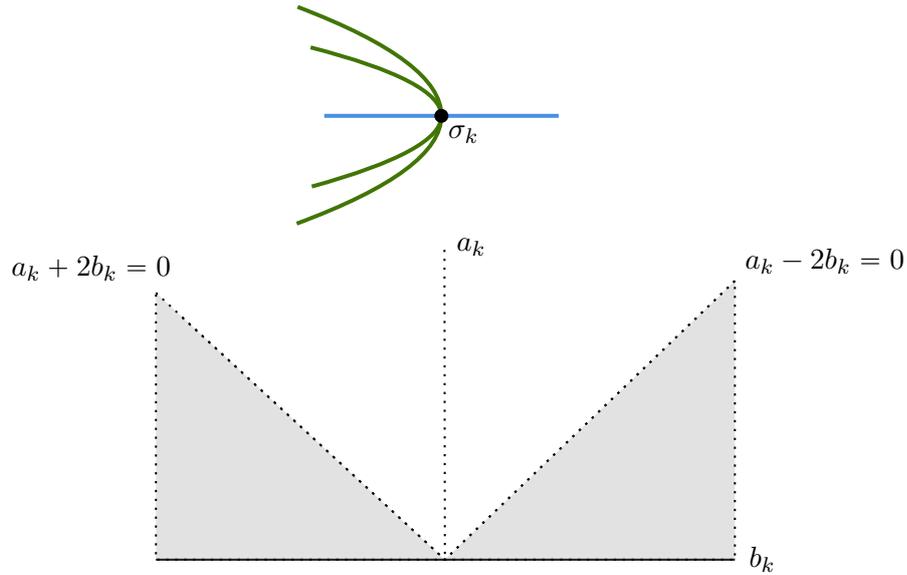
	\end{center}

	\par
	To conclude this section, we present a particular example. Let $T=\pi$, and consider the periodic problem
	\begin{equation}
		\label{E00}
		\left\{
		\begin{array}{l}
		-u''=\l u + \mathbf{1}_{[0,\tfrac{\pi}{4}]}u^{3}, \quad \quad t\in (0,\pi),  \\[7pt]
		u(0)=u(\pi), \quad u'(0)=u'(\pi),
		\end{array}
		\right.
	\end{equation}
	which is \eqref{eq:1.1} with
	$$a(t)=\mathbf{1}_{\left[0,\tfrac{\pi}{4}\right]}(t):=
		\begin{cases}
		1, & \text{ if  $t\in [0,\tfrac{\pi}{4}]$}, \\
		0, & \text{ if  $t\in (\tfrac{\pi}{4},\pi)$}.
		\end{cases}
	$$
	Here, we use the notation $\mathbf{1}_{A}$ to denote the indicator function of the subset $A$. We apply the results obtained in this section to study the bifurcation diagram of this problem near the eigenvalues $\s_{k}=4k^2$, with $k$ odd. Indeed, in this particular case, the coefficients \eqref{eq:coeff} are
	\[
		a_k=\frac{3}{8\pi}, \quad b_k=\frac{1}{2k\pi^2}, \quad c_k=\frac{1}{8\pi}, \quad d_k=\frac{1}{2k\pi^2}, \quad e_k=\frac{3}{8\pi},
	\]
	thus the conditions in \eqref{ass:H} are satisfied,	and the following result holds (see Figure \ref{F2}).
	\begin{theorem}
		\label{th:3.18}
		Assume that $k\in\N$ is odd. Then, exactly four branches of non-trivial periodic solutions of \eqref{E00} with $2k$ zeros in $[0,T)$ emanate locally from $(4k^2,0)\in \R\times H^{2}(\mathbb{T})$ for $\l<4k^2$, while the unique solutions in a neighborhood of $(4k^2,0)\in \R\times H^{2}(\mathbb{T})$ for $\l>4k^2$, $\l\sim 4k^2$, are the trivial ones. Moreover, these non-trivial branches are locally given by
		$$\left(\l,u_{i,k}(\l)\right)=\Psi_{k}(\l,\sqrt{4k^2-\l} \, z_{i,k}(\l),\sqrt{4k^2-\l} \, w_{i,k}(\l)), \quad \l< 4k^2, \, \l\sim 4k^2, \quad i\in\{1,2,3,4\},$$
		where $(z_{i,k}(\l),w_{i,k}(\l))$ are  $\mc{C}^{1}$ functions defined in a left neighborhood of $4k^2$ satisfying
		\begin{align*}
		\left(z_{1,k}(4k^2),w_{1,k}(4k^2)\right)&=\left( \tfrac{2\sqrt{k}\pi}{\sqrt{3k\pi+8}}, \tfrac{2\sqrt{k}\pi}{\sqrt{3k\pi+8}}\right), &
		\left(z_{2,k}(4k^2),w_{2,k}(4k^2)\right)&=\left( \tfrac{-2\sqrt{k}\pi}{\sqrt{3k\pi+8}}, \tfrac{-2\sqrt{k}\pi}{\sqrt{3k\pi+8}}\right),\\
		\left(z_{3,k}(4k^2),w_{3,k}(4k^2)\right)&=\left( \tfrac{2\sqrt{k}\pi}{\sqrt{3k\pi-8}}, \tfrac{-2\sqrt{k}\pi}{\sqrt{3k\pi-8}}\right), &
		\left(z_{4,k}(4k^2),w_{4,k}(4k^2)\right)&=\left( \tfrac{-2\sqrt{k}\pi}{\sqrt{3k\pi-8}}, \tfrac{2\sqrt{k}\pi}{\sqrt{3k\pi-8}}\right),
		\end{align*}
		and $\Psi_{k}\colon \mf{H}^{-1}(0) \longrightarrow \mf{F}^{-1}(0)\cap \mc{O}_{k}$ is the homeomorphism given in Theorem \ref{Th3.13}. 
	\end{theorem}
										
	\begin{figure}[h!]
		\begin{center}
								
		\tikzset{every picture/.style={line width=0.75pt}}         
												
		\begin{tikzpicture}[x=0.75pt,y=0.75pt,yscale=-0.75,xscale=0.75]
		
		\draw [color={rgb, 255:red, 74; green, 144; blue, 226 }  ,draw opacity=1 ][line width=1.5]    (68,132) -- (413,129.11) ;
		
		\draw  [color={rgb, 255:red, 65; green, 117; blue, 5 }  ,draw opacity=1 ][line width=1.5]  (213.47,189.39) .. controls (353.5,148.58) and (353.18,109.02) .. (212.49,70.71) ;
		
		\draw  [color={rgb, 255:red, 65; green, 117; blue, 5 }  ,draw opacity=1 ][line width=1.5]  (201.27,221) .. controls (357,160.08) and (357.25,98.48) .. (202.01,36.2) ;
		
		\draw  [fill={rgb, 255:red, 0; green, 0; blue, 0 }  ,fill opacity=1 ] (313.39,129.11) .. controls (313.39,126.3) and (315.57,124.02) .. (318.25,124.02) .. controls (320.94,124.02) and (323.11,126.3) .. (323.11,129.11) .. controls (323.11,131.92) and (320.94,134.2) .. (318.25,134.2) .. controls (315.57,134.2) and (313.39,131.92) .. (313.39,129.11) -- cycle ;
		
		\draw [line width=1.5]    (120,21) -- (119,228) ;

		\draw (318.16,140.24) node [anchor=north west][inner sep=0.75pt]    {$4k^2$};
		
		\draw (415,132.51) node [anchor=north west][inner sep=0.75pt]    {$\lambda $};
		
		\draw (129,12.4) node [anchor=north west][inner sep=0.75pt]    {$H^{2}(\mathbb{T})$};
													
		\end{tikzpicture}
	\end{center}
	\caption{Local structure of the solutions of \eqref{E00} in a neighborhood of the bifurcation point $\s_k=4k^2$, with $k$ odd: the trivial solutions are plotted in blue, while the branches of nontrivial solutions given in Theorem \ref{th:3.18} are plotted in green. These nontrivial solutions have $2k$ zeros in $[0,\pi)$.}
		\label{F2}
	\end{figure}
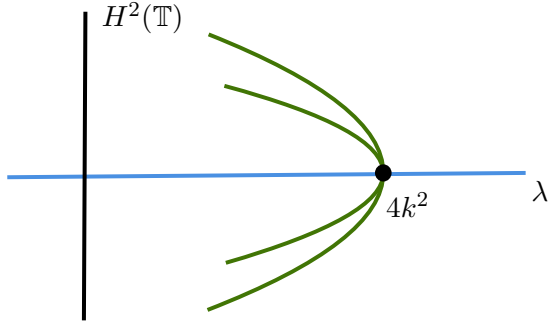
	When $k\in \N$ is even, with the weight of this example we have $b_k=d_k=0$, thus the analysis of this section cannot be applied at the corresponding eigenvalues $\s_k$.

	\section{Global bifurcation analysis}
		\label{section:4}
		The results of Section \ref{section:3} are local, in the sense that they give information about the solutions of \eqref{eq:1.1} only in a neighborhood of the bifurcation points $(\s_k,0)$. In this section, we complete such results by obtaining \emph{global} information about the components bifurcating at those bifurcation points. The techniques we use rely on global bifurcation alternatives which have been summarized in Appendix \ref{app:B.3} for the readers' convenience. In order to apply such tools, it is fundamental to consider the a priori bounds on the solutions of \eqref{eq:1.1} obtained in Proposition \ref{pr:2.4}. Thus, throughout this section, in addition to \eqref{ass:aloc} we also assume that the weight $a$ satisfies \eqref{ass:aglob}.
	
	Those bounds and the structure of equation \eqref{eq:1.1} guarantee that, for each $k\in\N\cup\{0\}$ and any compact $K\subset\R$, there exists $M=M(k,K)>0$ such that every solution $(\l,u)$ with $\l\in K$ and $u$ having $2k$ zeros in $[0,T)$ satisfies
	\[
	\|u\|_{H^2}<M.
	\]									
	\subsection{Global bifurcation at $\l=\s_0$}
	\label{sec:4.1}
	This section is devoted to study the global structure of the connected component $\mathscr{C}_{0}$ of non-trivial solutions emanating from the bifurcation point $(\s_{0},0)=(0,0)\in\R\times H^{2}(\mathbb{T})$. We recall that the local existence and uniqueness of this component have been obtained in Theorem \ref{LCRW}. With the notation used there, we denote by $\mathscr{C}_{0}^{+}\subset \mathscr{C}_{0}$ the connected component of $\mf{F}^{-1}(0)\setminus\mc{T}$ such that, for some $\rho>0$,
	$$\mathscr{C}^{+}_{0}\cap B_{\rho}(0,0) = \mathscr{C}_{0,\mathrm{loc}}^{+}= \left\{\left(\L(s),s\left(1+\Gamma(s)\right)\right) \colon s\in(0,\varepsilon)\right\}.$$
	Similarly, $\mathscr{C}_{0}^{-}\subset \mathscr{C}_{0}$ denotes the connected component of $\mf{F}^{-1}(0)\setminus\mc{T}$ such that, for some $\rho>0$,
	$$\mathscr{C}^{-}_{0}\cap B_{\rho}(0,0) = \mathscr{C}_{0,\mathrm{loc}}^{-}= \left\{\left(\L(s),s\left(1+\Gamma(s)\right)\right) \colon s\in(-\varepsilon,0)\right\}.$$
	Observe that by Theorem \ref{LCRW}, near $(\s_{0},0)$, $\mathscr{C}_{0}^{+}$ consists of positive solutions of \eqref{eq:1.1}, while $\mathscr{C}^{-}_{0}$ consists of negative solutions. Moreover, from the definition of $\mathscr{C}^{\pm}_{0}$, we have $\mathscr{C}^{\pm}_{0}\cap \mc{T}=\emptyset$.
	We start with the following qualitative result.
													
	\begin{proposition}
		\label{pr:4.1}
		Assume \eqref{ass:aloc} and \eqref{ass:aglob}. Then, the connected components $\mathscr{C}_{0}^{+}$ and $\mathscr{C}^{-}_{0}$ defined above are disjoint. 
		Moreover, $\mathscr{C}_{0}^{+}$ consists of positive solutions of \eqref{eq:1.1}, while $\mathscr{C}^{-}_{0}$ of negative solutions. Finally, $\mathscr{C}_{0}^{+}=-\mathscr{C}_{0}^{-}$, where we recall that, given a subset $\mathscr{C}\subset \R\times H^{2}(\mathbb{T})$, we denote $-\mathscr{C}:=\left\{(\l,-u) \colon (\l,u)\in\mathscr{C}\right\}$.
	\end{proposition}
													
	\begin{proof}
		It is enough to prove that $\mathscr{C}^{+}_{0}$ consists of positive solutions of \eqref{eq:1.1}. The rest follows since, if $u$ is a solution of \eqref{eq:1.1}, also $-u$ is. Locally, this assertion was proven in Theorem \ref{LCRW}. Suppose now that the global claim is false. Then, there exist sequences $\{(\l_{n},u_{n})\}_{n\in\N}\subset \mathscr{C}^{+}_{0}$ and $\{t_{n}\}_{n\in\N}\subset [0,T]$ such that 
		$$\lim_{n\to+\infty}\l_{n}=\l_{0}\in\R, \quad u_{n}(t) > 0 \text{ for all  $t\in [0,T]$}, \quad \text{ and } \quad\lim_{n\to+\infty} u_{n}(t_{n})=0. $$
		As commented above, by Proposition \ref{pr:2.4}(ii) together with the differential equation of \eqref{eq:1.1}, $\{u_n\}$ is bounded in $H^{2}(\mathbb{T})$; thus, up to subsequences that we do not relabel, by compactness there exist $u_0\in \mathfrak{F}^{-1}(0)\subset  H^{2}(\mathbb{T})$ and $t_0\in[0,T]$ such that
		$$\lim_{n\to +\infty} t_n=t_0, \qquad  \lim_{n\to+\infty}u_{n}=u_{0} \text{ in $H^{2}(\mathbb{T})$}.$$
		Moreover, the Sobolev embedding $H^{2}(\mathbb{T}) \hookrightarrow \mc{C}^{1}(\mathbb{T})$ implies that $u_{n}, u_{0}\in\mc{C}^{1}(\mathbb{T})$ for all $n\in\N$,
		\begin{equation}
			\label{ECC}
		\lim_{n\to+\infty} u_{n}=u_{0}, \quad \text{in $\mc{C}^{1}(\mathbb{T})$},
		\end{equation}
		and $u_{0}(t_{0})=0$. If $u_{0}'(t_{0})\neq 0$, there exists $t_{\ast}\in (0,T)$, $t_{\ast}\sim t_{0}$, such that $u_{0}(t_{\ast})<0$.  But this contradicts \eqref{ECC}, because the functions $u_n$ are positive. Hence, $u'_{0}(t_{0})=0$ and $u_0$ is a solution of the Cauchy problem
		\begin{equation}
			\label{ECCk2}
			\left\{
			\begin{array}{l}
			-u_{0}''=\l_{0} u_{0} + a(t)u_{0}^{3}, \quad t\in (0,T),  \\
			u_{0}(t_0)=0, \\
			u_{0}'(t_{0})=0.
			\end{array}
			\right.
		\end{equation}
		The uniqueness ensured by the Cauchy--Lipschitz theorem (possibly applied on each side of the discontinuity points of $a(t)$), forces $u_0=0$, i.e., $(\l_0,u_0)=(\l_0,0)\notin\mathscr{C}^{+}_{0}$, because $\mathscr{C}^{+}_{0}\cap\mc{T}=\emptyset$.
	\end{proof}
	
	\begin{theorem}[\textbf{Global structure of $\mathscr{C}_{0}$}]
		\label{th:4.5}
		Assume \eqref{ass:aloc} and \eqref{ass:aglob}. Then, the connected component $\mathscr{C}_{0}$ is unbounded and 
		\begin{equation}
			\label{Eq4.5}
			\mc{P}_{\l}(\mathscr{C}_{0})= (-\infty,0],
		\end{equation}
		where $\mc{P}_{\l}\colon\R\times H^{2}(\mathbb{T})\to \R$ is the projection on the first component: $\mc{P}_{\l}(\l,u)=\l$. In particular, problem \eqref{eq:1.1} admits at least a positive and a negative solution for each $\l<0$.
	\end{theorem}
													
	\begin{proof}
		We apply Theorem \ref{TGB} to the nonlinearity 
		\begin{equation}
			\label{DO}
			\mf{F}\colon\R\times H^{2}(\mathbb{T})\longrightarrow L^{2}(\mathbb{T}), \quad \mf{F}(\l,u):=u''+\l u + a(t) u^{3}.
		\end{equation}
		We proceed to verify the assumptions of Theorem \ref{TGB}. First of all, $\mf{F}$ is orientable in the sense of Fitzpatrick, Pejsachowicz and Rabier since its domain $\R\times H^{2}(\mathbb{T})$ is simply connected (see \cite{FPRb}). {\rm{(F2)}} holds by definition (cf., \eqref{DO}). {\rm(F3)} follows from Lemma \ref{L3} and {\rm (F4)} from Lemma \ref{LF4}. Hypothesis {\rm(F5)} follows from \eqref{E1.1}. Moreover, the linearization
		$$\mf{L}\colon\R\longrightarrow \Phi_{0}(H^{2}(\mathbb{T}), L^{2}(\mathbb{T})), \quad \mf{L}(\l)[v]=v''+\l v,$$
		is clearly analytic and $\chi[\mf{L}, \s_{0}]=1$ by Proposition \ref{P3.5}, i.e., it is odd. Therefore, by applying Theorem \ref{TGB}, we obtain that $\mathscr{C}_{0}$ is unbounded or there exists $\s_{m}\in \Sigma(\mf{L})$, $\s_{m}\neq \s_{0}$, such that $(\s_{m},0)\in\mathscr{C}_{0}$. We will now show that the second alternative cannot occur. Proposition \ref{pr:2.4}(i) implies that $\s_{m}\leq 0$, because $\mathscr{C}_{0}\subset (-\infty,0]\times H^{2}(\mathbb{T})$. As $\s_{m}\neq \s_{0}=0$, we deduce that $\s_{m}<0$. This is impossible, because $\Sigma(\mf{L})\subset [0,+\infty)$. Therefore, $\mathscr{C}_{0}$ is unbounded in $\R\times H^{2}(\mathbb{T})$. To prove \eqref{Eq4.5}, assume by contradiction  that there exists $\l_{\ast}<0$, such that 
		\begin{equation}
			\label{PYS}
			\mc{P}_{\l}(\mathscr{C}_{0})=[\l_{\ast},0] \quad \text{ or } \quad  \mc{P}_{\l}(\mathscr{C}_{0})=(\l_{\ast},0].
		\end{equation}
		By Proposition \ref{pr:2.4}(ii), $\mathscr{C}_{0}$ is bounded in $[\l_{\ast},0]\times H^{2}(\mathbb{T})$, hence in $\R\times H^{2}(\mathbb{T})$ by \eqref{PYS}. This contradiction with Theorem \ref{TGB} proves \eqref{Eq4.5}. Finally, the last part of the statement follows from the properties of $\mathscr{C}_{0}^\pm$ established in Proposition \ref{pr:4.1}; in particular because $-\mathscr{C}_{0}^+=\mathscr{C}_{0}^-$.
	\end{proof}
													
	\subsection{Global bifurcation from $\s_k$, $k\in\N$: general case}
	\label{sec:4.2}
	
	In this section, we apply the analytic global bifurcation Theorem \ref{th:glob_analytic1} in order to obtain global information of the local branches of solutions emanating from the bifurcation points $(\s_{k},0)$, for $k\in\N$. To apply the analytic bifurcation theorem, the operator $\mf{F}$ must be analytic. This fact comes from the definition of the operator \eqref{EqDe} (for the definition of analyticity in Banach spaces, we refer the reader to Section 6 of \cite{JJ5}). In the sequel, we denote
	$$\mc{R}(\mf{F}):=\{(\l,u)\in\mathbb{R}\times H^{2}(\mathbb{T}) \colon \partial_{u}\mf{F}(\l,u)\in GL(H^{2}(\mathbb{T}),L^{2}(\mathbb{T}))\}.$$
	Recall that, by Theorem \ref{Th3.13}, $(\l,u)\in \mc{R}(\mf{F})\cap\mf{F}^{-1}(0)$ if and only if $\partial_{(x_1,x_2)}\mf{H}(\Psi_{k}^{-1}(\l,u))\in GL(\mathbb{R}^{2})$.
	\par The main result of this subsection is the following. Essentially, it states that if there exists a local curve of solutions for the finite-dimensional equation $\mathfrak{H}(\lambda, x_1, x_2)=0$, emanating from $(\s_{k},0,0)$ and satisfying some reasonable properties, then the corresponding local solution curve of $\mathfrak{F}(\lambda, u) = 0$ via the homeomorphism $\Psi_k$, can be extended to a global curve of solutions and, moreover, it satisfies a global alternative.
	
	\begin{theorem}
		\label{th:glob_analytic}
		Assume \eqref{ass:aloc} and \eqref{ass:aglob}. Let $\varepsilon>0$, $k\in\N$, and suppose that 
		$$\gamma_k\colon(0,\varepsilon)\longrightarrow \mf{H}^{-1}(0)$$
		is an analytic injective curve of solutions of the equation $\mf{H}(\l,x_1,x_2)=0$ such that
		\begin{itemize}
			\item[{\rm(a)}] $\lim_{s\downarrow 0}\gamma_k(s)=(\s_{k},0,0)$.
			\item[{\rm(b)}] $\gamma_k(s)\notin \{(\l,0,0) \colon \l\in\R\}$ for every $s\in(0,\varepsilon)$.
			\item[{\rm(c)}] $\partial_{(x_1,x_2)}\mf{H}(\gamma_k(s))\in GL(\R^{2})$ for every $s\in (0,\varepsilon)$.
		\end{itemize}
		Then, the corresponding analytic curve of solutions of $\mf{F}(\l,u)=0$ given by
		$$ \rho_k:(0,\varepsilon)\longrightarrow\mf{F}^{-1}(0), \quad {\rho}_k(s):=(\Psi_{k}\circ \gamma_k)(s),$$
		admits a locally injective continuous path $\G_k: \R_{>0}\longrightarrow \mf{F}^{-1}(0)$, for which there exists $0<\delta<\varepsilon$ satisfying $\Gamma_k(s)=\rho_k(s)$ for all $s\in (0,\delta)$. Moreover, $\Gamma_k$ satisfies one of the following alternatives:
		\begin{enumerate}
			\item[{\rm (i)}]   $\mc{P}_{\l}(\G_{k}(\R_{> 0}))=(-\infty,\s_{k})$,\item[{\rm (ii)}]   $\Gamma_k$ is a closed loop, i.e., there exists $S>0$ such that $\Gamma_k(S)=(\s_{k},0)$.
		\end{enumerate}
	\end{theorem}
	\begin{proof}
		From hypotheses (a)--(c) on the curve $\gamma_{k}$ and  Theorem \ref{Th3.13}, we deduce that $\rho_{k}$ satisfies
		$$
		\rho_{k}(0,\varepsilon)\subset\mf{F}^{-1}(0)\setminus\mc{T}, \quad \lim_{s\downarrow0}\rho_{k}(s)=\Psi_{k}(\s_{k},0,0)=(\s_{k},0), \;\; \text{and} \;\; \rho_{k}(0,\varepsilon)\subset\mc{R}(\mf{F}).
		$$
		Thus, applying Theorem \ref{th:glob_analytic1}, there exist $\d\in(0,\e)$ and a locally injective continuous path $\G_{k}:\R_{> 0}\to\mf{F}^{-1}(0)$ such that $\G_{k}(s)=\rho_{k}(s)$ for all $s\in(0,\d)$ and it satisfies one of the following  alternatives:
		\begin{enumerate}
			\item[{\rm (i)}]   $\lim_{s\to +\infty} \|\Gamma_{k}(s)\|_{\mathbb{R}\times H^{2}(\mathbb{T})}= +\infty$.
			\item[{\rm (ii)}]   $\Gamma_{k}$ is a closed loop, i.e., there exists $S>0$ such that $\Gamma_{k}(S)=(\s_{k},0)$.
		\end{enumerate}
		Note that the former is equivalent to $\mc{P}_{\l}(\G_{k}(\R_{> 0}))=(-\infty,\s_{k})$ as a direct consequence of Proposition \ref{pr:2.4}. The proof is concluded.
	\end{proof}
	
	In addition, we can prove that the solutions lying on the global curves $\G_k$ determined above have exactly $2k$ simple zeros in $[0,T)$. This is the content of the next result.
												
	\begin{proposition}
		\label{pr:4.5}
		For all $k\in\N$, the curves $\G_k(\R_{> 0})\setminus\{(\s_{k},0)\}$ obtained in Theorem \ref{th:glob_analytic} consist of solutions of \eqref{eq:1.1} with exactly $2k$ simple zeros in $[0,T)$.
	\end{proposition}
												
	\begin{proof}
		Fix $k\in\N$. Observe that Proposition \ref{pr:locbif} guarantees the validity of the statement in a neighborhood of the bifurcation point $(\s_k,0)$. Suppose instead that the global statement is false. Then, by reasoning as in the proof of Proposition \ref{pr:4.1}, there exists a sequence $\left\{(\l_{n},u_{n})\right\}_{n\in\N}\subset \G_k(\R_{> 0})\setminus\{(\s_k,0)\}$ and $\left(\l_{0},u_{0}\right)\in\G_k(\R_{> 0})\setminus\left\{\left(\s_{k},0\right)\right\}$ such that
		\begin{equation}
			\label{E4.10}
			\lim_{n\to+\infty}\l_{n}=\l_{0}, \quad \lim_{n\to+\infty}u_{n}=u_{0}  \text{ in $H^{2}(\mathbb{T})$},
		\end{equation}
		$u_{n}(t)$ has exactly $2k$ zeros in $[0,T)$ for each $n\in \N$, while the number of zeros of $u_{0}(t)$ is different from $2k$. For all $j\in\{1,\dots,k\}$, consider the sequences $\left\{s_{j,n}\right\}_{n\in\N}, \left\{t_{j,n}\right\}_{n\in\N}\subset [0,T]$ such that, for all $n\in\N$,
		\begin{gather*}
		0\leq s_{1,n}<t_{1,n}<s_{2,n}<t_{2,n}<\dots<s_{k,n}<t_{k,n}\leq T, \\
		u_{n}(s_{j,n})=u_{n}(t_{j,n})=0, \quad u_{n}(t)\neq 0 \text{ for all $t\in \left(s_{j,n},t_{j,n}\right)$}.
		\end{gather*}
		Moreover, if $s_{1,n}=0$, then $t_{k,n}<T$. By compactness, up to subsequences that we do not relabel, for all $j\in\{1,\dots,k\}$ there exists $s_{j,0}, t_{j,0}\in [0,T]$ such that
		$$\lim_{n\to+\infty} s_{j,n}=s_{j,0}, \quad \lim_{n\to+\infty} t_{j,n}=t_{j,0}.$$
		In addition, $0\leq s_{1,0}\leq t_{1,0}\leq s_{2,0}\leq t_{2,0}\leq\dots\leq s_{k,0}\leq t_{k,0}\leq T$.	By \eqref{E4.10} and the Sobolev embedding $H^{2}(\mathbb{T}) \hookrightarrow \mc{C}^{1}(\mathbb{T})$, we deduce that 
		\begin{equation}
			\label{EqLi}
			\lim_{n\to +\infty} \|u_{n}-u_{0}\|_{L^{\infty}}=0.
		\end{equation}
		In particular,
		$$0=\lim_{n\to +\infty} u_{n}(s_{j,n})=u_{0}(s_{j,0}), \quad 0=\lim_{n\to +\infty} u_{n}(t_{j,n})=u_{0}(t_{j,0}).$$
		Assume that there exists $j\in\{1,\dots,k\}$ such that $s_{j,0}=t_{j,0}=:t^*$. Since, by Rolle's theorem, for all $n\in\N$ there exists $p_n\in(s_n,t_n)$ with $u_n'(p_n)=0$, from \eqref{E4.10} we have
		\[
		\lim_{n\to+\infty}p_n=t^*, \qquad \lim_{n\to+\infty} u_n'(p_n)=u_0'(t^*)=0.
		\]
		Then, $u_0\in\mc{C}^1([0,T])$ solves the Cauchy problem
		\begin{equation}
			\label{ECCk2bis}
			\left\{
			\begin{array}{l}
			-u_{0}''=\l_{0} u_{0} + a(t)u_{0}^{3}, \quad t\in (0,T),  \\
			u_{0}(t^*)=0, \\
			u_{0}'(t^*)=0,
			\end{array}
			\right.
		\end{equation}
		and the uniqueness given by the Cauchy--Lipschitz theorem gives $u_0=0$. Then, $\l_0=\s_k$, because $(\s_k,0)$ is the unique point on the trivial branch from which solutions with $2k$ zeros in $[0,T)$ bifurcate. This is impossible because we had supposed $(\l_0,u_0)\neq(\s_k,0)$.	The same contradiction can be obtained if $t_{j,0}=s_{j+1,0}$ for some $j\in\{1,\dots,k\}$ (for notation convenience, we set $s_{k+1,0}:=s_{1,0}+T$).

		Therefore, the only case left to consider is when $s_{1,0}<t_{1,0}<s_{2,0}<t_{2,0}<\dots<s_{k,0}<t_{k,0}<s_{1,0}+T$. Since we are assuming that $u_0$ has a number of zeros different from $2k$, there exist $j\in\{1,\dots,k\}$ and $t^*\in(s_{j,0},t_{j,0}) \cup (t_{j,0},s_{j+1,0})$ with $u_0(t^*)=0$. If also $u_0'(t^*)=0$ we conclude as before. If $u_0'(t^*)\neq 0$, instead, we get a contradiction with \eqref{EqLi}, because $u_0$ changes sign in a neighborhood of $t^*$, while, by construction, $u_n$ have a fixed sign in a neighborhood of $t^*$.
	\end{proof}
													
	\subsection{Global bifurcation from $\s_k$, $k\in\N$: case of even weights}
	\label{sec:4.3} 
	With respect to the previous section, when the weight $a$ is further assumed to be even, we can get more precise information about the global structure of the connected components containing the local branches emanating from the bifurcation points $(\s_{k},0)\in\R\times H^{2}(\mathbb{T})$, $k\geq 1$, whose local existence has been established in Theorem \ref{CR}.
	
	With the notation used there, we denote by $\mathscr{C}_{\mathbf{\mathtt{ev}},k}^{+}\subset \mathscr{C}_{\mathbf{\mathtt{ev}},k}$ the connected component of $\mf{F}_{\mathbf{\mathtt{ev}}}^{-1}(0)\setminus\mc{T}$ such that, for some $\rho>0$,
	$\mathscr{C}^{+}_{\mathbf{\mathtt{ev}},k}\cap B_{\rho}(\s_k,0) = \mathscr{C}_{\mathbf{\mathtt{ev}},k,\mathrm{loc}}^{+}$. Similarly, $\mathscr{C}_{\mathbf{\mathtt{ev}},k}^{-}\subset \mathscr{C}_{\mathbf{\mathtt{ev}},k}$ denotes the connected component of $\mf{F}^{-1}_{\mathbf{\mathtt{ev}}}(0)\setminus\mc{T}$ such that, for some $\rho>0$, $\mathscr{C}^{-}_{\mathbf{\mathtt{ev}},k}\cap B_{\rho}(\s_k,0) = \mathscr{C}_{\mathbf{\mathtt{ev}},k,\mathrm{loc}}^{-}$, and analogous notations will be used for the case of odd solutions. Recall that, close to the bifurcation point $(\s_k,0)$, these components consist of even/odd solutions of \eqref{eq:1.1} with exactly $2k$ simple zeros in $[0,T)$. Moreover, since $u$ is a solution of \eqref{eq:1.1} if, and only if, $-u$ is, we also have that $-\mathscr{C}_{\mathbf{\mathtt{ev}},k}^{+}=\mathscr{C}_{\mathbf{\mathtt{ev}},k}^{-}$ and  $-\mathscr{C}_{\mathbf{\mathtt{odd}},k}^{+}= \mathscr{C}_{\mathbf{\mathtt{odd}},k}^{-}$.
	
	To make the notation uniform, in this section we denote by $\mathscr{C}_{\mathbf{\mathtt{ev}},0}$ the component denoted by $\mathscr{C}_0$ in Section \ref{sec:4.1} and whose global behavior has been established in Theorem \ref{th:4.5}. Indeed, as pointed out in Remark \ref{re:3.9}, except at the bifurcation point $(\s_0,0)=(0,0)$, it consists of even positive or negative solutions of \eqref{eq:1.1}.
													
	By reasoning as in the proof of Proposition \ref{pr:4.5}, it is possible to obtain the following result, showing that the number of zeros of the solutions is maintained along the connected components that we have just introduced.
	\begin{proposition}
		\label{Pr4.6}
		Assume \eqref{ass:aloc} and \eqref{ass:aglob}. Then, for each $k\in\N$, the following statements hold.
		\begin{itemize}
			\item[(a)] $\mathscr{C}_{\mathbf{\mathtt{ev}},k}\setminus\{(\s_{k},0)\}\subset \R\times H^{2}_{\mathbf{\mathtt{ev}}}(\mathbb{T})$ consists of even solutions of \eqref{eq:1.1} with exactly $2k$ zeros in $[0,T)$. Thus, the connected components $\left\{\mathscr{C}_{\mathbf{\mathtt{ev}},n}\right\}_{n=0}^{\infty}\subset\R\times H^{2}_{\mathbf{\mathtt{ev}}}(\mathbb{T})$ are pairwise disjoint, i.e.,
			$$\mathscr{C}_{\mathbf{\mathtt{ev}},n}\cap \mathscr{C}_{\mathbf{\mathtt{ev}},m}=\emptyset, \quad \text{if $n, m\in\N\cup\{0\}$ and $n\neq m$}.$$
			\item[(b)] $\mathscr{C}_{\mathbf{\mathtt{odd}},k}\setminus\{(\s_{k},0)\}\subset \R\times H^{2}_{\mathbf{\mathtt{odd}}}(\mathbb{T})$ consists of odd solutions of \eqref{eq:1.1} with exactly $2k$ zeros in $[0,T)$. Thus, the connected components $\left\{\mathscr{C}_{\mathbf{\mathtt{odd}},k}\right\}_{n=1}^{\infty}\subset\R\times H^{2}_{\mathbf{\mathtt{odd}}}(\mathbb{T})$ are pairwise disjoint, i.e.,
			$$\mathscr{C}_{\mathbf{\mathtt{odd}},n}\cap \mathscr{C}_{\mathbf{\mathtt{odd}},m}=\emptyset, \quad \text{if $n\neq m$}.$$
		\end{itemize}
	\end{proposition}
	Now we are ready to obtain the global behavior of the components of even and odd solutions.
	\begin{theorem}[\textbf{Global structure of $\mathscr{C}_{\mathbf{\mathtt{ev}},k}$ and $\mathscr{C}_{\mathbf{\mathtt{odd}},k}$}]
		\label{th:glog_ev}
		Assume \eqref{ass:aloc} and \eqref{ass:aglob}. Then, for each $k\in \N$, the following statements hold.
		\begin{itemize}
			\item[(a)]
			The connected component $\mathscr{C}_{\mathbf{\mathtt{ev}},k}$ is unbounded and 
			\begin{equation}
				\label{Eq4.51}
				\mc{P}_{\l}(\mathscr{C}_{\mathbf{\mathtt{ev}},k})= (-\infty,\s_{k}].
			\end{equation}
			In particular, \eqref{eq:1.1} admits at least two even solution  with $2k$ zeros in $[0,T)$ for each $\l<\s_{k}$.
			\item[(b)]
			The connected component $\mathscr{C}_{\mathbf{\mathtt{odd}},k}$ is unbounded and 
			\begin{equation}
				\label{Eq4.52}
				\mc{P}_{\l}(\mathscr{C}_{\mathbf{\mathtt{odd}},k})= (-\infty,\s_{k}].
			\end{equation}
			In particular, \eqref{eq:1.1} admits at least two odd solution  with $2k$ zeros in $[0,T)$ for each $\l<\s_{k}$.
		\end{itemize}
	\end{theorem}
													
	\begin{proof}
		We prove only statement (a), as the proof of (b) is analogous. We apply Theorem \ref{TGB} to
		\begin{equation}
			\label{DObis}
			\mf{F}_{\textbf{\texttt{ev}}}\colon\R\times H^{2}_{\textbf{\texttt{ev}}}(\mathbb{T})\longrightarrow L^{2}_{\textbf{\texttt{ev}}}(\mathbb{T}), \quad \mf{F}_{\textbf{\texttt{ev}}}(\l,u):=u''+\l u + a(t) u^{3}.
		\end{equation}
		The validity of the assumptions of Theorem \ref{TGB} can be checked as in Theorem \ref{th:4.5}, the key point being that $\chi[\mf{L}_{\textbf{\texttt{ev}}}, \s_{k}]$ is odd, by \eqref{eq:algmult}. Therefore, Theorem \ref{TGB} guarantees that $\mathscr{C}_{\textbf{\texttt{ev}},k}$ is unbounded or there exists $\s_{m}\in \Sigma(\mf{L}_{\textbf{\texttt{ev}}})$, $\s_{m}\neq \s_{k}$, such that $(\s_{m},0)\in\mathscr{C}_{\textbf{\texttt{ev}},k}$. By Proposition \ref{Pr4.6}(a) and Theorem \ref{LCRWw}(ii), the second alternative cannot occur. Therefore, the connected component $\mathscr{C}_{\textbf{\texttt{ev}},k}$ is unbounded. Suppose now by contradiction that there exists $\l_{\ast}<\s_{k}$, such that 
		\begin{equation}
			\label{PYSbis}
			\mc{P}_{\l}(\mathscr{C}_{\textbf{\texttt{ev}},k})=[\l_{\ast},\s_{k}] \quad 	\text{ or }  \quad \mc{P}_{\l}(\mathscr{C}_{\textbf{\texttt{ev}},k})=(\l_{\ast},\s_{k}].
		\end{equation}
		Note that $\l_\ast$ is necessarily smaller than $\s_k$ because of Propositions \ref{Pr4.6}(a) and  \ref{pr:2.4}(i).	By Proposition \ref{pr:2.4}(ii), $\mathscr{C}_{\textbf{\texttt{ev}},k}$ is bounded in $[\l_{\ast},\s_k]\times H^{2}_{\textbf{\texttt{ev}}}(\mathbb{T})$, hence in $\R\times H^{2}_{\textbf{\texttt{ev}}}(\mathbb{T})$ by \eqref{PYSbis}. This contradiction with Theorem \ref{TGB} proves \eqref{Eq4.51}.
													
		Finally, the last part of the statement follows from the properties of $\mathscr{C}_{\textbf{\texttt{ev}},k}^\pm$; in particular because $-\mathscr{C}_{\textbf{\texttt{ev}},k}^+=\mathscr{C}_{\textbf{\texttt{ev}},k}^-$.
	\end{proof}								
													
	\section{Additional examples and final remarks}
	\label{section:5}
	We conclude this work by presenting additional examples as well as some observations and discussions related to the results obtained in this work. We will also complement our analysis with some numerical experiments to illustrate the behavior of \eqref{eq:1.1} in some specific cases.
	\begin{remark}{\bf (Additional branches).}
		\label{re:5.1}
		By adapting the techniques of Section \ref{sec:3.4}, we can show that a larger number of branches of nontrivial solutions of \eqref{eq:1.1} might bifurcate, in some cases, at the bifurcation points $(\s_k,0)$, $k\in\N$. In particular, consider an even weight, i.e., $a(T-t)=a(t)$ for a.e. $t\in[0,T]$, satisfying
		\begin{equation}
			\label{eq:8p}
			(a_k-3c_k)(a_ke_k-9c_k^2)>0 \quad \text{ and } \quad (e_k-3c_k)(a_ke_k-9c_k^2)>0,
		\end{equation}
		where $a_k,c_k,e_k>0$ are the constants introduced in \eqref{eq:coeff}. Observe that, since the weight is even, necessarily $b_k=d_k=0$ in this case; thus \eqref{ass:H}(iv) is not satisfied, and the results of Section \ref{sec:3.4} cannot be applied directly. Nevertheless, the analysis can be adapted as we explain now. First of all, the results of Lemmas \ref{L20} and \ref{L31} are still valid, since Theorem \ref{th:4.3} in this case gives
		\[
		\mathscr{C}_k(x,y)=\begin{pmatrix}
			x\left(a_kx^2+3c_ky^2\right) \\
			y\left(3c_k x^2+e_ky^2\right)
		\end{pmatrix},
		\]
		and direct computations show that $\mathscr{C}_k(x,y)=0$ if and only if $(x,y)=(0,0)$. Then, one can perform the change of variable \eqref{CV}, obtaining that the existence of solutions of \eqref{eq:1.1} in a neighborhood of $(\s_k,0)$ is equivalent to the existence of solutions of \eqref{eq:f} in a neighborhood of $(\s_k,0,0)$. Since $f(\s_k,z,w)=0$ in this case reads
		\begin{equation}
			\left\{
			\begin{array}{l}
				-z+z\left(a_kz^2+3c_kw^2\right)=0,\\
				-w+w\left(3c_kz^2+e_kw^2\right)=0,
			\end{array}
			\right.
		\end{equation}
		whose solutions, apart from $(0,0)$, are
		\begin{align*}
		(\bar z_{1,k},\bar w_{1,k})&=\left(0,\sqrt{\frac{1}{e_k}}\right), & (\bar z_{2,k},\bar w_{2,k})&=\left(0,-\sqrt{\frac{1}{e_k}}\right), \\ (\bar z_{3,k},\bar w_{3,k})&=\left(\sqrt{\frac{1}{a_k}},0\right), &
		(\bar z_{4,k},\bar w_{4,k})&=\left(-\sqrt{\frac{1}{a_k}},0\right), \\
		(\bar z_{5,k},\bar w_{5,k})&=\left(\sqrt{\frac{e_k-3c_k}{a_ke_k-9c_k^2}},\sqrt{\frac{a_k-3c_k}{a_ke_k-9c_k^2}}\right), & (\bar z_{6,k},\bar w_{6,k})&=\left(-\bar z_{5,k},-\bar w_{5,k}\right), \\
		(\bar z_{7,k},\bar w_{7,k})&=\left(\sqrt{\frac{e_k-3c_k}{a_ke_k-9c_k^2}},-\sqrt{\frac{a_k-3c_k}{a_ke_k-9c_k^2}}\right), & (\bar z_{8,k},\bar w_{8,k})&=\left(-\bar z_{7,k},-\bar w_{7,k}\right).
		\end{align*}
		Observe that $(\bar z_{i,k},\bar w_{i,k})$, $i\in\{5,6,7,8\}$ are real thanks to \eqref{eq:8p}. Direct computations also show that the same assumption in addition guarantees that all these are regular zeros of $f(\s_k,z,w)$. Thus, by applying the implicit function theorem as in Section \ref{sec:3.4}, we can obtain an analogue of Theorem \ref{th:local_bif} ensuring the existence of 8 local branches of nontrivial solutions bifurcating subcritically at $(\s_k,0)$.
		
		As a particular example, consider the problem
        \begin{equation}
        \label{eq:5.2}
		\left\{
		\begin{array}{l}
		-u''=\l u + \left(\mathbf{1}_{[0.3,0.5]}+\mathbf{1}_{[\pi-0.5,\pi-0.3]}\right)u^{3}, \quad \quad t\in (0,\pi),  \\[7pt]
		u(0)=u(\pi), \quad u'(0)=u'(\pi),
		\end{array}
		\right.
	\end{equation}
        for which
		\[
		a_1\sim 0.0403486,\quad c_1\sim 0.0384041,\quad e_1\sim 0.0449571.
		\]
		Thus, \eqref{eq:8p} is satisfied. Figure \ref{Fig:5.1} shows the corresponding global bifurcation diagram of solutions with $2$ nodes in $[0,\pi)$, bifurcating from $(\s_1,0)=(4,0)$. In the diagram we plot the values of $\l$, $u(0)$ and $u'(0)$, in order to univocally identify the solutions of \eqref{eq:1.1}. Such a diagram has been obtained by means of some numerical simulations that we have performed.
		\begin{figure}[h!]
			\centering
			\begin{overpic}[scale=0.35]{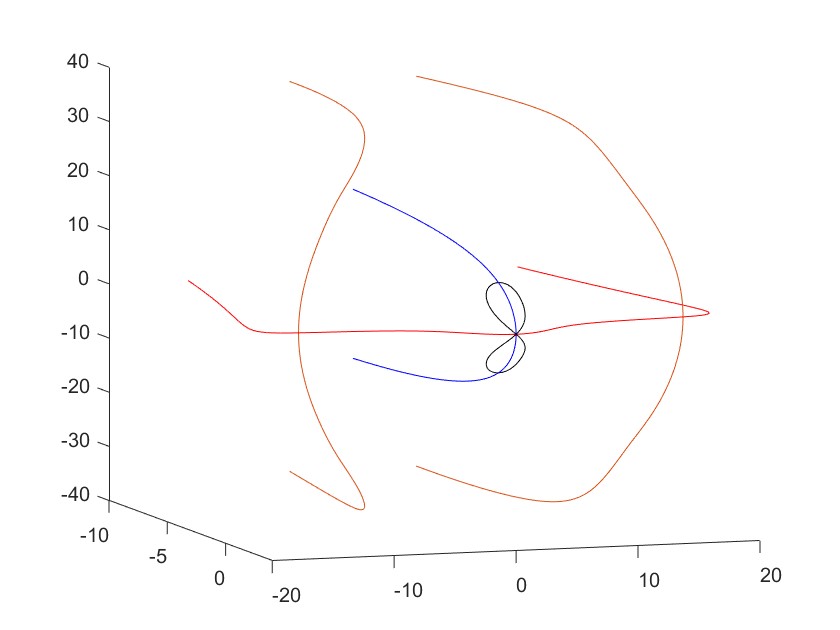}
				\put (20,5) {\small{$\l$}}
				\put (67,2) {\small{$u(0)$}}
				\put (0,68) {\small{$u'(0)$}}
			\end{overpic}
			\caption{Bifurcation diagram of solutions of \eqref{eq:5.2}. Eight branches of nontrivial solutions with 2 nodes in $[0,\pi)$ bifurcate locally from $(\s_1,0)=(4,0)$.}
			\label{Fig:5.1}
		\end{figure}
		
		The blue and red curves in Figure~\ref{Fig:5.1} are the components $\mathscr{C}_{\mathbf{\mathtt{odd}},1}$ and $\mathscr{C}_{\mathbf{\mathtt{ev}},1}$, respectively, obtained in Theorem \ref{th:glog_ev}. The black branches, instead, in a neighborhood of $(\s_1,0)$, consist of solutions that are neither even nor odd, as a consequence of the uniqueness given by Theorem \ref{CR}. Globally, they obey the analytic global alternative given in Theorem \ref{th:glob_analytic}. This example also shows that secondary bifurcations might occur on these components, like the orange components bifurcating from the red one.	
	\end{remark}
												
	\begin{remark}{\bf (Optimality of the global results and possibility of additional components).}\label{re:5.2}
	Consider the problem
    \begin{equation}
        \label{eq:5.3}
		\left\{
		\begin{array}{l}
		-u''=\l u + \left(\mathbf{1}_{[0.3,0.5]}+0.95\cdot\mathbf{1}_{[\pi-0.5,\pi-0.3]}\right)u^{3}, \quad \quad t\in (0,\pi),  \\[7pt]
		u(0)=u(\pi), \quad u'(0)=u'(\pi),
		\end{array}
		\right.
	\end{equation}
    for which
	\[
	a_1\sim 0.0393399, \quad b_1\sim 0.00095947, \quad c_1\sim 0.037444, \quad d_1\sim 0.00101251, \quad e_1\sim 0.0438331.
	\]
	Observe that the weight now is no longer even. The analysis of Section \ref{sec:3.4} can be adapted in the spirit of Remark \ref{re:5.1} to prove that 8 branches of solutions with 2 zeros in $[0,\pi)$ bifurcate locally from $(\s_1,0)=(4,0)$. The results of our numerical simulations, represented in Figure \ref{Fig:5.2}, show in addition that two pairs of such branches form, globally, a loop, while the other ones give rise to components that can be extended for all $\l<\s_1=4$. This confirms that the two alternatives given by Theorem \ref{th:glob_analytic} can occur simultaneously.
		\begin{figure}[h!]
			\centering
			\begin{overpic}[scale=0.35]{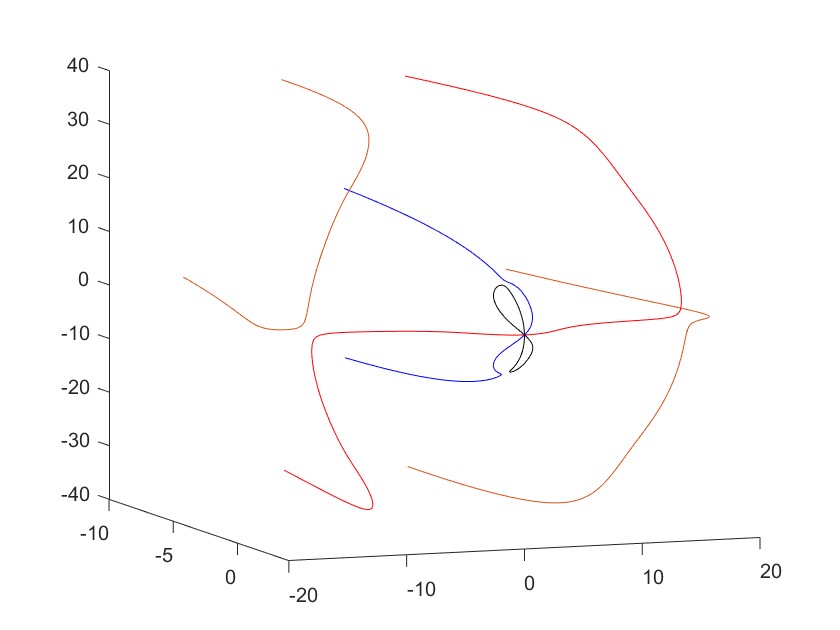}
				\put (20,5) {\small{$\l$}}
				\put (67,2) {\small{$u(0)$}}
				\put (0,68) {\small{$u'(0)$}}
			\end{overpic}
			\caption{Bifurcation diagram of solutions of \eqref{eq:5.3}. Eight branches of nontrivial solutions with 2 nodes in $[0,\pi)$ bifurcate locally from $(\s_1,0)=(4,0)$. Two pairs of such branches, plotted in black, give rise to two loops in the global bifurcation diagram, according to Theorem \ref{th:glob_analytic}.}
			\label{Fig:5.2}
		\end{figure}						
	Observe that this case can be seen as an asymmetric perturbation of the example of Remark \ref{re:5.1}. As shown in Figure \ref{Fig:5.2}, such an asymmetric situation provokes the breaking of the secondary bifurcations that were present in the symmetric case of Figure \ref{Fig:5.1}, giving rise to isolated components in the bifurcation diagram that are represented in orange in Figure \ref{Fig:5.2}. This is a typical phenomenon when asymmetric weights are considered (see \cite{CLGT}, \cite{LGT} and \cite{T} for similar examples).
	\end{remark}
											
	\begin{remark}{\bf (Subharmonic solutions).}
		For simplicity in the exposition, the paper focuses in the analysis of $T$-periodic solutions with winding number $k\geq0$ of equation 
		\begin{equation}
			\label{eq}
			-u''=\l u+a(t)u^3.
		\end{equation}
		By extending $a(t)$ in a $T$-periodic way, the same local and global study can be done analogously to look for  $nT$-periodic solutions of \eqref{eq}, with $n\geq2$, having winding number $k\geq0$. The bifurcation points to the nontrivial $nT$-periodic solutions with winding number $k$ would be 
		\[
			\s_{n,k}:=\left(\frac{2\pi k}{nT}\right)^2.
		\]
		As usual in this context, determining whether or not a $nT$-periodic solution has as minimal subharmonic order $n$, i.e., it is not a $\ell T$-periodic solution for all $\ell\in\{1,\dots,n-1\}$, is a difficult task (see for instance \cite[Introduction]{Ra2}). We can say the following:
		\begin{itemize}
			\item[(a)] If $\gcd(n,k)=1$, a straightforward coprimality argument proves that the minimal subharmonic order is $n$. Hence, by the local uniqueness of Theorem \ref{CR} and Proposition \ref{pr:locbif}, given a sufficiently small $\rho>0$, all the nontrivial solutions in $B_{\rho}(\s_{n,k},0)$ have minimal subharmonic order $n$ and winding number $k$.
			\item[(b)]  By Proposition \ref{pr:4.5}, all the global components bifurcating from $(\s_{n,k},0)$ have winding number $k$ and minimal subharmonic order $n$. Therefore, if $m/j=n/k$ but $\gcd(m,j)\neq1$, then the $mT$-periodic solutions with winding number $j$ are bounded away from any component bifurcating from $(\s_{n,k},0)$.
		\end{itemize} 
	\end{remark}

	\appendix
													
	\section{Preliminaries on nonlinear spectral theory}
	\label{A1}
	In this appendix we collect some fundamental concepts about nonlinear spectral theory that are used in this work to determine the spectrum of the operator $\p_u\mf{F}(\l,u)$. We send the interested reader to \cite{LG01,LGMC} for a detailed explanation of these topics.
													
	We start with some definitions. Consider a \emph{Fredholm operator family}, i.e., a continuous map $\mathfrak{L}\in \mathcal{C}(J,\Phi_{0}(U,V))$ where $J\subset\R$ is a bounded interval and $\Phi_{0}(U,V)$ the space of Fredholm operators of index 0 from the Banach space $U$ to the Banach space $V$. The \emph{generalized spectrum} of $\mathfrak{L}$, $\Sigma(\mathfrak{L})$,  is defined as the set of its \emph{generalized eigenvalues}, i.e.,
	\begin{equation*}
		\Sigma(\mathfrak{L}):=\{\lambda\in J\colon \mathfrak{L}(\lambda)\notin GL(U,V)\}.
	\end{equation*}
	One of the most important concepts in spectral theory is the concept of algebraic multiplicity. The classical definition of algebraic multiplicity is related to eigenvalues of compact operators in Banach spaces. In 1988, J. Esquinas and J. L\'{o}pez-G\'{o}mez \cite{Es,ELG}, inspired by the works by Krasnosel'skii \cite{Kr}, Rabinowitz \cite{Ra1} and Magnus \cite{Ma}, generalized the concept of algebraic multiplicity to every Fredholm operator family $\mf{L}\colon J \to \Phi_{0}(U,V)$, not necessarily of the form $\mathfrak{L}(\lambda)=K-\lambda I_U$, $U=V$, with $K$ compact. For $\mathfrak{L}\in \mathcal{C}^r(J,\Phi_{0}(U,V))$, we set
	$$\mathfrak{L}_{j}:=\frac{1}{j!}\mathfrak{L}^{(j)}(\lambda_{0}), \quad 1\leq j\leq r.$$
	  The following pivotal concept, introduced by Esquinas and L\'{o}pez-G\'{o}mez in \cite{ELG}, generalizes the transversality condition \eqref{ii.4} of Crandall and Rabinowitz \cite{CR}.
	\begin{definition}
		\label{de2.3}
		Let $\mathfrak{L}\in \mathcal{C}^{r}([a,b],\Phi_{0}(U,V))$ and $1\leq \kappa \leq r$. Then, $\lambda_{0}\in \Sigma(\mathfrak{L})$ is said to be a \emph{$\kappa$-transversal eigenvalue} of $\mathfrak{L}$ if
		\begin{equation*}
			\bigoplus_{j=1}^{\kappa}\mathfrak{L}_{j}\left(\bigcap_{i=0}^{j-1}N(\mathfrak{L}_{i})\right)
			\oplus R(\mathfrak{L}_{0})=V \quad \text{ with } \quad \mathfrak{L}_{\kappa}\left(\bigcap_{i=0}^{\kappa-1}N(\mathfrak{L}_{i})\right)\neq \{0\}.
		\end{equation*}
	\end{definition}
	For these eigenvalues, the following \emph{generalized algebraic multiplicity} was introduced in \cite{ELG}:
	\begin{equation}
		\label{ii.3}
		\chi[\mathfrak{L}, \lambda_{0}] :=\sum_{j=1}^{\kappa}j\cdot \dim \mathfrak{L}_{j}\left(\bigcap_{i=0}^{j-1}N[\mathfrak{L}_{i}]\right).
	\end{equation}
	In particular, when $N[\mf{L}_0]=\mathrm{span}\{\phi_0\}$ for some $\phi_0\in U\setminus\{0\}$ such that $\mf{L}_1\phi_0\notin R[\mf{L}_0]$, then
	\begin{equation}
		\label{ii.4}
		\mf{L}_1(N[\mf{L}_0])\oplus R[\mf{L}_0]=V
	\end{equation}
	and hence, $\l_0$ is a 1-transversal eigenvalue of $\mf{L}(\l)$ with $\chi[\mf{L},\l_0]=1$. More generally, if condition \eqref{ii.4} holds, then
	\begin{equation}
		\label{CTCR}
		\chi[\mf{L},\l_0]=\dim N[\mf{L}_0].
	\end{equation}
	One can prove that $\chi$ equals the classical algebraic multiplicity in the classical setting, when $\mathfrak{L}(\lambda)= \lambda I_{U} - K$,
	for some compact operator $K\colon U\to U$. Indeed, for $\lambda_0\in \s(K)$ one has
	\begin{equation*}
		\label{1.1.90}
		\chi [\lambda I_U-K,\lambda_{0}]=\mathrm{dim\,}N[(\l_0 I_{U}-K)^{\nu(\l_0)}],
	\end{equation*}
	where $\nu(\l_0)$ is the \emph{algebraic ascent} of $\l_0$, i.e., the least integer, $\nu\geq 1$, such that
	\[
		N[(\l_0 I_{U}-K)^{\nu}]=N[(\l_0 I_{U}-K)^{\nu+1}].
	\]
	
	\section{Preliminaries on bifurcation theory}\label{app:B}
	In this section we review some abstract results on local and global bifurcation theory that are used in this paper to analyze equation \eqref{eq:1.1}.
													
	\subsection{The Crandall--Rabinowitz theorem}\label{app:B.1}
	Here we recall the celebrated Crandall--Rabinowitz bifurcation theorem. It was stated and proved in \cite{CR}. Let $(U,V)$ be a pair of real Banach spaces such that $U\hookrightarrow V$ and consider a map $\mf{F}\colon\R\times U\to V$ of class $\mc{C}^{r}$, $r\geq 2$, satisfying the following assumptions:
	\begin{enumerate}
		\item[(F1)] $\mathfrak{F}(\lambda,0)=0$ for every $\lambda\in\mathbb{R}$;
		\item[(F2)] $\mf{L}(\l):=\p_{u}\mathfrak{F}(\lambda,0)\in\Phi_{0}(U,V)$ for each $\lambda\in\mathbb{R}$;
		\item[(F3)] $N[\mf{L}(\l_{0})]=\mathrm{span}\{\phi_{0}\}$ for some $\phi_0\in U\setminus\{0\}$ and $\l_{0}\in \R$.
	\end{enumerate}
	The theorem reads as follows.
	\begin{theorem}
		\label{Cr-Rb} 
		Let $r\geq 2$, $\mathfrak{F}\in\mc{C}^{r}(\R\times U,V)$ be a map satisfying conditions {\rm (F1)--(F3)} and $\l_{0}\in\Sigma(\mf{L})$ be a $1$-transversal eigenvalue of $\mf{L}$, that is, 
		$$\mf{L}_{1}(N[\mf{L}_{0}])\oplus R[\mf{L}_{0}]=V,$$
		where $\mf{L}_{1}(\l):=\tfrac{d}{d\l}\mf{L}(\l)$. Let $Y\subset U$ be a subspace such that $N[\mathfrak{L}_0] \oplus Y = U$. Then, there exist $\e>0$ and two maps of class $\mc{C}^{r-1}$,
		$$
			\l \colon (-\e,\e) \longrightarrow \R, \qquad \G\colon (-\e,\e)\longrightarrow Y,
		$$
		such that $\l(0)=\l_0$, $\G(0)=0$, and, for each $s\in(-\e,\e)$,
		\begin{equation}
			\mathfrak{F}(\l(s),u(s))=0, \qquad u(s):= s(\phi_0+\G(s)).
			\label{(2.2.1)}
		\end{equation}
		Moreover, there exists  $\r >0$ such that, if $\mathfrak{F}(\l,u)=0$ and $(\l,u)\in B_\r(\l_0,0)$, then either $u = 0$ or $(\l,u)=(\l(s),u(s))$ for some $s\in(-\e,\e)$. In addition, if $\mathfrak{F}$ is real analytic, so are $\l(s)$ and $u(s)$.
	\end{theorem}
			
	\subsection{Global bifurcation theory: eigenvalues with odd multiplicity}
	\label{app:B.3}
	In this subsection we present the global alternative theorem for the bifurcation of nonlinear Fredholm operators of index zero given by L\'{o}pez-G\'{o}mez and Mora-Corral in \cite{LG,LGMC1,LGMC} and sharpened in \cite{JJ,JJ2,JJ3,JJ5} for the case of the Fitzpatrick, Pejsachowicz and Rabier topological degree \cite{FPRa, FPRb, PR}. It is important to notice that it applies to the case of bifurcation points where the generalized algebraic multiplicity is odd. Throughout this section, given a pair of Banach spaces $(U,V)$, we consider a operator $\mathfrak{F}\colon\mathbb{R}\times U\to V$ of class $\mc{C}^{1}$ such that
	\begin{enumerate}
		\item[(F1)] $\mf{F}$ is orientable in the sense of Fitzpatrick, Pejsachowicz and Rabier (see \cite{FPRb, PR});
		\item[(F2)] $\mathfrak{F}(\lambda,0)=0$ for all $\lambda\in\mathbb{R}$;
		\item[(F3)] $\p_{u}\mathfrak{F}(\lambda,u)\in\Phi_{0}(U,V)$ for every $\lambda\in\mathbb{R}$ and $u\in U$;
		\item[(F4)] $\mathfrak{F}$ is proper on bounded and closed subsets of $\mathbb{R}\times U$;
		\item[(F5)] $\Sigma(\mf{L})$ is a discrete subset of $\R$.
	\end{enumerate}
	The \textit{trivial branch} is the subset 
	$$\mc{T}:=\{(\lambda,0)\colon \l\in \R\}\subset \R\times U,$$
	while the \textit{set of non-trivial solutions} is defined as the closed set
	\begin{equation*}
		\mf{S}=\left( \mf{F}^{-1}(0)\setminus \mc{T}\right)\cup \{(\l,0)\colon\l\in \Sigma(\mf{L})\}.
	\end{equation*}
	Moreover, we say that $\mathfrak{C}\subset \mf{S}$ is a \textit{component} of $\mf{S}$ if it is a closed and connected subset of $\mf{S}$ and it is maximal for the inclusion, i.e., if it is a connected component of $ \mf{S}$.

	\begin{theorem}[\textbf{Global alternative}]
		\label{TGB}
		Let $\mf{F}\in\mc{C}^{1}(\R\times U, V)$ satisfy {\rm(F1)--(F5)},  and let $\mf{C}$ be a component of the set of non-trivial solutions $\mf{S}$ such that $(\l_{0},0)\in\mf{C}$ for some $\l_0\in\Sigma(\mf{L})$. If $\mf{L}(\l)$, $\l\in \R$, is analytic, and the generalized algebraic multiplicity $\chi[\mf{L},\l_0]$ is odd, then one of the following non-excluding alternatives occur:
		\begin{itemize}
			\item[\rm(i)] $\mf{C}$ is unbounded in $\R\times U$;
			\item[\rm(ii)] there exists $\l_{1}\in \Sigma(\mf{L})$, $\l_{1}\neq\l_{0}$, such that $(\l_{1},0)\in \mf{C}$.
		\end{itemize}
	\end{theorem}
										
\subsection{Analytic global bifurcation theory}

In this final subsection, we state another global bifurcation theorem, for the case of analytic nonlinearities. It goes back to Dancer \cite{Da73,Da732,Da} (see also Buffoni--Tolland \cite{BT}). These results where related to the special case of $1$-transversal eigenvalues, where the theorem of Crandall and Rabinowitz applies. Recently, in \cite{JJ5}, the global alternative was generalized to the setting we present here.
\par Throughout this section, given a pair $(U,V)$ of real Banach spaces, we consider an analytic operator $\mathfrak{F}\in\mathcal{C}^{\omega}(\mathbb{R}\times U,V)$ satisfying the following properties:
\begin{enumerate}
	\item[(F1)] $\mathfrak{F}(\lambda,0)=0$ for all $\lambda\in\mathbb{R}$;
	\item[(F2)] $\partial_{u}\mathfrak{F}(\lambda,u)\in\Phi_{0}(U,V)$ for all $(\lambda,u)\in\mathbb{R}\times U$;
	\item[(F3)] $\mf{F}$ is proper on closed and bounded subsets of $\R\times U$.
\end{enumerate}
Given an analytic nonlinearity $\mf{F}\in\mc{C}^{\o}(\mathbb{R}\times U,V)$ satisfying conditions {\rm (F1)--(F3)}, it is said that $(\lambda,u)\in\mathfrak{F}^{-1}(0)$ is a \emph{regular point} of $\mf{F}$ if $\partial_{u}\mathfrak{F}(\lambda,u)\in GL(U,V)$. The set of regular points of $\mf{F}$ will be denoted by $\mc{R}(\mf{F})$.
\begin{theorem}
			\label{th:glob_analytic1}
			Let $\mathfrak{F}\in\mathcal{C}^{\omega}(\mathbb{R}\times U,V)$ be an analytic map satisfying {\rm (F1)--(F3)}. Suppose that there exist $\varepsilon>0$, $\l_{0}\in\R$ and an analytic injective curve $\gamma:(0,\varepsilon)\to \R\times U$ such that  $\gamma((0,\varepsilon))\subset \mf{F}^{-1}(0)\backslash \mc{T}$, $\lim_{s\da 0}\gamma(s)=(\l_{0},0)$, and $\gamma((0,\varepsilon))\subset \mc{R}(\mf{F})$. Then, there exists a locally injective continuous path $\G: \R_{> 0}\to \R\times U$, with $\Gamma(\R_{> 0})\subset \mf{F}^{-1}(0)$,  for which there exists $0<\delta<\varepsilon$ satisfying $\Gamma(s)=\gamma(s)$ for all $s\in (0,\delta)$. Moreover, $\Gamma$ satisfies one of the following non-excluding alternatives:
			\begin{enumerate}
				\item[{\rm (a)}]   $\lim_{s\to +\infty} \|\Gamma(s)\|_{\mathbb{R}\times U}= +\infty$;
				\item[{\rm (b)}]   $\Gamma$ is a closed loop, i.e., there exists $S>0$ such that $\Gamma(S)=(\l_0,0)$.
			\end{enumerate}
\end{theorem}
\noindent
In \cite[Th. 7.2]{JJ5}, Theorem \ref{th:glob_analytic1} was stated and proved under the additional hypothesis 
\begin{enumerate}
	\item[] $\l_{0}\in\Sigma(\mf{L})$ is an isolated eigenvalue such that $N[\mf{L}_{0}]=\mathrm{span}\{\phi_{0}\}$ for some $\phi_0\in U\backslash\{0\}$.
\end{enumerate}
However, the proof of \cite[Th. 7.2]{JJ5} can be directly adapted to the case of Theorem \ref{th:glob_analytic1}, also for the case of non-simple eigenvalues, independently of their multiplicity.

Observe that the global alternatives provided by Theorems \ref{TGB} and \ref{th:glob_analytic1} are independent. Indeed, if the connected component of $\mf{F}^{-1}(0)$ bifurcating from $(\l_{0},0)$, 
$\mf{C}$,  is bounded, then, according to Theorem \ref{th:glob_analytic1}, $\mf{F}^{-1}(0)$ contains a closed loop. But this does not necessarily entail the existence of some $(\l_1,0)\in\mf{C}$ with $\l_1\neq \l_{0}$, as guaranteed by Theorem \ref{TGB} when, in addition, $\chi$ is odd. Conversely, when $\mf{C}$ is bounded and $\chi$ is odd, then, owing to Theorem \ref{TGB}, $(\l_1,0)\in \mf{C}$ for some $\l_1\neq \l_{0}$, though this does not entail that any local analytic curve bifurcating from $(\l_{0},0)$ can be continued to a global closed loop.
													
	\section{Computation of the derivatives of $\psi$, $h_1$ and $h_2$}
	\label{AC}
													
	In this appendix, we start by detailing the computation of the derivatives at $(\s_{k},0)$ of the function $\psi\colon I_{k}\times \mc{N}_{k}\to N[\mf{L}(\s_{k})]^{\perp}$ satisfying \eqref{E12}. To obtain them, we will make use of the following derivatives of $\mf{F}$ evaluated at $(\s_k,0)$, which are given in Lemma \ref{L1}:
	\begin{align*}
		\partial_{\l}\mf{F}(\s_{k},0)\left[\mu\right]&=0, & \partial_{\l\l}\mf{F}(\s_{k},0)\left[\mu_{1},\mu_{2}\right]&=0, & \partial_{\l u}\mf{F}(\s_{k},0)\left[\mu,v\right]&=\mu v, \\
		\partial_{u}\mf{F}(\s_{k},0)\left[v\right]&=\mf{L}(\s_{k})\left[v\right], & \partial_{uu}\mf{F}(\s_{k},0)\left[v_{1},v_{2}\right]&=0,  & \partial_{uuu}\mf{F}(\s_{k},0)\left[v_{1},v_{2},v_{3}\right]&=6a(t)v_{1}v_{2}v_{3}.
	\end{align*}
													
	\begin{lemma}
		\label{LC1}
		The partial derivatives of $\psi\colon I_{k}\times \mc{N}_{k}\to N[\mf{L}(\s_{k})]^{\perp}$ at $(\s_{k},0)$ are given by
		\begin{gather}
			\label{E15}
			\partial_{x}\psi(\s_{k},0)[v_{1}]=0, \quad \partial_{xx}\psi(\s_{k},0)[v_{1},v_{2}]=0, \quad \text{for each $v_{1},v_{2}\in N[\mf{L}(\s_{k})]$},\\
			\label{E20}
			\partial_{\l}\psi(\s_{k},0)[\mu_{1}]=0, \quad \partial_{\l\l}\psi(\s_{k},0)[\mu_{1},\mu_{2}]=0, \quad \text{for each $\mu_{1}, \mu_{2} \in \R$},\\
			\label{E21}
			\partial_{\l x}\psi(\s_{k},0)[\mu,v]=0, \quad \text{for each $(\mu,v)\in\R\times N[\mf{L}(\s_{k})]$}.
		\end{gather}
	\end{lemma}

	\begin{proof}
		We differentiate the implicit expression \eqref{E12} with respect to the variable $x$ to obtain
		\begin{align}
			\label{E13}
			\left(I_{L^{2}}-Q_{k}\right)\partial_{u}\mf{F}(\l,x+\psi(\l,x))\left[I_{N_{k}}+\partial_{x}\psi(\l,x)\right]=0, \quad (\l,x)\in I_{k}\times \mc{N}_{k},
		\end{align}
		where we have denoted $N_{k}:= N[\mf{L}(\s_{k})]$. Evaluating \eqref{E13} at $(\s_{k},0)$ gives
		\begin{equation}
			\label{E14}
			\left(I_{L^{2}}-Q_{k}\right)\mf{L}(\s_{k})\left[I_{N_{k}}+\partial_{x}\psi(\s_{k},0)\right]=0,
		\end{equation}
		since $\psi(\s_{k},0)=0$. As $\left(I_{L^{2}}-Q_{k}\right)\mf{L}(\s_{k})=\mf{L}(\s_{k})$ and $\mf{L}(\s_{k})[N_{k}]=0$, \eqref{E14} reduces to
		$$\mf{L}(\s_{k})[\partial_{x}\psi(\s_{k},0)]=0.$$
		Hence, $\partial_{x}\psi(\s_{k},0)[v]\in N[\mf{L}(\s_{k})]$ for each $v\in N[\mf{L}(\s_{k})]$. Consequently,
		$$\partial_{x}\psi(\s_{k},0)[v]\in N[\mf{L}(\s_{k})]\cap N[\mf{L}(\s_{k})]^{\perp}=\{0\}.$$
		Therefore $\partial_{x}\psi(\s_{k},0)[v]=0$ for all $v\in N[\mf{L}(\s_{k})]$. This proves the first identity in \eqref{E15}. 
		To prove the second identity, we differentiate equation \eqref{E13} with respect to $x$ again, obtaining
		\begin{equation}
			\label{EO}
			\begin{aligned}
				& \left(I_{L^{2}}-Q_{k}\right)\partial_{uu}\mf{F}(\l,x+\psi(\l,x))\left[I_{N_{k}}+\partial_{x}\psi(\l,x),I_{N_{k}}+\partial_{x}\psi(\l,x)\right] \\
				& +\left(I_{L^{2}}-Q_{k}\right)\partial_{u}\mf{F}(\l,x+\psi(\l,x))\left[\partial_{xx}\psi(\l,x)\right]=0, \quad (\l,x)\in I_{k}\times \mc{N}_{k}.
			\end{aligned}
		\end{equation}
		By evaluating this expression at $(\s_{k},0)$, using that $\psi(\s_{k},0)=0$ and $\partial_{x}\psi(\s_{k},0)=0$, we deduce
		\begin{equation*}
			\left(I_{L^{2}}-Q_{k}\right)\partial_{uu}\mf{F}(\s_{k},0)\left[I_{N_{k}},I_{N_{k}}\right] 
			+\left(I_{L^{2}}-Q_{k}\right)\mf{L}(\s_{k})\left[\partial_{xx}\psi(\s_{k},0)\right]=0.
		\end{equation*}
		As $\partial_{uu}\mf{F}(\s_{k},0)=0$ and $\left(I_{L^{2}}-Q_{k}\right)\mf{L}(\s_{k})=\mf{L}(\s_{k})$, we get
		\begin{equation*}
			\mf{L}(\s_{k})[\partial_{xx}\psi(\s_{k},0)]=0.
		\end{equation*}
		Therefore, we conclude that
		\begin{equation*}
			\partial_{xx}\psi(\s_{k},0)[v_{1},v_{2}]\in N[\mf{L}(\s_{k})]\cap N[\mf{L}(\s_{k})]^{\perp}=\{0\}, 
		\end{equation*}
		for each $v_{1},v_{2}\in N[\mf{L}(\s_{k})]$, which shows the second identity of \eqref{E15}. Now, we differentiate the implicit expression \eqref{E12} with respect to the variable $\l$ to obtain, for all $(\l,x)\in I_{k}\times \mc{N}_{k}$,
		\begin{equation*}
			\left(I_{L^{2}}-Q_{k}\right)\partial_{\l}\mf{F}(\l,x+\psi(\l,x))+\left(I_{L^{2}}-Q_{k}\right)\partial_{u}\mf{F}(\l,x+\psi(\l,x))\left[\partial_{\l}\psi(\l,x)\right]=0.
		\end{equation*}
		By evaluating this expression at $(\s_{k},0)$, we obtain
		\begin{align*}
			(I_{L^{2}}-Q_{k})\partial_{\l}\mf{F}(\s_{k},0)+(I_{L^{2}}-Q_{k})\mf{L}(\s_{k})[\partial_{\l}\psi(\s_{k},0)]=0.
		\end{align*}
		As $\partial_{\l}\mf{F}(\s_{k},0)=0$ and $(I_{L^{2}}-Q_{k})\mf{L}(\s_{k})=\mf{L}(\s_{k})$, this expression reduces to
		\begin{align*}
			\mf{L}(\s_{k})[\partial_{\l}\psi(\s_{k},0)]=0,
		\end{align*}
		and, once again, 
		\begin{equation*}
			\partial_{\l}\psi(\s_{k},0)\left[\mu\right]\in N[\mf{L}(\s_{k})]\cap N[\mf{L}(\s_{k})]^{\perp}=\{0\}, 
		\end{equation*}
		for every $\mu\in\R$. We differentiate again with respect to $\l$ to obtain
		\begin{multline*}
			\left(I_{L^{2}}-Q_{k}\right)\partial_{\l\l}\mf{F}(\l,x+\psi(\l,x))+2 \left(I_{L^{2}}-Q_{k}\right)\partial_{\l u}\mf{F}(\l,x+\psi(\l,x))\left[\partial_{\l}\psi(\l,x)\right] \\
			+\left(I_{L^{2}}-Q_{k}\right)\partial_{uu}\mf{F}(\l,x+\psi(\l,x))\left[\partial_{\l}\psi(\l,x),\partial_{\l}\psi(\l,x)\right]\\
			+\left(I_{L^{2}}-Q_{k}\right)\partial_{u}\mf{F}(\l,x+\psi(\l,x))\left[\partial_{\l\l}\psi(\l,x)\right]=0,
		\end{multline*}
		for each $(\l,x)\in I_{k}\times \mc{N}_{k}$. Evaluating this expression at $(\s_{k},0)$ gives
		\begin{align*}
			\left(I_{L^{2}}-Q_{k}\right)\mf{L}(\s_{k})\left[\partial_{\l\l}\psi(\s_{k},0)\right]=0.
		\end{align*}
		Therefore, reasoning as above, 
		\begin{equation*}
			\partial_{\l\l}\psi(\s_{k},0)\left[\mu_1,\mu_2\right]\in N[\mf{L}(\s_{k})]\cap N[\mf{L}(\s_{k})]^{\perp}=\{0\}, 
		\end{equation*}
		for all $\mu_{1},\mu_{2}\in\R$. Finally, to prove \eqref{E21}, we differentiate \eqref{E13} with respect to $\l$ to obtain 
		\begin{align*}
			& \left(I_{L^{2}}-Q_{k}\right)\partial_{\l u}\mf{F}(\l,x+\psi(\l,x))\left[I_{N_{k}}+\partial_{x}\psi(\l,x)\right]  \\
			&+ \left(I_{L^{2}}-Q_{k}\right)\partial_{uu}\mf{F}(\l,x+\psi(\l,x))\left[I_{N_{k}}+\partial_{x}\psi(\l,x),\partial_{\l}\psi(\l,x)\right] \\
			& + \left(I_{L^{2}}-Q_{k}\right)\partial_{u}\mf{F}(\l,x+\psi(\l,x))\left[\partial_{\l x}\psi(\l,x)\right]=0
		\end{align*}
		for all $(\l,x)\in I_{k}\times \mc{N}_{k}$. 	Evaluating this expression at $(\s_{k},0)$ leads to
		\begin{align*}
			(I_{L^{2}}-Q_{k})|_{N_{k}} +\left(I_{L^{2}}-Q_{k}\right)\mf{L}(\s_{k})\left[\partial_{\l x}\psi(\s_{k},0)\right]=0.
		\end{align*}
		As $(I_{L^{2}}-Q_{k})(\varphi_{k})=(I_{L^{2}}-Q_{k})(\phi_{k})=0$ and $(I_{L^{2}}-Q_{k})\mf{L}(\s_{k})=\mf{L}(\s_{k})$, we have
		\begin{align*}
			\mf{L}(\s_{k})\left[\partial_{\l x}\psi(\s_{k},0)\right]=0.
		\end{align*}
		Therefore, once again, 
		\begin{equation*}
			\partial_{\l x}\psi(\s_{k},0)\left[\mu,v\right]\in N[\mf{L}(\s_{k})]\cap N[\mf{L}(\s_{k})]^{\perp}=\{0\}, 
		\end{equation*}
		for all $(\mu,v)\in\R\times N[\mf{L}(\s_{k})]$. This concludes the proof.
	\end{proof}
													
	Next, we detail the computation of the partial derivatives of the functions $h_{1}, h_{2}\colon I_{k}\times \O_{k}\subset \R\times \R^{2}\to \R$, introduced in \eqref{E16}, at $(\s_k,0,0)$.
	\begin{proposition}
		For all $i,j\in\{1,2\}$, the following identities hold:
		\begin{align}
			\partial_{\l}h_{i}(\s_{k},0,0)&=0, & \partial_{x_j}h_{i}(\s_{k},0,0)&=0, & \partial_{\l x_2}h_{1}(\s_{k},0,0)&=0, & \partial_{x_1 x_2}h_{i}(\s_{k},0,0)&=0,\\
			\partial_{\l\l}h_{i}(\s_{k},0,0)&=0, & \partial_{x_jx_j}h_{i}(\s_{k},0,0)&=0, & \partial_{\l x_1}h_{2}(\s_{k},0,0)&=0, & \partial_{\l x_i}h_{i}(\s_{k},0,0)&=1.
		\end{align}
	\end{proposition}
	
	\begin{proof}
		We prove the results for $h_{1}$, since the computations for $h_{2}$ are similar. Recall that
		\begin{align*}
			h_{1}(\l,x_1,x_2)&=\left(Q_{k}\mf{F}(\l,L(x_1,x_2)+\psi(\l,L(x_1,x_2))),\varphi_{k}\right)_{L^{2}} \\
			&=\left(Q_{k}\mf{F}(\l,x_1\varphi_{k}+x_2\phi_{k}+\psi(\l,x_1\varphi_{k}+x_2\phi_{k})),\varphi_{k}\right)_{L^{2}}.
		\end{align*}
		To simplify the notation, we set $z:=L(x_1,x_2)$. Therefore, differentiating this expression gives
		\begin{align*}
			\partial_{\l}h_{1}(\l,x_1,x_2)&=\left(Q_{k}\partial_{\l}\mf{F}(\l,z+\psi(\l,z)),\varphi_{k}\right)_{L^{2}}+\left(Q_{k}\partial_{u}\mf{F}(\l,z+\psi(\l,z))\left[\partial_{\l}\psi(\l,z)\right],\varphi_{k}\right)_{L^{2}}, \\
			\partial_{x_1}h_{1}(\l,x_1,x_2)&=\left(Q_{k}\partial_{u}\mf{F}(\l,z+\psi(\l,z))\left[\varphi_{k}+\partial_{x}\psi(\l,z)\left[\varphi_{k}\right]\right],\varphi_{k}\right)_{L^{2}}, \\
			\partial_{x_2}h_{1}(\l,x_1,x_2)&=\left(Q_{k}\partial_{u}\mf{F}(\l,z+\psi(\l,z))\left[\phi_{k}+\partial_{x}\psi(\l,z)\left[\phi_{k}\right]\right],\varphi_{k}\right)_{L^{2}}.
		\end{align*}
		By evaluating these derivatives at $(\s_{k},0,0)$ and using \eqref{E15} and \eqref{E20}, we obtain
		\begin{align*}
			\partial_{\l}h_{1}(\s_{k},0,0)&=\left(0,\varphi_{k}\right)_{L^{2}}=0, \\
			\partial_{x_1}h_{1}(\s_{k},0,0)&=\left(Q_{k}\mf{L}(\s_{k})[\varphi_{k}],\varphi_{k}\right)_{L^{2}}=0, \\
			\partial_{x_2}h_{1}(\s_{k},0,0)&=\left(Q_{k}\mf{L}(\s_{k})[\phi_{k}],\varphi_{k}\right)_{L^{2}}=0,
		\end{align*}
		since $\varphi_{k}$ and $\phi_{k}$ belong to $N[\mf{L}(\s_{k})]$. By differentiating again, we get
		\begin{align}
		\partial_{\l\l}h_{1}(\l,x_1,x_2) &= \left(Q_{k}\partial_{\l\l}\mf{F}(\l,z+\psi(\l,z)),\varphi_{k}\right)_{L^{2}}+2\left(Q_{k}\partial_{\l u}\mf{F}(\l,z+\psi(\l,z))[\partial_{\l}\psi(\l,z)],\varphi_{k}\right)_{L^{2}} \nonumber \\
		&\;\;\;\; +\left(Q_{k}\partial_{uu}\mf{F}(\l,z+\psi(\l,z))[\partial_{\l}\psi(\l,z),\partial_{\l}\psi(\l,z)],\varphi_{k}\right)_{L^{2}} \nonumber \\
		&\;\;\;\; +\left(Q_{k}\partial_{u}\mf{F}(\l,z+\psi(\l,z))[\partial_{\l\l}\psi(\l,z)],\varphi_{k}\right)_{L^{2}}, \nonumber\\[1pt]
		\partial_{x_1x_1}h_{1}(\l,x_1,x_2)&=  \left(Q_{k}\partial_{uu}\mf{F}(\l,z+\psi(\l,z))[\varphi_{k}+\partial_{x}\psi(\l,z)[\varphi_{k}],\varphi_{k}+\partial_{x}\psi(\l,z)[\varphi_{k}]],\varphi_{k}\right)_{L^{2}} \label{eq:h1_x1x1}\\
		&\;\;\;\; + \left(Q_{k}\partial_{u}\mf{F}(\l,z+\psi(\l,z))[\partial_{xx}\psi(\l,z)[\varphi_{k},\varphi_{k}]],\varphi_{k}\right)_{L^{2}},\\[1pt]
		\partial_{x_2x_2}h_{1}(\l,x_1,x_2)&= \left(Q_{k}\partial_{uu}\mf{F}(\l,z+\psi(\l,z))[\phi_{k}+\partial_{x}\psi(\l,z)[\phi_{k}],\phi_{k}+\partial_{x}\psi(\l,z)[\phi_{k}]],\varphi_{k}\right)_{L^{2}} \\
		&\;\;\;\; + \left(Q_{k}\partial_{u}\mf{F}(\l,z+\psi(\l,z))[\partial_{xx}\psi(\l,z)[\phi_{k},\phi_{k}]],\varphi_{k}\right)_{L^{2}}.
		\end{align}
		By evaluating these derivatives at $(\s_{k},0,0)$ and using \eqref{E15} and \eqref{E20} again, we obtain the desired conclusion. Finally, the mixed derivatives are given by
		\begin{align*}
			\partial_{\l x_1}h_{1}(\l,x_1,x_2)  & =  \left(Q_{k}\partial_{\l u}\mf{F}(\l,z+\psi(\l,z))[\varphi_{k}+\partial_{x}\psi(\l,z)[\varphi_{k}]],\varphi_{k}\right)_{L^{2}} \\
			&\;\;\;\; + \left(Q_{k}\partial_{uu}\mf{F}(\l,z+\psi(\l,z))[\varphi_{k}+\partial_{x}\psi(\l,z)[\varphi_{k}],\partial_{\l}\psi(\l,z)],\varphi_{k}\right)_{L^{2}} \\
			&\;\;\;\; + \left(Q_{k}\partial_{u}\mf{F}(\l,z+\psi(\l,z))[\partial_{\l x}\psi(\l,z)[\varphi_{k}]],\varphi_{k}\right)_{L^{2}},\\[1pt]
			\partial_{\l x_2}h_{1}(\l,x_1,x_2) & = \left(Q_{k}\partial_{\l u}\mf{F}(\l,z+\psi(\l,z))[\phi_{k}+\partial_{x}\psi(\l,z)[\phi_{k}]],\varphi_{k}\right)_{L^{2}} \\
			&\;\;\;\; + \left(Q_{k}\partial_{uu}\mf{F}(\l,z+\psi(\l,z))[\phi_{k}+\partial_{x}\psi(\l,z)[\phi_{k}],\partial_{\l}\psi(\l,z)],\varphi_{k}\right)_{L^{2}} \\
			&\;\;\;\; + \left(Q_{k}\partial_{u}\mf{F}(\l,z+\psi(\l,z))[\partial_{\l x}\psi(\l,z)[\phi_{k}]],\varphi_{k}\right)_{L^{2}},\\[1pt]
			\partial_{x_1 x_2}h_{1}(\l,x_1,x_2) &=  \left(Q_{k}\partial_{uu}\mf{F}(\l,z+\psi(\l,z))[\varphi_{k}+\partial_{x}\psi(\l,z)[\varphi_{k}],\phi_{k}+\partial_{x}\psi(\l,z)[\phi_{k}]],\varphi_{k}\right)_{L^{2}}, \\
			&\;\;\;\; + \left(Q_{k}\partial_{u}\mf{F}(\l,z+\psi(\l,z))[\partial_{xx}\psi(\l,z)[\phi_{k},\varphi_{k}]],\varphi_{k}\right)_{L^{2}}.
		\end{align*}
		By evaluating these derivatives at $(\s_{k},0,0)$ and using \eqref{E15}, \eqref{E20} and \eqref{E21}, we obtain 
		\begin{align*}
			\partial_{\l x_1}h_{1}(\s_{k},0,0)  &=\left(Q_{k}\varphi_{k},\varphi_{k}\right)_{L^{2}}=\left(\varphi_{k},\varphi_{k}\right)_{L^{2}}=1, \\
			\partial_{\l x_2}h_{1}(\s_{k},0,0)&=\left(Q_{k}\phi_{k},\varphi_{k}\right)_{L^{2}}=\left(\phi_{k},\varphi_{k}\right)_{L^{2}}=0, \\
			\partial_{x_1 x_2}h_{1}(\s_{k},0,0)&=0,
		\end{align*}
		and the proof is complete.
	\end{proof}
													
	\begin{proposition}
		The third order partial derivatives of $h_{1}, h_{2}\colon I_{k}\times \O_k\to\R$ at $(\s_k,0,0)$ are given by 
		\begin{equation}
			\label{E27}
			\begin{aligned}
				\partial_{x_1x_1x_1}h_1(\s_k,0,0)&=\left(6 a(t)\varphi_k^3,\varphi_k\right)_{L^2}, & \partial_{x_2x_2x_2}h_1(\s_k,0,0)&=\left(6a(t)\phi_k^3,\varphi_k\right)_{L^2},  \\
				\partial_{x_1x_1x_1}h_2(\s_k,0,0)&=\left(6a(t)\varphi_k^3,\phi_k\right)_{L^2}, &
				\partial_{x_2x_2x_2}h_{2}(\s_k,0,0)&=\left(6a(t)\phi_k^3,\phi_k\right)_{L^2}.
			\end{aligned}
		\end{equation}
		Moreover, the mixed derivatives are given by
		\begin{equation}
			\begin{aligned}
				\label{E29}
				\partial_{x_1x_1x_2} h_1 (\s_k,0,0)&=\left(6a(t) \varphi^2_k \phi_k,\varphi_k\right)_{L^2},  &\partial_{x_1x_2x_2} h_1 (\s_k,0,0)&=\left(6a(t) \varphi_k \phi^2_k,\varphi_k\right)_{L^2}, \\
				\partial_{x_1x_1x_2} h_2 (\s_k,0,0)&=\left(6a(t) \varphi^2_k \phi_k,\phi_k\right)_{L^2}, &\partial_{x_1x_2x_2} h_2 (\s_k,0,0)&=\left(6a(t) \varphi_k \phi^2_k,\phi_k\right)_{L^2}.
															\end{aligned}
		\end{equation}
		Finally, for all $i,j,\ell\in\{1,2\}$,
		\begin{equation}
			\label{E30}
			\partial_{\l x_jx_\ell}h_i(\s_k,0,0)=0, \qquad \partial_{\l \l x_j}h_i(\s_k,0,0)=0, \qquad \partial_{\l \l\l}h_i(\s_k,0,0)=0.
		\end{equation}
	\end{proposition}
													
	\begin{proof}
		We will proof the results for $h_{1}$, as the computations for $h_{2}$ are analogous. Differentiating \eqref{eq:h1_x1x1} with respect to $x_1$  gives
		\begin{align*}
			& \partial_{x_1x_1x_1} h_1 (\l,x_1,x_2) \\
			&= \left(Q_{k}\partial_{uuu}\mf{F}(\l,z+\psi(\l,z))[\varphi_{k}+\partial_{x}\psi(\l,z)[\varphi_{k}],\varphi_{k}+\partial_{x}\psi(\l,z)[\varphi_{k}],\varphi_{k}+\partial_{x}\psi(\l,z)[\varphi_{k}]],\varphi_{k}\right)_{L^{2}} \\
			& \quad +3 \left(Q_{k}\partial_{uu}\mf{F}(\l,z+\psi(\l,z))[\partial_{xx}\psi(\l,z)[\varphi_{k},\varphi_k],\varphi_{k}+\partial_{x}\psi(\l,z)[\varphi_{k}]],\varphi_{k}\right)_{L^{2}} \\
			& \quad + \left(Q_{k}\partial_{u}\mf{F}(\l,z+\psi(\l,z))[\partial_{xxx}\psi(\l,z)[\varphi_k,\varphi_{k},\varphi_{k}]],\varphi_{k}\right)_{L^{2}}.
		\end{align*}
		By evaluating this derivatives at $(\s_{k},0,0)$ and using Lemma \ref{L1}, \eqref{E15} and \eqref{E20}, we obtain
		\begin{align*}
			\partial_{x_1x_1x_1} h_1(\s_k,0,0) & = \left(Q_k \partial_{uuu}\mf{F}(\s_k,0)[\varphi_k,\varphi_k,\varphi_k],\varphi_k\right)_{L^2}\\
			&\;\;\;\; + \left(Q_k \mf{L}(\s_k)[\partial_{xxx}\psi(\s_k,0)[\varphi_k,\varphi_{k},\varphi_{k}]],\varphi_{k}\right)_{L^{2}} \\
			&=\left(Q_k(6a(t)\varphi_k^3),\varphi_k\right)_{L^2},
		\end{align*}
		where we have used that $Q_k\mf{L}(\s_{k})=0$. Finally, with the definition of $Q_k$ (see \eqref{eq:def_Q_k}), we deduce
		\begin{align*}
			\partial_{x_1x_1x_1} h_1(\s_k,0,0) & =\left(Q_k(6a(t)\varphi_k^3),\varphi_k\right)_{L^2} \\
			& = \left(6a(t)\varphi_k^3, \varphi_k\right)_{L^2} \left(\varphi_k,\varphi_k\right)_{L^2} + \left(6a(t)\varphi_k^3, \phi_k\right)_{L^2} \left(\phi_k,\varphi_k\right)_{L^2}=\left(6a(t)\varphi_k^3, \varphi_k\right)_{L^2}.
		\end{align*}
		This proves the first identity of \eqref{E27}, and the remaining ones can be obtained similarly.
														
		Passing to \eqref{E29}, we differentiate \eqref{eq:h1_x1x1} with respect to $x_2$:
		\begin{align*}
			& \partial_{x_1x_1x_2} h_{1}(\l,x_1,x_2) \\ & =\left(Q_{k}\partial_{uuu}\mf{F}(\l,z+\psi(\l,z))[\phi_{k}+\partial_{x}\psi(\l,z)[\phi_{k}],\varphi_{k}+\partial_{x}\psi(\l,z)[\varphi_{k}],\varphi_{k}+\partial_{x}\psi(\l,z)[\varphi_{k}]],\varphi_{k}\right)_{L^{2}} \\
			& \;\;\;\; + 2 \left(Q_{k}\partial_{uu}\mf{F}(\l,z+\psi(\l,z))[\partial_{xx}\psi(\l,z)[\phi_k,\varphi_{k}],\varphi_{k}+\partial_{x}\psi(\l,z)[\varphi_{k}]],\varphi_{k}\right)_{L^{2}} \\
			& \;\;\;\; + \left(Q_{k}\partial_{uu}\mf{F}(\l,z+\psi(\l,z))[\phi_{k}+\partial_{x}\psi(\l,z)[\phi_{k}],\partial_{xx}\psi(\l,z)[\varphi_{k},\varphi_{k}]],\varphi_{k}\right)_{L^{2}} \\
			& \;\;\;\; + \left(Q_{k}\partial_{u}\mf{F}(\l,z+\psi(\l,z))[\partial_{xxx}\psi(\l,z)[\phi_k,\varphi_{k},\varphi_{k}]],\varphi_{k}\right)_{L^{2}}.
		\end{align*}
		By reasoning as above, we get
		\begin{align*}
			\partial_{x_1x_1x_2} h_{1}(\s_k,0,0) = \left(Q_k(6a(t)\varphi_k^2 \phi_k),\varphi_k\right)_{L^2}=\left(6a(t)\varphi_k^2 \phi_k,\varphi_k\right)_{L^2},
		\end{align*}
		which proves the first identity of \eqref{E29}. The rest of the identities in \eqref{E29} can be obtained analogously. Finally, we differentiate \eqref{eq:h1_x1x1} with respect to $\l$:
		\begin{align*}
			& \partial_{\l x_1x_1}h_{1}(\l,x_1,x_2)\\
			& =  \left(Q_{k}\partial_{\l uu}\mf{F}(\l,z+\psi(\l,z))[\varphi_{k}+\partial_{x}\psi(\l,z)[\varphi_{k}],\varphi_{k}+\partial_{x}\psi(\l,z)[\varphi_{k}]],\varphi_{k}\right)_{L^{2}} \\
			& \;\;\;\; +\left(Q_{k}\partial_{uuu}\mf{F}(\l,z+\psi(\l,z))[\partial_{\l}\psi(\l,z),\varphi_{k}+\partial_{x}\psi(\l,z)[\varphi_{k}],\varphi_{k}+\partial_{x}\psi(\l,z)[\varphi_{k}]],\varphi_{k}\right)_{L^{2}} \\
			& \;\;\;\;  + 2 \left(Q_{k}\partial_{uu}\mf{F}(\l,z+\psi(\l,z))[\partial_{\l x}\psi(\l,z)[\varphi_{k}],\varphi_{k}+\partial_{x}\psi(\l,z)[\varphi_{k}]],\varphi_{k}\right)_{L^{2}} \\
			& \;\;\;\; + \left(Q_{k}\partial_{\l u}\mf{F}(\l,z+\psi(\l,z))[\partial_{xx}\psi(\l,z)[\varphi_{k},\varphi_{k}]],\varphi_{k}\right)_{L^{2}} \\
			& \;\;\;\; + \left(Q_{k}\partial_{u}\mf{F}(\l,z+\psi(\l,z))[\partial_{\l xx}\psi(\l,z)[\varphi_{k},\varphi_{k}]],\varphi_{k}\right)_{L^{2}}.
		\end{align*}
		By evaluating these derivatives at $(\s_{k},0,0)$, from Lemma \ref{L1}, \eqref{E15} and \eqref{E20}, we conclude
		\begin{align*}
			\partial_{\l x_1x_1}h_{1}(\s_k,0,0)=0.
		\end{align*}
		This proves one of the identities in \eqref{E30}, and the remaining ones can be obtained similarly, by using the third order derivatives of $\mf{F}$ that involve $\l$, which are given in Lemma \ref{L1}.
	\end{proof}
													
	\section*{Acknowledgments} This work has been supported by the Research Project PID2021-123343NB-I00 of the Ministry of Science and Innovation of Spain and by the Institute of Interdisciplinary Mathematics of Complutense University of Madrid. A.~T. has also been supported by the Ram\'{o}n y Cajal program RYC2022-038091-I, funded by MCIN/AEI/10.13039/501100011033 and by the FSE+.

\end{document}